\documentclass[final,10pt,reqno]{amsart}

\usepackage{comment} 

\usepackage[left=1.0in, right=1.0in, top=1.1in, bottom=1.1in]{geometry}

\usepackage[foot]{amsaddr} 
\usepackage{stmaryrd}
\usepackage{empheq}
\usepackage{pgfplots}
\pgfplotsset{compat=1.15}
\usepackage{mathrsfs}
\usetikzlibrary{arrows}
\usepackage[inner]{showlabels}

\usepackage{quiver}
\usepackage{hyperref}
\usepackage{verbatim}
\usepackage{amssymb}
\usepackage{mathrsfs}
\usepackage{amsthm}
\usepackage{amsmath}
\usepackage{amsfonts}
\usepackage{latexsym}
\usepackage{mathtools} 
\usepackage{enumitem} 
\usepackage[english]{babel} 
\usepackage{mathabx}
\usepackage{scalerel,stackengine}

\usepackage{float}

\usepackage[colorinlistoftodos,prependcaption,textsize=tiny]{todonotes}

\stackMath
\newcommand\reallywidehat[1]{%
	\savestack{\tmpbox}{\stretchto{%
			\scaleto{%
				\scalerel*[\widthof{\ensuremath{#1}}]{\kern-.6pt\bigwedge\kern-.6pt}%
				{\rule[-\textheight/2]{1ex}{\textheight}}
			}{\textheight}%
		}{0.5ex}}%
	\stackon[1pt]{#1}{\tmpbox}%
}

\emergencystretch=\maxdimen
\theoremstyle{plain}
\newtheorem{teor}{Theorem}
\newtheorem{obs}[teor]{Remark}
\newtheorem{prop}[teor]{Proposition}
\newtheorem{coro}[teor]{Corollary}
\newtheorem{lema}[teor]{Lemma}
\newtheorem{defi}[teor]{Definition}

\newenvironment{hproof}{%
	\proof}{\endproof}

\numberwithin{equation}{section}
\numberwithin{teor}{section}

\newcommand{\norma}[1]{{\left\vert\kern-0.25ex\left\vert\kern-0.25ex\left\vert #1 
		\right\vert\kern-0.25ex\right\vert\kern-0.25ex\right\vert}}
\def\R{\mathbb{R}}

\def\C{\mathcal{C}}
\def\N{\mathbb{N}_0}
\def\T{\mathbb{T}}
\def\grad{\text{grad}}
\def\S{\mathbb{S}}
\def\Z{\mathbb{Z}}

\def\p{\text{Op}}

\addto\extrasenglish{
	
}

\title[Solvability of the MHS equations in general domains]{Solvability of MHS equations with Grad-Rubin boundary conditions in general domains}

\author[D.S\'anchez-Sim\'on del Pino]{Daniel S\'anchez-Sim\'on del Pino}
\email{\href{mailto:sanchez@iam.uni-bonn.de}{sanchez@iam.uni-bonn.de}}

\author[J. J.L. Vel\'azquez]{Juan J.L. Vel\'azquez }
\address{University of Bonn, Institute for Applied Mathematics Endenicher Allee 60, D-53115, Bonn, Germany.}
\email{\href{mailto:velazquez@iam.uni-bonn.de}{velazquez@iam.uni-bonn.de}}
\begin{document}
	
	\begin{abstract}
		In this paper we study the solvability of the Magnetohydrostatic equations with Grad Rubin boundary conditions in general domains. Earlier results for this problem were obtained in \cite{Alo-Velaz-Sanchez-2023, Alo-Velaz-2022}, where  particularly simple geometries were considered. In this article we develop a theory that allows to solve these boundary value problems for a larger class of domains than in \cite{Alo-Velaz-Sanchez-2023, Alo-Velaz-2022}. We will give precise applications to more physically relevant situations, like the case of the space between two circumferences or spheres and domains close to them.
	\end{abstract}
	\maketitle

	\tableofcontents

	\section{Introduction}
	In order to describe high temperature plasmas it is necessary to consider equations from fluid mechanics and electromagnetism combined. The archetypal example are the equations of Magnetohydrodynamics, or MHD. This set of partial differential equations was firstly introduced by H. Alfvén and can be applied to a variety of settings, from laboratory to astrophysical applications. In several situations it is relevant to study the static regime, i.e. the regime in which the velocity of the fluid is zero. This corresponds to the so called Magnetohydrostatics, or MHS, which is modeled by the following system of PDEs:
	
	\begin{equation}\label{mhs:1}
		\left\{\begin{array}{ll}
			j\times B=\nabla p\\
			\nabla\times B=j\\
			\nabla\cdot B=0
		\end{array}
		\right..
	\end{equation}
	
	This set of equations is widely used in the field of Astrophysics, for example, where they can be used to model solar winds. They are very big ejections of plasma in which the point velocity can be taken to be negligible in comparison with the size of the plasma cloud. We refer the interested reader to \cite{Priest_2014} for more information on the astronomical background of the MHS equations. 
	
	The solvability of boundary value problems for the MHS equations was first considered by Grad and Rubin \cite{Grad-Rubin-1958}, where the authors suggested  different boundary conditions and gave some possible iterative schemes to give a solution for the resulting problems. There it was proposed to consider domains $\Omega$ in which its boundary $\partial\Omega$ was divided into two disjoint regions, $\partial\Omega_+$ and $\partial\Omega_-$, corresponding to the subsets of $\partial\Omega$ where we have outflux or influx of magnetic field, respectively. 
	
	In this article we will study the solvability of the system \eqref{mhs:1} for the following set of boundary conditions, proposed in \cite{Grad-Rubin-1958},
	
	\begin{equation}\label{mhs:boundary}
		\left\{\begin{array}{ll}
			j\times B=\nabla p & \text{in }\Omega\\
			j=\nabla\times B & \text{in }\Omega\\
			\nabla\cdot B=0 & \text{in }\Omega\\
			B\cdot n=f & \text{on }\partial\Omega\\
			B_\tau =g& \text{on }\partial\Omega_{-}
		\end{array}
		\right.,
	\end{equation}
	where $n$ denotes the outward normal to the boundary,  $\partial\Omega_-:=\{x\in\partial\Omega\,:\, B\cdot n<0\}$, $\partial\Omega_-:=\{x\in\partial\Omega\,:\, B\cdot n>0\}$ and $B_\tau=B-(B\cdot n)n$ denotes the tangential component of a vector field on $\partial\Omega$.
	
	One can notice that these equations are equivalent to the steady states of the incompressible Euler system (c.f. \cite[Chapter 1.4.1]{Vicol}) via the identifications $B\leftrightarrow v$, $j\leftrightarrow \omega$ and $p+\frac{|B|^2}{2}\leftrightarrow -p$. Therefore, it is natural to resort to the available techniques for such equations. One of the best known methods is the so-called ``Vorticity Transport Method'', introduced by H.D. Alber in \cite{Alber-1992} to construct solutions for the stationary Euler system that have non-vanishing vorticity. In terms of the MHS equations, he obtained a well posedness result for \eqref{mhs:1} for which as boundary data the pressure, $p$, and the current, $j$, were prescribed in $\partial\Omega_-$.
	
	The first rigorous results of this type for the system \eqref{mhs:boundary} are the ones in \cite{Alo-Velaz-2022}, where a well posedness result was proved for $\Omega=\T\times[0,L]$, with $\T$ the one dimensional torus $\R/\Z$. The techniques there were extended to the three dimensional case in \cite{Alo-Velaz-Sanchez-2023}, where $\Omega=\T^2\times[0,L]$. The geometry considered in those articles is slightly artificial because but it allows to perform the required computations more easily using Fourier series. In this paper we extend the method in \cite{Alo-Velaz-2022} in order to study more physically relevant geometries. 
	
	The steady states of the incompressible Euler system have been studied extensively with other different choices of boundary conditions. First of all, Grad and Rubin suggested a numerical scheme in \cite{Grad-Rubin-1958} to solve the equations. This was extended in \cite{Amari_Boulmezaoud_Mikic} for the case of force-free fields, where $j=\alpha(x)B$. These are also known as Beltrami fields. However, to the best of the authors' knowledge, there is no mathematically rigorous proof of the convergence of these methods.  There are also well posedness results for the \eqref{mhs:1} system, but not with our set of boundary conditions. A non-extensive list of articles, which are mostly written in the context of the stationary Euler's equation, is  \cite{Alo-Velaz-2021,Bineau-1972,Buffoni-Wahlen-2019,Seth-2020,Tang-Xin-2009}. 
	\subsection{Main ideas and techniques}
	Our results, as those in \cite{Alo-Velaz-Sanchez-2023,Alo-Velaz-2022}, are of perturbative nature. In other words, we will assume that we have an \textit{a priori} function $B_0$ satisfying certain properties, so we obtain solution $B$ as a perturbation of $B_0$. In order to shed light on the meaning of the main theorem in this paper, we introduce first the main techniques  in \cite{Alo-Velaz-Sanchez-2023,Alo-Velaz-2022} and how they can be adapted to the general setting that we consider here. 
	
	Such techniques constitute an adaptation of the Vorticity Transport Method of \cite{Alber-1992} to prove well posedness of the stationary Euler Equations in three dimensions. This method uses what is one of the main features of the MHS system, which is the interplay between the hyperbolic and the elliptic behaviour of the system. This dual character of the MHS system can be easily seen by taking the curl of the first equation in \eqref{mhs:1}, decomposing the problem into two different PDE systems:
	
	\begin{equation}\label{vorticitytransport}
		\left\{
		\begin{array}{ll}
			\nabla\times B=j\\
			\nabla\cdot B=0
		\end{array}\right. \qquad \text{and}\qquad (B\cdot \nabla)j=(j\cdot \nabla)B.
	\end{equation}
	
	The first problem corresponds to the div-curl system, which can be solved by fixing the normal component of $B$ on the boundary once we know $j$. The second problem in \eqref{vorticitytransport} is a first order transport PDE, which can be solved after fixing $j_0$ on $\partial\Omega_-$, as long as the vector field $B$ satisfies some basic properties. The procedure devised by Alber consists in a fixed point argument, where solutions of the MHS system are obtained as fixed points of a given operator $T$, that is constructed schematically as follows. Let $B$ be a divergence free vector field.
	
	\begin{itemize}
		\item Given $j_0$ defined on $\partial\Omega_-$, construct a current by solving the transport system 
		\begin{equation}\label{transportexp}
			\left\{\begin{array}{ll}
				(B\cdot \nabla)j=(j\cdot \nabla)B & \text{on }\Omega\\
				j=j_0 & \text{on }\partial\Omega_-
			\end{array}\right.,
		\end{equation}
		which leads to a current $j$ depending on the magnetic field $B$ and the boundary data for the current, $j_0$. 
		\item Solve the div-curl system 
		\begin{equation}
			\left\{
			\begin{array}{ll}
				\nabla\times W=j\\
				\nabla\cdot W=0\\
				W\cdot n=f
			\end{array}\right.,
		\end{equation}
		thus obtaining a new magnetic field $W=W[f,B,j_0]$. 
	\end{itemize}
	
	The corresponding solution $B$, of the problem \eqref{mhs:1} is then reformulated as a fixed point using the procedure above, so it will be the solution of the equation $B=W[f,B,j_0].$ The main problem in this arguments lies on the fact that, with our choice of boundary conditions, we do not have an immediate way of assigning values to $j_0$, so the first step of our procedure is not well defined. This can be circumvented by making use of the remaining boundary conditions we have for $W$, namely, $W_{\tau}=g$, where the subscript $\tau$ indicates that we are taking the tangential component on $\partial\Omega_-$. More precisely, we can fix $j_0$ by imposing that 
	
	\begin{equation}\label{integraleqn}
		W_\tau[f,B,j_0]=g,
	\end{equation}
	where we have made explicit the dependence of $W_\tau$ on $B$ and $j_0$. We claim that, for a fixed $B$, there is a function $\textsf{G}$, depending only on the boundary conditions, and a linear operator $A[B]$, so that \eqref{integraleqn} can be written as 
	
	\begin{equation} \label{operadorprincipal}
		A[B]j_0=\textsf{G}.
	\end{equation}
	In the sequel, this will be reinterpreted as an integral equation for $j_0$. Assuming that $A[B]$ is invertible, we can write the operator $T$ as 
	
	\begin{equation}\label{ecuacionpuntofijo} 
		T[B]=W[f,B,A^{-1}[B](\textsf{G})].
	\end{equation}
	
	We will be able to prove the invertibility of $A[B]$ for an open class of magnetic fields $B$ and a general class of smooth domains $\Omega$. The kind of theorem that we will prove reads as follows
	
	\begin{teor}\label{teoremacasos}
		Let $\alpha \in (0,1)$. Let $\Omega\in \R^d$  with $d\in\{2,3\}$ satisfying that $\partial\Omega$ has two connected components, $\partial\Omega_-$ and $\partial\Omega_+$. Consider $B_0\in C^{2,\tilde{\alpha}}$ with $1>\tilde{\alpha}>\alpha$ be a  divergence free magnetic field without vanishing points and such that $B_0\cdot n<0$ (resp. $>0$) on $\partial\Omega_-$ (resp. on $\partial\Omega_+$) satisfying the MHS system \eqref{mhs:1}. Assume further that $\text{curl}\,B_0\in C^{2,\alpha}.$ If the operator $A[B_0]$ defined in \eqref{operadorprincipal} has a kernel consisting only on the zero function, then there exists (small) constants $M=M(\alpha)>0$ and $\delta=\delta(M)>0$ such that, if $ f\in C^{2,\alpha}(\partial\Omega)$ and $g\in C^{2,\alpha}(\partial\Omega_-)$ satisfy 
		
		$$\int_{\partial\Omega}f =0,$$
		and 
		
		$$\|f-B_0\cdot n\|_{C^{2,\alpha}}+\|B_{0,\tau} -g\|_{C^{2,\alpha}}+\|\text{curl}\,B_0\|_{C^{2,\alpha}}\leq \delta,$$
		then there exists a solution $(B,p)\in C^{2,\alpha}(\Omega)\times C^{2,\alpha}(\Omega)$ of the problem 
		
		\begin{equation} \label{ec:maineqn}
			\left\{\begin{array}{ll}
				(\nabla\times B)\times B=\nabla p & \text{in }\Omega\\
				\nabla\cdot B=0 & \text{in }\Omega\\
				B\cdot n=f & \text{on }\partial\Omega\\
				B_\tau =g & \text{on }\partial\Omega_-
			\end{array} \right..	
		\end{equation}
		Moreover, this is the unique solution satisfying 
		
		$$\|B-B_0\|_{C^{2,\alpha}}\leq M.$$
	\end{teor} 

	\begin{obs}
		The statement of this theorem implies that for any solution $B_0\in C^{2,\tilde{\alpha}}$ of the boundary value problem \eqref{ec:maineqn} that satisfies certain properties (c.f. \ref{asunciones1}), there are solutions arbitrarily close to $B_0$ as long as the current $\omega=\nabla\times B_0$ is small enough. Since potential fields are $C^\infty$ and have zero curl, one can rephrase the theorem above by saying that potential vector fields are interior points, with respect to the $C^{2,\alpha}$ topology, of the set of solutions of the MHS system. 
	\end{obs}
	
	\begin{obs}
		The statement of this theorem is a bit unprecise because the operator $A$ has not been fully defined yet. A more precise formulation of the theorem will be given in Section \ref{outline}
	\end{obs}

	These results will be applied to the case of domains with rotational symmetry, where we can prove that $A[B_0]$ is invertible for some choices of $B_0$
	
	\begin{teor}\label{teoremacasos2}
		Let $L>1$ and $\alpha \in (0,1)$. Let $\Omega=\{r\in\R^d \,:\, 1<|r|<L^2\}$ with $d\in\{2,3\}$, and denote $\partial\Omega_-$ and $\partial\Omega_+$ the components of the boundary given by $\{|r|^2=1\}$ and $\{|r|^2=L^2\}$, respectively. Then, there exists a (small) constant $M=M(\alpha,L)>0$ such that, if $f\in C^{2,\alpha}(\partial\Omega)$ and $g\in C^{2,\alpha}(\partial\Omega_-)$ satisfy 
		
		$$\int_{\partial\Omega_-}f=0,$$
		and 
		
		$$\|f\|_{C^{2,\alpha}}+\|g\|_{C^{2,\alpha}}\leq M,$$
		then, there exists a solution $(B,p)\in C^{2,\alpha}(\Omega)\times C^{2,\alpha}(\Omega)$ of the problem 
		
		$$\left\{\begin{array}{ll}
			(\nabla\times B)\cdot B=\nabla p & \text{in }\Omega\\
			\nabla\cdot B=0 & \text{in }\Omega\\
			B\cdot n=f+1 & \text{on }\partial\Omega\\
			B_\tau =g & \text{on }\partial\Omega_-
		\end{array} \right.,$$
		where $B_0=\frac{r}{|r|^d}$, and $n$ is the outer normal vector to $\Omega$. Moreover, this is the unique solution satisfying 
		
		$$\|B-B_0\|_{C^{2,\alpha}}\leq M.$$
	\end{teor} 
	
	Note that this theorem is a generalization of the results in \cite{Alo-Velaz-Sanchez-2023,Alo-Velaz-2022} to the case where, instead of $\T\times[0,L]$ or $\T^2\times [0.L]$, the domain is the space between two concentric circles or spheres. The continuity of the operator $A[B]$ with respect to $B$ will allow to obtain similar results for domains that are close. (c.f. Section \ref{asunciones2d} and \ref{asunciones3d}).
	
	Even though the domains considered here have more physical relevance, one has to notice that the magnetic fields considered are still non-physical. Indeed, the fact that we have an influx of magnetic field on the inner boundary implies the existence of a magnetic monopole. On the other hand, these equations also make sense as steady states of the Euler System. In this situation, this set of boundary condition would correspond to a source of fluid which is being injected into $\Omega$. This is an interpretation much more common in Fluid Dynamics, and is the approach followed in \cite{Gieetal2,Gieetal} and in the fourth chapter of \cite{monographrusos}, where the time dependent version of the Euler's equation with outflow-inflow boundary condition (the name given there to the Grad Rubin boundary condition for the Euler System) is studied. The time dependent problem, however, is much different than the time independent, as noted out in \cite{monographrusos}, as in the former there is no need to resort to the solvability of any equation like that of \eqref{integraleqn}.
	
	\subsection{General assumptions}\label{asunciones1}
	In this subsection we will describe the general assumptions that we will make in the domains and on the boundary values in order to apply the technique described before. These correspond mainly to the requirements needed to define the operator $A[B]$.
	
	We will work with bounded domains $\Omega\subset \R^d$, $d=2,3$, given by $\Omega=\Omega_1\setminus \overline{\Omega}_0$, where both $\Omega_0$ and $\Omega_1$ are simply connected open domains in $\R^d$, with smooth, $C^\infty$, boundary, and $\overline{\Omega}_0\subset\Omega_1$. It is clear that $\partial\Omega= \partial\Omega_1\cup\partial\Omega_0$. $\partial\Omega_1$ and $\partial\Omega_0$ play the role of $\partial\Omega_+$ and $\partial\Omega_-$, respectively. Heuristically, this corresponds to domains with a hole or cavity. This is, for instance, the case of the space between two spheres or between two disks. We will assume that $0\in \Omega_0$, without loss of generality.
	\begin{figure}[h]
		\includegraphics[width=0.4\textwidth]{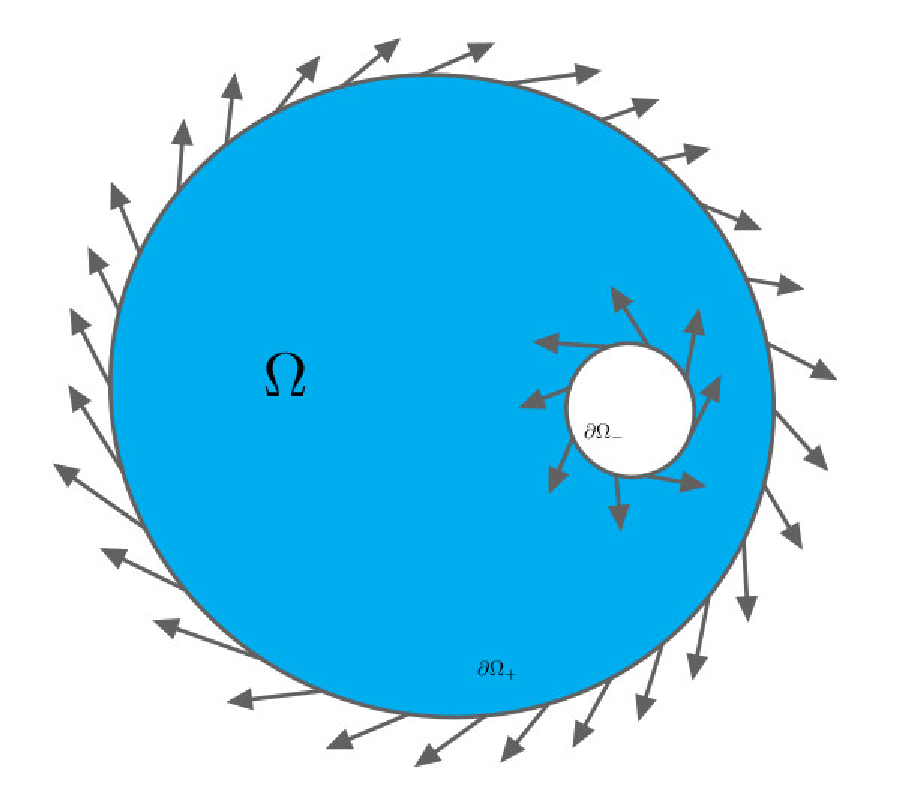}
		\caption{A graphic representation of the domains $\Omega$ we will work with. The arrows represent the magnetic field $B$ on the boundary}
	\end{figure}
	
	Regarding the magnetic fields $B_0$, we will work with fields satisfying that $B_0\cdot n<0$ on $\partial\Omega_-$, and $B_0\cdot n>0$ on $\partial\Omega_+$, where $n$ is the outer normal vector to $\partial\Omega$. Furthermore, we will assume that $B_0$ does not have vanishing points. This is needed in order to make sure that the characteristic lines of $B$ emanating from $\partial\Omega_-$ reach $\partial\Omega_+$, thus enabling us to solve the transport system \eqref{transportexp}. These conditions are stable under perturbations on the magnetic field. On the other hand, the set of magnetic fields  for which such conditions are satisfied is not empty , in the sense that one can construct a rather general class of domains $\Omega$ and fields $B_0$ for which the assumptions above hold. We will describe precisely these conditions in Sections \ref{asunciones2d} and \ref{asunciones3d}.
	\subsection{Plan of the paper}
	First, we will give a precise formulation of Theorem \ref{teoremacasos} in Section \ref{outline}, where we also will outline the proof. To that end, we will define the operator $A[B]$ in \eqref{operadorprincipal}. Section \ref{auxiliaryresults} is devoted to some technical results required to complete the proof of Theorem \ref{teoremacasos}. We begin by stating some classical results of pseudodifferential analysis and proving regularity estimates for kernels of pseudodifferential operators with limited regularity. This will allow us to write \eqref{integraleqn} as an invertibility problem for a class of pseudodifferential operators. In this section, we will also give a construction of the div-curl system in three dimensions that will allow us to write the leading order term of the operator $A[B]$ in the 3D case. 
	
	In Sections \ref{2D} and \ref{3D} we will prove the main result in the 2D and 3D cases, respectively. We will prove continuity estimates for $A[B]$ with respect to $B$. Furthermore, we will see that, in the case when $B$ is $C^\infty$ (the case of, for example, potential vector fields), the operators $A$ involved are Fredholm of index zero, i.e. $\text{dim}(\text{coker}A)-\text{dim}(\text{ker}A)=0$. This will allow us to formulate a criteria of solvability of the equation \eqref{integraleqn} in terms of the kernel of the operator $A[B]$.
	
	\section{Main result and outline of the proof}\label{outline}
	
	In this section we will give a rigorous description of the equation \eqref{integraleqn} in terms of an equation for an operator $A[B]$, thus making precise the heuristics we described before. We will consider the 2D and 3D cases separatedly, as the details of the construction of the operators $A[B]$ differs in each case. 
	
	\subsection{The 2D case}\label{integralequationforthecurrent2D}
	One of the steps required to obtain the equation for $j_0$ is solving the div-curl system. The precise construction of the solution in the 2D case is given in the following proposition.
	
	\begin{prop}\label{prop:divcurl2D}
		Let $\Omega\subset \R^2$ be a domain as described in Subsection \ref{asunciones1} and $j\in C^{1,\alpha}(\Omega)$ be a scalar function. Take $f\in C^{2,\alpha}(\partial\Omega)$ be such that the compatibility condition holds
		
		$$\int_{\partial\Omega} f dS=0,$$
		and assume further that $f>0$ (resp. $<0$) on $\partial \Omega_+$ (resp. $\partial\Omega_-$). Then, the following problem
		
		\begin{equation}\label{eq:divcurl2D}
			\left\{\begin{array}{ll}
				\nabla\times B=j & \text{in }\Omega\\
				\nabla\cdot B=0 & \text{in }\Omega\\
				B\cdot n=f & \text{on }\partial\Omega
			\end{array}\right.,
		\end{equation}
		has a solution $B\in C^{2,\alpha}$. This solution is not unique, as the general solution to this problem can be written as 
		
		$$B+J\cdot\nabla \phi,$$
		where $J\in \R$ and $\phi$ is the unique solution of the elliptic problem 
		
		$$\left\{\begin{array}{ll}
			\Delta \phi=0 & \text{in }\Omega\\
			\phi=0 & \text{on }\partial\Omega_+\\
			\phi=1 & \text{on }\partial\Omega_-
		\end{array}\right..$$
	\end{prop}
	\begin{proof}
		The main idea of the proof consists on writing $B$ as the perpendicular gradient of a scalar function. In order to be able to do so, we need to modify the problem slightly to make sure that the required compatibility conditions hold. To that end, we define the magnetic field $B_{mono}$ by means of
		
		$$B_{mono}:=\nabla\left(\ln(|x^2+y^2|)\right),$$
		which will be used to cancel the flux of $B$ along each connected component of the boundary. It is readily seen that the integral of $B_{mono}$ along $\partial\Omega_-$ (resp. $\partial\Omega_+$) does not vanish. Therefore, there exists a $\lambda \neq0$ satisfying that 
		
		$$\int_{\partial\Omega_-} B\cdot n =\lambda \int_{\partial\Omega_-} B_{mono}\cdot n.$$
		
		Now, denote $f_-:=(B-\lambda B_{mono})\cdot n|_{\partial\Omega_-}$ and $f_+:=(B-\lambda B_{mono})\cdot n|_{\partial\Omega_+}$. Due to the divergence free condition of both $B$ and $B_{mono}$ in $\Omega$, we find that 
		
		\begin{equation}\label{eq:vanishingint} 
			\int_{\partial\Omega_\pm}f_{\pm}dx=0.
		\end{equation}
		
		Then, we can construct a solution $B$ to the div-curl system \eqref{eq:divcurl2D} as 
		$$\lambda B_{mono}+\nabla^\perp u,$$
		where the perpendicular gradient is defined as the differential operator $\nabla^\perp:=\left(-\frac{\partial}{\partial y},\frac{\partial}{\partial x}\right)$, and $u$ is the solution of the following elliptic problem 
		
		$$\left\{\begin{array}{ll}
			\Delta u=j & \text{in }\Omega\\
			u=h^- & \text{on }\partial\Omega_-\\
			u=h^+ & \text{on }\partial\Omega_+\\
		\end{array}\right..$$
		
		The functions $h^{\pm}:\partial\Omega_\pm\longrightarrow \R$ are defined, in terms of a parametrization $\gamma^{\pm}:[0,2\pi]\longrightarrow \partial\Omega_{\pm}$ of the two curves $\partial\Omega_\pm$, as 
		
		\begin{equation}\label{haches}
			h^{\pm}(\gamma^{\pm}(t)):=\int_0^{t} f_{\pm}(\gamma(s))\|\gamma'(s)\|ds.
		\end{equation}
		
		Note that, by construction, both $f_{\pm}$ have integral zero along any of both surfaces. Therefore, the functions $h^{\pm}$ are well defined on the curves $\partial\Omega_\pm$.

		Uniqueness of the div-curl problem does not hold in general, as it depends crucially on the topology of the domain we consider. The solution of this problem is unique up to the addition of a potential vector field tangent to the boundary, i.e. a vector field $\tilde {B}$ satisfying 
		
		$$\left\{\begin{array}{ll}
			\nabla\times \tilde{B}=0 & \text{in }\Omega\\
			\nabla\cdot \tilde{B}=0 & \text{in }\Omega\\
			\tilde{B}\cdot n=0 & \text{on }\partial\Omega
		\end{array}\right..$$
		
		This set  of $\tilde{B}$ forms a vector space whose dimension depends on the topology of the domain. In our case, due to our choice of $\Omega$ (c.f. Subsection \ref{asunciones1}), the dimension of this linear space is $1$ and is generated by $\nabla \phi$, where $\phi$ is the unique solution of the Dirichlet problem 
		
		\begin{equation}\label{deRham} \left\{\begin{array}{ll}
				\Delta \phi=0 & \text{in }\Omega\\
				\phi=0 & \text{on }\partial\Omega_-\\
				\phi=1 & \text{on }\partial\Omega_+\\
			\end{array}\right..
		\end{equation}	
	\end{proof}
	
	We now define an operator, denoted by $A[B]$, so that the solvability of the equation \eqref{integraleqn} is equivalent to the invertibility of $A$.

	\begin{defi}\label{defia}
		Let $B\in C^{2,\alpha}(\overline{\Omega};\R^3)$ be a divergence free vector field that does not vanish anywhere in $\Omega$. We define the map 
		
		$$A[B]:C^{1,\alpha}(\partial\Omega_-)\longrightarrow C^{2,\alpha}(\partial\Omega_-)$$
		as $A[B]j_0=n\cdot \nabla\varphi|_{\partial\Omega_-}$, where $\varphi$ is the unique solution of
		
		\begin{equation}\label{scheme}
			\left\{\begin{array}{ll}
				\Delta \varphi=j & \text{in }\Omega\\
				\varphi=0 & \text{on }\partial\Omega_-\\
				\varphi=0 & \text{on }\partial\Omega_+\\
			\end{array}\right.,
		\end{equation}
		with $j$ the solution of the following transport equation,
		
		$$\left\{\begin{array}{ll}
			(B\cdot \nabla)j=0 & \text{on }\Omega\\
			j=j_0 &\text{on }\partial\Omega_-
		\end{array}\right..$$ 
	\end{defi}

	Now, we are able to formulate the equation \eqref{integraleqn} for $j_0$ in terms of the operator $A[B]$ as follows
	
	\begin{coro}\label{eqnintegral2xD}
		Let $\Omega\subset \R^2$ and $B\in C^{2,\alpha}(\overline{\Omega};\R^3)$ satisfy the hypothesis described in Subsection \ref{asunciones1}. Let $j_0\in C^{1,\alpha}(\partial\Omega_-)$ and define $W\in C^{2,\alpha}$ as a solution of 
		
		$$\left\{\begin{array}{ll}
			\nabla\times W=j & \text{in }\Omega\\
			\nabla\cdot W=0 & \text{in }\Omega\\
			W\cdot n=f & \text{on }\partial\Omega_-
		\end{array}\right. \quad \text{with }\quad \left\{\begin{array}{ll}
			B\cdot \nabla j & \text{in }\Omega\\
			j=j_0 & \text{on }\partial\Omega_-
		\end{array}\right.. $$
		
		Then, for $g\in C^{2,\alpha}(\partial\Omega_-)$, $W_\tau=g$ holds if and only if 
		
		\begin{equation}\label{integraleqn2DD} 
			\lambda B_{mono,\tau}+n\cdot \nabla v|_{\partial\Omega_-}+Jn\cdot \nabla\phi|_{\partial\Omega_-} +A[B]j_0=g,
		\end{equation}
		where $\phi$ is the solution of \eqref{deRham},  $\varphi$ solves the following elliptic problem 
		
		$$\left\{\begin{array}{ll}
			\Delta v =0 & \text{in }\Omega\\
			v=h^{\pm} & \text{on }\partial\Omega_{\pm}\\
		\end{array}\right.,$$
		with $h^\pm$ defined in \eqref{haches} and where $J$ is an arbitrary constant.
	\end{coro}
	
	As we can see, the precise formula for \eqref{integraleqn} is in \eqref{integraleqn2DD}, and its solvability is determined by the invertibility of the operator $A[B]$, up to the addition of a multiple of $(\nabla\phi)_{\tau}$. This extra degree of freedom in \eqref{integraleqn2DD} is determined by noting that it can be used to obtain a uni-valued pressure: 
	\begin{prop}\label{puneterapresion}
		Assume that $A[B]$ is invertible, and let $j_0$ be a solution of \eqref{integraleqn2DD} for a given $J$. Assume that  
		
		\begin{equation}\label{condicion}
			A[B]^{-1}(n\cdot\nabla \phi|_{\partial\Omega_-}),
		\end{equation}
		never vanishes in $\partial\Omega_-$. Then, we   define 
		\begin{equation} \label{Jota} 
			J=\left(\int_{\partial\Omega_-} f\,A[B]^{-1}(n\cdot\nabla \phi|_{\partial\Omega_-})\right)^{-1}\int_{\partial\Omega_-}f\left(A[B]^{-1}g-\lambda A[B]^{-1}B_{mono}-A^{-1}(n\cdot\nabla v|_{\partial\Omega_-})\right).
		\end{equation}  
		
		Now, define the pressure of the resulting $j$ to be 
		
		\begin{equation}\label{presion} p:=\int_{\textit{\textbf{x}}_0}^{\textit{\textbf{x}}}j\times W d\vec{\ell},
		\end{equation}
		with $\textit{\textbf{x}}_0\in \Omega$ arbitrary, but fixed, $W$ as in Corolary \ref{eqnintegral2xD}, and where the integral is taken along any path in $\Omega$ connecting 
		$\textit{\textbf{x}}_0$ and \textit{\textbf{x}}.
		
		Then, if $W=B$, $p$ is a well defined uni-valued function if and only if $J$ is taken as in \eqref{Jota}.

	\end{prop} 
	\begin{proof}
		In order to prove that \eqref{presion} is well defined, we need to check that the integral that defines it does not depend on the path we take to join $\textbf{\textit{x}}_0$ and $\textbf{\textit{x}}$. This is achieved if $j\times W$ is curl-free and its integral along $\partial\Omega_-$ equals zero. The first requirement holds by construction. Indeed, employing elementary vector identities,
		
		\begin{align*}
			\nabla\times (j\times W)&=j(\nabla \cdot W)-W(\nabla\cdot j)+(j\cdot \nabla)W-(W\cdot \nabla)j.
		\end{align*}
		
		The first term on the right hand side vanishes because $\nabla\cdot W=0$. The only non-vanishing component of $j$ is the third one and it is independent of $z$, so $\nabla\cdot j=0$ as well. Finally, since $W=j$, we infer that the curl of $j\times W$ is 
		
		$$(W\cdot \nabla )j=0.$$
		
		Since $W=B$, this condition holds for construction. On the other hand, the integral of $j\times W$ along $\partial\Omega_-$ equals 
		
		$$\int_{\partial\Omega_-}j\times Wd\vec{\ell}=\int_{\partial\Omega_-}j_0\cdot f\,d\vec{\ell}.$$
		
		Thanks to Corollary \ref{eqnintegral2xD}, and the fact that $A$ is invertible, we can write such equation as 
		
		$$\int_{\partial\Omega_-}f\left(A[B]^{-1}g-\lambda A[B]^{-1}B_{mono}-A[B]^{-1}(n\cdot\nabla v|_{\partial\Omega_-})+JA[B]^{-1}(n\cdot\nabla \phi|_{\partial\Omega_-})\right)=0.$$
		
		Due to condition \eqref{condicion}, formula \eqref{Jota} follows.
	\end{proof}
	
	With this notations, we can then give a precise statement of the theorem for the case of 2D: 
	
	\begin{teor}\label{teor:mainteor2D}
		Let $\alpha \in (0,1)$. Let $\Omega\in \R^2$ as described in Section \ref{asunciones1}. Consider $B_0\in C^{2,\tilde{\alpha}}$ with $1>\tilde{\alpha}>\alpha>0$ a smooth divergence free magnetic field without vanishing points and such that $B_0\cdot n<0$ (resp. $>0$) on $\partial\Omega_-$ (resp. on $\partial\Omega_+$) satisfying the MHS system. Assume further that $\text{curl}\, B_0\in C^{2,\alpha}$ and that the operator $A[B]$ in Definition \ref{defia} with $B=B_0$ has a kernel consisting only on the zero function. Then, if \eqref{condicion} does not vanish, there exists (small) constants $M>0$, $\delta=\delta(M)>0$  such that, if $ f\in C^{2,\alpha}(\partial\Omega)$ and $g\in C^{2,\alpha}(\partial\Omega_-)$ satisfy the following integrability condition
		
		$$\int_{\partial\Omega}f = 0,$$
		and that the following smallness condition holds
		$$\|f-B_0\cdot n\|_{C^{2,\alpha}}+\|B_{0,\tau} -g\|_{C^{2,\alpha}}+\|\text{curl}\,B_0\|_{C^{2,\alpha}}\leq \delta,$$
		Then there exists a solution $(B,p)\in C^{2,\alpha}(\Omega)\times C^{2,\alpha}(\Omega)$ of the problem \eqref{ec:maineqn}.
		Moreover, this is the unique solution satisfying 
		
		$$\|B-B_0\|_{C^{2,\alpha}}\leq M.$$
	\end{teor} 
	\begin{hproof} The idea behind the proof is similar to that of \cite{Alo-Velaz-2022}, although new technicalities arise due to the generality of the domain. 
		
		As indicated in the introduction, the solution $B$ will be obtained as a fixed point of a given operator $T$. Following the heuristics that led to \eqref{integraleqn}, $T$ can be written as $T[B]:=B_s[f,J,\mathscr{T}[B](j_0)]$, where $W=B_s[f,J,\mathscr{T}[B](j_0)]$ is the solution of the div-curl system 
		
		\begin{equation}\label{eq:divcurlprueba} \left\{\begin{array}{ll}
				\nabla\times W=j &\text{in }\Omega\\
				\nabla\cdot W=0 & \text{in }\Omega\\
				W\cdot n=f & \text{on }\partial\Omega
			\end{array}\right.,
		\end{equation}
		obtained in Proposition \ref{prop:divcurl2D}, with $J$ as in Proposition \ref{puneterapresion}, and $j:=\mathscr{T}[B]j_0$ is given by the solution of 
		
		\begin{equation}\label{eq:transportprueba}
			\left\{\begin{array}{ll}
				B\cdot\nabla j=0 &\text{in }\Omega\\
				j=j_0 & \text{on }\partial\Omega_-
			\end{array}\right.,
		\end{equation}
		with $j_0$ solving \eqref{integraleqn2DD}.
		
		We do not know \textit{a priori}, for an arbitrary magnetic field $B$, whether $A[B]$ is invertible nor whether \eqref{condicion} vanishes. Therefore, it is not clear for now if $T$ is a well defined operator. Now, assume that $A[B_0]$ is invertible and that \eqref{condicion} does not vanish anywhere. Now, assume that a continuity statement for $A$ with respect to $B$ holds. Then, by means of a Neumann series, we can prove that $A[B]$ remains invertible and \eqref{condicion} does not vanish at any point if $B$ is close to $B_0$. More precisely: 
		
		\textbf{Claim 1} Given $B_1$, $B_2$ divergence free vector fields that do not vanish anywhere, there exists a constant $C$ satisfying 
		
		$$\|A[B_1]-A[B_2]\|_{\mathcal{L}(C^{1,\alpha},C^{2,\alpha})}\leq C\frac{\max(\|B_1\|_{C^2},\|B_2\|_{C^{2}})}{\min_{i=1,2}\left(\min_{\omega\in \partial\Omega_-}|B_i(\omega)\cdot n(\omega)|\right)}(\|B_1-B_2\|_{C^{2,\alpha}}).$$
		
		This will be proved in Theorem \ref{mainteor2D2D}. In such a case, by means of a Neumann series argument, for $M$ small enough, we find that if $A[B_0]$ is invertible, then $A[B_0+b]$ is invertible as well. Furthermore, using again a Neumann series argument, we find that 
		
		$$\|A[B_0+b]^{-1}-A[B_0]^{-1}\|_{\mathcal{L}(C^{1,\alpha},C^{2,\alpha})}\leq C\|A[B]^{-1}\|_{\mathcal{L}(C^{2,\alpha})}\|b\|_{C^{2,\alpha}}.$$
		
		Therefore, if $A[B_0]$ is invertible, and \eqref{condicion} does not vanish in $\partial\Omega_-$, for $M$ small enough, $T$ is a well defined operator 
		
		$$T:X\longrightarrow C^{2,\alpha}(\overline{\Omega},\R^2),$$
		where 
		
		$$X=\{B_0+b\, :\, \text{div}\,b=0\,,\,\|b\|_{C^{2,\alpha}}<M \}\subseteq C^{1,\alpha}.$$
		Note that $X$ is a complete metric space when endowed by the metric in $C^{1,\alpha}$. The fixed points of $T$ will be the solutions of the MHS equation with our choice of boundary value, since the election of $J$  leads to a uni-valued pressure due to Proposition \ref{puneterapresion}.
		
		We we will use Banach's fixed point. Therefore we need to prove that $X$ remains invariant under $T$ and that  $T$ is contractive.  To that end notice that if $B_0$ is itself a solution of the MHS system, then  it is a fixed point of $T$, but with different boundary conditions. More precisely, 
		
		$$B_0=B_s[B\cdot n, \tilde{J}, \mathscr{T}[B_0](\text{curl}\,B_0|_{\partial\Omega_-})].$$
		i.e., since $B_0$ is a solution of the MHS system, it satisfies the div-curl problem 
		
		$$\left\{\begin{array}{ll}
			\nabla\times B_0=\omega &\text{in }\Omega\\
			\nabla\cdot  B_0=0 &\text{in }\Omega
		\end{array}.\right.,$$
		with $\omega:=\text{curl}\, B_0$ satisfying 
		
		$$(B_0\cdot \nabla )\omega=0\Rightarrow \omega=\mathscr{T}[B_0](\omega_0),$$
		and with $\omega_0:=\text{curl}\,B_0|_{\partial\Omega_-}$ satisfying \eqref{integraleqn}. Finally, since $B_0$ has a well defined pressure, and since \eqref{condicion} does not vanish due to the hypothesis, we conclude that $\tilde{J}$ is given by \eqref{Jota} with $g=B_{0,\tau}$ and $f=B_0\cdot n$.  Due to linearity of the div-curl system, one can write 
		
		\begin{equation}\label{eq:camponoperturbado}
			B_0=B_s[B_0\cdot n,\tilde{J},\mathscr{T}[B_0](\omega_0)]=B_s[B_0\cdot n,0,0]+B_s[0,\tilde{J},0]+B_s[0,0,\mathscr{T}[B_0](\omega_0)],
		\end{equation}
		and 
		\begin{equation}\label{eq:camponoperturbado2}
			T[B_0+b]=B_s[f,{J},\mathscr{T}[B_0+b](j_0)]=B_s[f,0,0]+B_s[0,{J},0]+B_s[0,0,\mathscr{T}[B_0](j_0)],
		\end{equation}
		
		Therefore, 
		
		\begin{equation}\label{eq:perturbacion}
			T[B_0+b]-B_0=B_s[B_0\cdot n-f,0,0]+B_s[0,J-\tilde{J},0]+B_s[0,0,\mathscr{T}[B_0](j_0)-\mathscr{T}[B_0](\omega_0)]=(I)+(II)+(III).
		\end{equation}
		
		By assumption, we can estimate $(I)$ by 
		
		$$\|(I)\|\leq C\|f-B\cdot n|_{\partial\Omega}\|_{C^{2,\alpha}}\leq C\delta.$$
		
		For the other two terms on the right hand side of \eqref{eq:perturbacion}, we have to use the estimates on the operator $A$. Recall from \eqref{integraleqn2DD} that $J$ and $\tilde{J}$  sastisfy
		
		\begin{equation}\label{eq:Jtilde}
			B_{0,\tau}=n\cdot \nabla\tilde{v}|_{\partial\Omega_-}+\tilde{\lambda}B_{mono,\tau}+\tilde{J}(n\cdot \nabla\phi|_{\partial\Omega_-})+A[B_0]\omega_0
		\end{equation}
		and 
		
		\begin{equation}\label{eq:Juota}
			B_{0,\tau}+g=n\cdot \nabla {v}|_{\partial\Omega_-}+{\lambda}B_{mono,\tau}+{J}(n\cdot \nabla\phi|_{\partial\Omega_-})+A[B_0+b]j_0
		\end{equation}
		
		If we now apply the operator $A[B+b]^{-1}$ to \eqref{eq:Jtilde} and \eqref{eq:Juota} and substract the resulting identities, we obtain the formula 
		
		\begin{equation} \label{ecuaciondiferencial} 
			j_0+(J-\tilde{J})A[B_0+b]^{-1}(n\cdot \nabla\phi|_{\partial\Omega_-})=\textsf{G}+A[B_0+b]^{-1}A[B_0]\omega_0,
		\end{equation}
		where $\textsf{G}$ is given by 
		
		$$A[B+b]^{-1}\left(g+n\cdot \nabla(\tilde{v}-v)|_{\partial\Omega_-}+(\tilde{\lambda}-\lambda)B_{mono,\tau}\right)\Rightarrow \|\textsf{G}\|_{C^{1,\alpha}}\leq C(\|g\|_{C^{2,\alpha}}+\|f\|_{C^{2,\alpha}}).$$
		
		Now, due to equation \eqref{Jota}, we can multiply \eqref{ecuaciondiferencial} by $f $ and integrate in $\partial\Omega_-$. This means that, by construction of $J$, the term $j_0$ vanishes, leading to the formula 
		
		$$J-\tilde{J}=\left(\int_{\partial\Omega_-}fA[B_0+b]^{-1}((n\cdot \nabla\phi|_{\partial\Omega_-})\right)^{-1}\left(\int_{\partial\Omega_-}f\left(\textsf{G}+A[B_0+b]^{-1}A[B_0]\omega_0\right)\right).$$
		
		Due to the smallness condition on $f,g$ and $\omega_0$, we find that 
		
		$$|J-\tilde{J}|\leq C\delta\Rightarrow \|(II)\|_{C^{2,\alpha}}\leq C\delta. $$
		
		Finally, in order to obtain a suitable estimate for the term $(III)$ in \eqref{eq:perturbacion}, once we know that $|J-\tilde{J}|$ is of order $\delta$, we can take \eqref{ecuaciondiferencial}, and notice that it can be written as 
		
		$$j_0-\omega_0=\textsf{G}-(J-\tilde{J})A[B_0+b]^{-1}(n\cdot \nabla\phi|_{\partial\Omega_-})+A[B_0+b]^{-1}A[B_0]\omega_0-\omega_0.$$
		
		Again, due to the smallness condition on of $\omega_0$, we can estimate $\|j_0-\omega_0\|_{C^{1,\alpha}}$ by $C\delta$. Therefore, since $\omega_0$ lies in $C^{2,\alpha}$, we can write  $\|(III)\|_{C^{2,\alpha}}\leq C\|\mathscr{T}[B+b](j_0)-\mathscr{T}[b]\omega_0\|_{C^{1,\alpha}}\leq C\|j_0-\omega_0\|_{C^{1,\alpha}}+C\|b\|_{C^{2,\alpha}}\|\omega\|_{C^{2,\alpha}}\leq  C\delta$. As a result, taking all of the terms together, we find that 
		
		$$\|T[B_0+b]-B_0\|_{C^{2,\alpha}}\leq C\delta.$$
		
		Taking  $M\leq 1$ and $\delta=\delta(M)$   small enough, we conclude that $T\subset X$. Analogously, employing that solutions of the first order transport equations \eqref{transportexp} satisfy suitable estimates in $C^{0,\alpha}$ for fields in $C^{2,\alpha}$ (c.f. \cite[Section 7.1]{Alo-Velaz-2022} and \cite[Proposition 2.15]{Alo-Velaz-Sanchez-2023}), we derive that $T$ is a contraction with respect to the $C^{1,\alpha}$ metric, as long as $M$ is small enough.
		
		We notice that a crucial step is proving that the operator $A[B]$ is invertible. Due to the properties of $A[B]$ when $B_0$ it is smooth, we can check that, actually, we do not need to prove both surjectivity and injectivity, but only one of both.
		
		\textbf{Claim 2} If $B$ is $C^\infty$, the operator $A[B]$ is a Fredholm operator of index $0$. Therefore, it is bijective if and only if it is injective (or surjective), as one can infer from Definition \ref{defiFredholm}. Due to the fact that $C^\infty$ is dense in $(C^{2,\tilde{\alpha}},\|\cdot\|_{2,\alpha})$, and the fact that the set of Fredholm operators are open in $\mathcal{L}(C^{2,\alpha})$, we can extend this property to any $B\in C^{2,\tilde{\alpha}}$.  This is the content of Proposition \ref{Fredholm2D}.

	\end{hproof}
	\subsection{The 3D case}\label{integralequcurr3d}
	
	The formulation of the equation \eqref{integraleqn} in three dimensions is more cumbersome, due to the fact that $j$ now is a vector, instead of just a scalar. As in Section \ref{integralequationforthecurrent2D} we will give a precise description of the equation in terms of an equation for a given linear operator $A$.
	
	We begin by giving some definitions from differential geometry that are useful to give a precise formulation of equation $\eqref{integraleqn}$. All of these may be found in the classical reference \cite{Lee}.
	
	\begin{defi}Let $S\subset \R^3$ be a smooth surface and consider the metric $g$ that it inherits from the ambient metric.
		\begin{itemize}
			\item For a given $f\in C^1(S;\R)$, we define its gradient $\text{grad}_S\,f$ as the unique vector field that satisfies 
			
			$$df(X)=\langle \text{grad}_S\,f,\,X\rangle_g.$$
			\item For a given $u$ vector field of class $C^1$, we define its divergence in $S$, $\text{div}^\|_S u$, as the unique continuous function satisfying 
			
			$$\int_S \text{div}^\|_{S}\,u(x)\cdot v(x)dS(x)=\int_S \langle u(x),\, \text{grad}\, v(x)\rangle_g dS(x)$$
			for any $v\in C^1_c(S)$.
		\end{itemize}
	\end{defi}
	
	It is relevant to know the expression in coordinates of these operators
	\begin{prop}
		Let $S$, $f$ and $u$ as before. Consider a parametrization of the surface, $\Gamma:U\subset \R^2\longrightarrow S$ with coordinates $(x_1,x_2)$. Then, in this set of coordinates 
		
		$$(\text{grad}_S\,f)^i=g^{ij}\frac{\partial (f\circ \Gamma)}{\partial x_j}\qquad \text{and}\qquad \text{div}^\|_S u=\frac{1}{\sqrt{|g|}}\frac{\partial}{\partial x_k}\left(\sqrt{|g|}u^k\right).$$
	\end{prop}
	
	We also recall now the definition of the Lie derivative in a surface. 
	
	\begin{defi}\label{Lie}
		Let $S\subset \R^3$ be a smooth compact surface. Let $X$ be a $C^1$ vector field. Now, denote by $\varphi_t$ its flow map, i.e. the one parameter family of diffeomorphisms given by 
		
		$$\frac{d}{dt}\varphi_t (x)=X(\varphi_t(x))\quad \varphi_0(x)=x.$$
		
		Then, we define the Lie derivative of the function $f$ as 
		
		$$\mathcal{L}_X f(x):=\lim_{t\rightarrow 0}\frac{f(\varphi_{t}(x))-f(x)}{t}.$$
	\end{defi}
	
	Again, it is relevant to know how to write this operator in a set of coordinates.
	
	\begin{prop}
		Let $S$, $f$ and $X$ as before. Then, consider a parametrization of the surface, $\Gamma:U\subset \R^2\longrightarrow S$ with coordinates $(x_1,x_2)$. Then, in this set of coordinates 
		
		$$\mathcal{L}_X f(x)=X^k\frac{\partial (f\circ \Gamma)}{\partial x^k}.$$
	\end{prop}

	Analogously to the Helmholtz decomposition in $\R^3$, one can decompose a vector field $u$ in a surface $S$ into a \textit{divergence free} and a \textit{gradient} part. 
	
	\begin{prop}\label{hodgevector}
		Let $u$ be a $C^1$ vector field tangent to an orientable surface $S\subset \R^3$. Then, if $S$ is simply connected, there exist two functions $\psi$ and $\phi$, unique up to constants, so that 
		
		$$u=n\times \text{grad}_{S}\psi+\text{grad}_{S}\varphi,$$
		where $n$ is the outward unit normal vector to $\partial\Omega_-$. We will use the notation 
		
		$$n\times\text{grad}_{S}\psi=G.$$
	\end{prop}
	\begin{proof}
		This is a consequence of the Hodge theorem, whose proof can be found in \cite{Schwarz}. Since $S$ is simply connected, it does not admit any non-constant harmonic field, so any $1$ form can be decomposed as the differential of a function plus the codifferential of a $2$-form. If we now employ the identification of $1$ forms and vector fields, we obtain the result. 
	\end{proof}
	
	A very useful construction when we are dealing with oriented surfaces that are boundaries of a given open set, is that of the decomposition into a normal and a tangential part. 
	
	\begin{defi}
		Consider $\Omega\subset \R^3$ an open set with $\partial\Omega$ being an orientable surface. As such, there exists a well defined outer normal vector $n$. Then, we can define the normal and the tangential part of a vector $j:\overline{\Omega}\longrightarrow \R^3$ on $\partial\Omega$ as the function $j^\rho:=j\cdot n$ and the vector field $j^\|:=j-n(j\cdot n)$.  
	\end{defi}
	
	One very convenient property of domains $\Omega$ whose boundary $\partial\Omega$ is a compact orientable surface is the possibility to extend this tangential and normal components in a meaningful way.  First of all, consider the boundary $\partial\Omega$. We have chosen it to be a closed smooth manifold, so there exists a positive $\varepsilon$ satisfying that the signed distance to the boundary $\partial\Omega$, $\text{dist}$, is a smooth function in $\mathcal{U}:=\text{dist}^{-1}(-\varepsilon,\varepsilon)$. This is the standard construction of the tubular neighborhood (c.f. \cite[Chapter 6]{Lee}). We now choose the convention so that the sign of the distance is positive in the interior of $\Omega$. Using these classical theorems one can prove that $\mathcal{U}$ is actually diffeomorphic to $\partial\Omega\times(-\varepsilon,\varepsilon).$ We can even give a local parametrization of $\mathcal{U}$ in terms of local parametrizations of $\partial\Omega_-$. Let $\Gamma:U(\subseteq \R^2)\longrightarrow\partial\Omega_-$ be a local parametrization of the surface. Then, we can give a parametrization of $\mathcal{U}$ just by taking 
	
	\begin{equation} \label{feassparametrizaciones} 
		U\times(-\varepsilon,\varepsilon)\ni (x',\rho)\mapsto \Gamma(x')+\rho\nu(x'),
	\end{equation} 
	where $\nu(x')$ is the unit inner normal vector in $\Omega$, i.e., $\nu(x')=-n(x')$. The basis of the tangent space to $\mathcal{U}$ is then given, in terms of this parametrization, by 
	
	\begin{equation}\label{conazo}
		\frac{\partial}{\partial x_i}:=\partial_i \Gamma(x')+\rho\partial_i \nu(x')\qquad i\in\{1,2\},\quad  \text{ and }\quad \frac{\partial}{\partial\rho}\equiv \partial_\rho=\nu(x').
	\end{equation}

	This choice allows us to write some divergence operators in a convenient form. First of all, notice that the resulting coordinate tangent vector $\partial_\rho$ equals $\nu(x')$. This is therefore orthogonal to the other coordinate tangent vectors $\partial/\partial x_i$ in \eqref{conazo}. Furthermore, its euclidean norm equals one. As a result, the matrix for the metric in this coordinates is rather simple, being 
	
	$$g_{ij}(x',\rho)=\left(\begin{array}{cc}
		h(x',\rho) & 0\\
		0 & 1
	\end{array}\right),$$
	where $h(x',\rho)$ is a $\rho$-dependent metric that corresponds to the metrics of the surfaces $\rho$ constant. We can then write the divergence of a given vector field $u$ in $\mathcal{U}$ by means of the divergence of the surfaces defined by $\rho=\text{const.}$
	
	\begin{lema}\label{divergencia}
		Let $u$ be a vector field defined on $\mathcal{U}$. Then, its divergence $\text{div } u$ equals, in the system of coordinates $(x',\rho)\in U\times(-\varepsilon,\varepsilon),$
		
		$$\text{div }\,u= \text{div}_{\rho}\,^\| u^\| +\frac{1}{\sqrt{|h(x',\rho)|}}\frac{\partial }{\partial \rho}\left(\sqrt{|h(x,\rho)|}u^\rho\right)$$
		where $u^{\rho}:=u\cdot \partial_\rho$, $u^\|=u-u^\rho \partial_\rho$ and $\text{div}_{\rho}\,^\|$ equals the divergence of a vector field $u^\|$ tangent to the surface $\rho$ constant.
	\end{lema}
	
	We can use now these formula to give a precise formulation of equation \eqref{integraleqn}. We recall that, given a magnetic field $B$ and an initial current $j_0$, the equation is given by taking the unique current $j$ satisfying 
	
	$$\left\{\begin{array}{ll}
		(B\cdot\nabla)j=(j\cdot \nabla)B & \text{in }\Omega\\
		j=j_0  & \text{on }\partial\Omega_-\\
	\end{array}\right..$$
	and solving the div-curl problem 
	
	\begin{equation} \label{divcurlotravezzzz}\left\{\begin{array}{ll}
			\nabla\times W=j &\text{on }\Omega\\
			\nabla\cdot W=0 &\text{on }\Omega\\
			B\cdot n=f & \text{on } \partial\Omega
		\end{array}
		\right..
	\end{equation}
	
	Therefore, the equation \eqref{integraleqn} is given by 
	
	\begin{equation} \label{intttttttt3D}
		W_{\tau}[B,j_0]|_{\partial\Omega_-}=g\quad \text{on } \partial\Omega_-.
	\end{equation}
	
	Since the current $j$ is uniquely determined by its tangential and normal part, we can write it as 
	
	\begin{equation}\label{integralequ3D2}
		W_{\tau}(B,j^\|_0+j^\rho_0)=g \quad \text{on }  \partial\Omega_-.
	\end{equation}
	
	In order for this procedure to make sense, we need that $\text{div}\,j$ vanishes in $\Omega$. Arguing as in \cite{Alo-Velaz-Sanchez-2023}, one can check that if $j$ is divergence free at $\partial\Omega_-$, and $B$ is divergence free, then $j$ remains divergence free in the whole domain $\Omega$. Therefore, we have to ensure that the divergence of $j$ equals zero on $\partial\Omega_-$. We will now describe the condition on $j_0^\|$ and $j_0^\rho$, depending on $B$, in order to obtain a divergence free current.

	\begin{lema}Let $j$ be a vector field defined on $\Omega$. Then, consider its tangential and normal part, $j^\|$ and $j^\rho$, which are well defined as a vector field and as a function on $\partial\Omega_-$, respectively. Assume further that $B$ is a divergence free vector field satisfying the assumptions of Section \ref{asunciones1}. Then, the divergence of $j$ vanishes on $\partial\Omega_-$ if, and only if, when we take the decomposition 
		
		$$j^\|_0=G+\text{grad}_{\partial\Omega_-}\varphi$$
		of Proposition \ref{hodgevector}, the function $\varphi$ satisfies the following PDE 
		
		\begin{equation} \label{eliptica}
			B^\rho\Delta_{\partial\Omega_-}\varphi=j_0^\rho\frac{\partial B^\rho}{\partial\rho}+\mathcal{L}_{\text{grad}_{\partial\Omega_-}\varphi}B^\rho+\mathcal{L}_{G}B^\rho-\mathcal{L}_{B^\|}j_0^\rho,
		\end{equation}
		where $\Delta_{\partial\Omega_-}$ is the Laplace-Beltrami operator on $\partial\Omega_-$, and $\mathcal{L}$ is the Lie derivative  (c.f. Definition \ref{Lie}).
	\end{lema}
	\begin{proof}
		We begin by noticing that the transport equation \eqref{transportexp} can be written as a condition on a certain geometrical quantity. That is, we can write it as 
		\begin{equation}\label{transport}
			[B,j]=0,
		\end{equation}
		where $[B,j]$ is the conmutator of two vector fields. It is well known from differential geometry (e.g. \cite[Chapter 12]{Lee}) that this is a vector field on the manifold with boundary $\Omega$, so one can read the equation \eqref{transport} in any set of coordinates. In particular, if one chooses a parametrization of the surface $\partial\Omega_-$ given by a function $\Gamma(x^1,x^2)$, then in the coordinates $(x_1,x_2,\rho)$ the equation for the component $\rho$ reads 
		
		$$B^\rho\frac{\partial j^\rho}{\partial\rho}+B^1\frac{\partial j^\rho}{\partial x^1}+B^2\frac{\partial j^\rho}{\partial x^2}=j^\rho\frac{\partial B^\rho}{\partial\rho}+j^1\frac{\partial B^\rho}{\partial x^1}+j^2\frac{\partial B^\rho}{\partial x^2}.$$
		
		This can actually be simplified a bit further, just by noticing that, since both $j^\rho$ and $B^\rho$ are well defined functions on the surface $\partial\Omega_-$, one can write the equation above as 
		
		$$B^\rho\frac{\partial j^\rho}{\partial\rho}=j^\rho\frac{\partial B^\rho}{\partial\rho}+\mathcal{L}_{j^\|}B^\rho-\mathcal{L}_{B^\|}j^\rho.$$
		This allows gives an expression for the dependence on $\rho$ of the current in the equation of the divergence, so we derive that such equation reads, at $\partial\Omega_-$,
		
		$$0=\text{div}\,j|_{\partial\Omega_-}=\text{div}_{\partial\Omega_-}j^\|+\partial_\rho j^\rho=\text{div}_{\partial\Omega_-}j^\|-\frac{1}{B^\rho}\left(j^\rho\frac{\partial B^\rho}{\partial\rho}+\mathcal{L}_{j^\|}B^\rho-\mathcal{L}_{B^\|}j^\rho\right).$$
		
		Due to Lemma \ref{hodgevector} we can write this as 
		
		\begin{equation}\label{eq:elipticaphi}
			\begin{split}
				B^\rho\Delta_{\partial\Omega_-}\varphi&=j^\rho\frac{\partial B^\rho}{\partial\rho}+\mathcal{L}_{j^\|}B^\rho-\mathcal{L}_{B^\|}j^\rho\\
				&=j^\rho\frac{\partial B^\rho}{\partial\rho}+\mathcal{L}_{\text{grad}_{\partial\Omega_-}\varphi}B^\rho+\mathcal{L}_{G}B^\rho-\mathcal{L}_{B^\|}j^\rho,
			\end{split}
		\end{equation}
	\end{proof}
	
	 This will allow us to give a precise description of the equation \eqref{integraleqn}. We begin by defining two operators that will be of use in the sequel. 
	
	\begin{defi}\label{defia3D}
		Let $\Omega\subset \R^3$ and $B:\Omega\longrightarrow \R^3$ be as in Section \ref{asunciones1}. For a given $\psi:\partial\Omega_-\longrightarrow \R$, we define  $A[B]\psi$ as 
		
		$$A[B]\psi=\text{div}^\|_{\partial\Omega_-}\textsf{W}_\tau,$$
		where $\textsf{W}_\tau$ is the tangential component on $\partial\Omega_-$ of the magnetic field $\textsf{W}$ that satisfies the following div-curl system, 
		
		$$\left\{\begin{array}{ll}
			\nabla\times\textsf{W}=j & \text{in }\Omega\\
			\nabla\cdot \textsf{W}=0 & \text{in }\Omega\\
			\textsf{W}\cdot n=0 & \text{on }\partial\Omega
		\end{array}\right.,$$
		where $j$ is the solution of the transport problem 
		
		$$\left\{\begin{array}{ll} 
			(B\cdot \nabla)j=(j\cdot \nabla)B & \text{in }\Omega\\
			j=j_0 &\text{on }\partial\Omega_-
		\end{array}\right.,$$
		where $j_0=n\times \text{grad}_{\partial\Omega_-}\psi+\text{grad}_{\partial\Omega_-}\varphi$, and $\varphi$ is the solution of \eqref{eliptica} with $j^\rho=0$.
	\end{defi}
	
	\begin{defi}\label{defie3D}
		Let $\Omega\subset \R^3$ and $B:\Omega\longrightarrow \R^3$ be as in Section \ref{asunciones1}. For a given $j^\rho:\partial\Omega_-\longrightarrow \R$, we define  $E[B]j^\rho$ as 
		
		$$E[B]j^\rho=\text{div}^\|_{\partial\Omega_-}\overline{\textsf{W}}_\tau,$$
		where $\overline{\textsf{W}}_\tau$ is the tangential component on $\partial\Omega_-$ of the magnetic field $\overline{\textsf{W}}$ that satisfies the following div-curl system, 
		
		$$\left\{\begin{array}{ll}
			\nabla\times \overline{\textsf{W}}=j & \text{in }\Omega\\
			\nabla\cdot \overline{\textsf{W}}=0 & \text{in }\Omega\\
			\overline{\textsf{W}}\cdot n=0 & \text{on }\partial\Omega
		\end{array}\right.,$$
		with $j$ being a solution of the transport problem 
		
		$$\left\{\begin{array}{ll} 
			(B\cdot \nabla)j=(j\cdot \nabla)B & \text{in }\Omega\\
			j=j_0 &\text{on }\partial\Omega_-
		\end{array}\right.,$$
		where $j_0=\text{grad}_{\partial\Omega_-}\varphi+j^\rho$, and $\varphi$ is the solution of \eqref{eliptica} with $\psi=0$.
	\end{defi}
	
	With these two definitions, we can give a compact description of the equation \eqref{integraleqn}, whose solvability will, again, be determined by the invertibility of the operator $A$: 
	
	\begin{prop}
		Let $\Omega\in \R^3$ and $B:\Omega\longrightarrow \R^3$ be as in Section \ref{asunciones1}. Then, given $f\in C^{2,\alpha}(\partial\Omega)$ and $g\in C^{2,\alpha}(\partial\Omega_-;\R^3)$ tangent to $\partial\Omega_-$, the equation \eqref{integraleqn} holds if and only if 
		
		\begin{equation*} 
			A[B]\psi+E[B]\left(\text{div}_{\partial\Omega_-}^\|(n\times g)\right)=\text{div}_{\partial\Omega_-}^\|\left(g-H_\tau\right),
		\end{equation*}
		where $H$ is the unique solution of 
		
		\begin{equation}\label{divcurlboundary}
			\left\{\begin{array}{ll}
				\nabla\times H=0 & \text{in }\Omega\\
				\nabla\cdot H=0 & \text{in }\Omega\\
				H\cdot n=f & \text{on }\partial\Omega
			\end{array}\right..
		\end{equation}
	\end{prop}
	\begin{proof} 
		
		Due to the linearity of the problem, it is easy to see that the equation \eqref{integraleqn} can be written as 
		
		$$W_\tau[B,j_{0,1}]+W_\tau[B,j_{0,2}]=g-H_\tau,$$
		where $j_{0,1}=n\times \text{grad}_{\partial\Omega_-}\psi+\text{grad}_{\partial\Omega_-}\varphi^1$, with $\varphi^1$ solving \eqref{eliptica} with $j^\rho=0$, and $j_{0,2}=\text{grad}_{\partial\Omega_-}\varphi^2+j^\rho$, with $\varphi^2$ solving \eqref{eliptica} with $G=0$. 
		
		Note that we already know $j^\rho_0$. Indeed,  since the solution of our problem has to satisfy $\nabla\times W=j$, we can write 
		
		$$j^\rho|_{\partial\Omega_-}=\left(\nabla\times W\right)\cdot n|_{\partial\Omega_-}=\nabla \cdot \left(n\times W\right)|_{\partial\Omega_-}+W\cdot\left(\nabla\times n\right)|_{\partial\Omega_-},$$
		where $n$ is extended as a smooth vector in $\mathcal{U}$. Notice that, actually, this extension can taken to be $-\nabla (\text{dist}(\cdot))$. As a result, the second term vanishes, for it is the curl of a gradient. On the other hand, since $n$ is orthogonal to the surfaces of $\rho$ constant, we conclude that 
		
		\begin{equation}\label{jrho}
			j^\rho|_{\partial\Omega_-}=\text{div}^\| (n\times g).
		\end{equation}
		
		In other words, the vector field $j^\rho_0$ is uniquely determined by the boundary data $g$. Since, due to Lemma \ref{hodgevector}, the vector field $g$ is uniquely defined by $\text{div}_{\partial\Omega}^\| g$ and $\text{div}_{\partial\Omega}^\|(n\times g)$, and due to the Definitions \ref{defia3D} and \ref{defie3D}, the result follows.
	\end{proof}

	This construction allows us to prove a 3D equivalent to Theorem \eqref{teor:mainteor2D}. The conditions are essentially the same.  Note however, that, since in 3D our domains are simply connected, there is not additional degree of freedom. Therefore, the condition \eqref{condicion} does not play any role. The resulting theorem then reads
	
	\begin{teor}\label{teor:mainteor3D}
		Let $\alpha \in (0,1)$. Let $\Omega\in \R^2$ as described in Section \ref{asunciones1}. Consider $B_0\in C^{2,\tilde{\alpha}}$ a smooth divergence free magnetic field without vanishing points and such that $B_0\cdot n<0$ (resp. $>0$) on $\partial\Omega_-$ (resp. on $\partial\Omega_+$) satisfying the MHS system. Assume further that $\text{curl}\, B_0\in C^{2,\tilde{\alpha}}$ and that the operator $A$ defined in Definition \ref{defia3D} with $B=B_0$ has a kernel consisting only on the zero function. Then, there exists (small) constants $M$, $\delta=\delta(M)>0$  such that, if $ f\in C^{2,\alpha}(\partial\Omega)$ and $g\in C^{2,\alpha}(\partial\Omega_-)$ satisfy the following integrability condition
		
		$$\int_{\partial\Omega_-}f =0,$$
		the following smallness condition
		$$\|f-B_0\cdot n\|_{C^{2,\alpha}}+\|B_{0,\tau} -g\|_{C^{2,\alpha}}+\|\text{curl}\,B_0\|_{C^{2,\alpha}}\leq \delta $$
		then there exists a solution $(B,p)\in C^{2,\alpha}(\Omega)\times C^{2,\alpha}(\Omega)$ of the problem \eqref{ec:maineqn}. Moreover, this is the unique solution satisfying 
		
		$$\|B-B_0\|_{C^{2,\alpha}}\leq M.$$
	\end{teor} 
	
	\section{Some auxiliary results}\label{auxiliaryresults}
	\subsection{Some Tools from pseudodifferential analysis}
	In the course of this work we will use extensively some properties of pseudodifferential operators and their definition on manifolds. We will include a brief discussion here, where we will state the main definition and properties we will employ. All of these results are classical and may be found in \cite{Taylor-1981,Treves}. We begin by recalling the definition of symbol classes:
	\begin{defi}Let $\Omega\subset\R^d$ be an open set and $a(x,\xi)\in C^\infty(\Omega;\R^d)$. We say that $a(x,\xi)$ belongs to the symbol class $S^{m}(\Omega)$ if for every compact set $K\subset\Omega$ and multiindices $\alpha,\beta$, there exists a constant $C_K$ satisfying
		$$|\partial^\beta_x\partial^\alpha_\xi a(x,\xi)|\leq C_{K,\alpha,\beta}(1+|\xi|)^{m-|\alpha|}.$$
	\end{defi}
	
	Once we have defined the class of symbols we are working with, we can discuss the operator they give rise to
	\begin{defi}
		Let $\Omega\subset\R^d$ be an open set and $a(x,\xi)$ be symbol in $S^m(\Omega)$. We define the pseudodifferential operator $\textsf{A}:C^\infty_c(\Omega)\rightarrow C^\infty(\Omega)$ as
		$$\textsf{A}u=\frac{1}{(2\pi)^d}\int_{\R^d}a(x,\xi)\widehat{u}(\xi)e^{i\xi \cdot x}d\xi.$$ 
		
		The number $m$ is usually denoted the degree of the pseudodifferential operator.
	\end{defi}
	
	It is possible to define pseudodifferential operators in arbitrary manifolds via local charts. One then has to show that the property of being a pseudodifferential operator is invariant under diffeomorphisms. More precisely, we have the following property, c.f. \cite{Taylor-1981}: 
	
	\begin{teor}
		$\textsf{A}$ is invariant under diffeomorphisms, i.e. if $U$ is another open set in $\R^d$ and $\phi:\Omega\rightarrow U$ is a diffeomorphism, the operator $A$ defined as 
		
		$$u\mapsto \p(a)\left(u\circ \phi\right)\circ \phi^{-1}$$
		defines another pseudodifferential operator with a symbol $\tilde{c}(x,\xi)$ that satisfies 
		
		\begin{equation}\label{expansion}
			\tilde{c}(\phi(x),\xi)- \sum_{\alpha\geq 0}\frac{1}{\alpha!}\varphi_\alpha(x,\xi)\partial_x^\alpha c(x,J_x^t\xi)\in S^{m-N-1}(\Omega),
		\end{equation}
		where $J_x$ is the jacobian matrix of the map $\phi$, and $\varphi_\alpha(x,\xi)$ is a polynomial in $\xi$ of degree at most $|\alpha|/2$ with $\varphi_0(x,\xi)=1$. 
	\end{teor}

	Invariance under diffeomorphisms allows one to define pseudodifferential operators on smooth manifolds $M$, simply by stating that a linear operator $\textsf{A}:C^{\infty}(M)\longrightarrow C^\infty(M)$ is a pseudodifferential operator of degree $m$ if and only if, for any chart $(\varphi,U)$, the localization of $\textsf{A}$ given by 
	
	$$C_c^\infty(\varphi(U))\ni u\mapsto \textsf{A}(u\circ \varphi)\circ \varphi^{-1}$$
	is a pseudodifferential operator of degree $m$ in $\varphi(U)$. 
	
	Moreover, one can realize that, if the first term in the asymptotic expansions \eqref{expansion} of $c$ is homogeneous in $\xi$, it transform like an element in a section of $T^\star M$. Therefore, one can define the principal symbol (the first term in the asymptotic expansion) as a section in $T^\star M$. 
	
	Furthermore, in the case of $M$ compact, we can compose pseudodifferential operators, and the corresponding principal symbol is the product of the principal symbols:
	
	\begin{prop}\label{composimbolos}
		Let $M$ be a compact smooth manifold and $\textsf{A},\textsf{B}:C^\infty(M)\longrightarrow C^\infty(M)$ two pseudodifferential operators of degree $m_\textsf{A}$ and $m_\textsf{B}$ with principal symbols $\sigma_\textsf{A}$ and $\sigma_\textsf{B}$, respectively. Then, its composition is a pseudodifferential operator of degree $m_\textsf{A}+m_\textsf{B}$ with principal symbol
		
		$$\sigma_{\textsf{A}\cdot \textsf{B}}(x,\xi)=\sigma_\textsf{A}(x,\xi)\cdot \sigma_\textsf{B}(x,\xi).$$
	\end{prop}
	
	This observation also allows one  to define a symbolic calculus for the principal symbol that extends to the case of vector valued pseudodifferential operators, which allows us to work with pseudodifferential operators acting not only on $C^\infty(M)$, but also on sections of vector bundles. In this article we will only deal with two very particular vector bundles, so we will avoid referring to the most general theory. We refer to Chapter 2 in \cite{Treves} for further details and more general results. The definition of this vector valued operators is just via charts again. However, one has to keep in mind the new vector structure. Therefore, we introduce the operation $\varphi_\star X$ where, for every $X=(X_1(y),...,X_n(y))$ vector field defined in $\varphi(U)$ for $(U,\varphi)$ a local chart, $\varphi_\star X$ is the vector field in $U$ given by 
	
	$$\varphi_\star X(p)=\sum_{i=1}^dX_id_{\varphi^{-1}(p)}\varphi^{-1}(e_i),$$
	where $\{e_1,...,e_n\}$ is the canonical basis for $\R^d$. Reciprocally, given a vector field $X$ given in the coordinate chart $(U,\varphi)$ by 
	
	$$\sum_{i=1}^d X_i\frac{\partial}{\partial \textsf{x}_i},$$
	$\varphi^\star X$ defines a vector field in $\varphi(U)$ simply by $\varphi^\star X=(X_1,...,X_d)$. Using this notation, we can define pseudodifferential operators on vector fields.
	
	\begin{defi}
		Let $M$ be a compact manifold, and consider its tangent bundle $TM\longrightarrow M$. Then, we say that a linear map $\textsf{A}:\Gamma(TM)\longrightarrow  \Gamma(TM)$ is a pseudodifferential operator of degree $m$ if, for every chart $(U,\varphi)$, the map 
		
		$$C_c^\infty(\varphi(U);\R^n)\ni X\mapsto \varphi^\star\textsf{A}(\varphi_\star X)$$
		is a vector valued pseudodifferential operator of degree $m$. 
	\end{defi}
	
	Similar kind of constructions allow us to define pseudodifferential operators from $C^{\infty}(M)$ to sections in the tangent bundle and vice versa. In this case, the principal symbol is well defined again, but in this case, for each $(x,\xi)\in T^\star M$, $\sigma(x,\xi)$ is a linear map, i.e. a matrix, where the rules of composition still apply. 
	
	One of the main properties of these pseudodifferential operators are its natural continuity properties when acting on spaces of functions. This is the following property: 
	
	\begin{teor}
		Let $M$ be a closed manifold and $s\in \R$. Then: 
		\begin{itemize}
			\item If $a\in S^m(M)$, it holds that 
			
			$$\textsf{A}:H^s(M)\longrightarrow H^{s-m}(M)$$
			is a bounded operator. 
			\item if $\textsf{A}:C^\infty(M)\longrightarrow \Gamma(TM)$ is a pseudodifferential operator of degree $m$, 
			
			$$\textsf{A}:H^s(M)\longrightarrow H^{s-m}(TM)$$
			is a bounded operator. 
			\item if $\textsf{A}:\Gamma(TM)\longrightarrow C^\infty(M)$ is a pseudodifferential operator of degree $m$, 
			
			$$\textsf{A}:H^s(TM)\longrightarrow H^{s-m}(M)$$
			is a bounded operator. 
			\item if $\textsf{A}:\Gamma(TM)\longrightarrow \Gamma(TM)$ is a pseudodifferential operator of degree $m$, 
			
			$$\textsf{A}:H^s(TM)\longrightarrow H^{s-m}(TM)$$
			is a bounded operator. 
		\end{itemize}
	\end{teor}
	
	One further property of this kind of operators can be extracted under suitable assumptions on the principal symbol.  Notice that, if the principal symbol of a pseudodifferential operator was invertible, then one could naively just define the inverse operator $\textsf{A}$ by defining $\textsf{A}^{-1}=\p(a^{-1})$. This is what happens, for example, with differential operators and multipliers in $\R^n$. This heuristically holds for pseudodifferential operators as well, but only for high frequencies. This control of the high frequencies leads to the invertibility of the operator \textsf{A} up to a compact perturbation, which implies that the operator \textsf{A} is Fredholm. Let us recall the definition:

	\begin{defi}\label{defiFredholm}
		Let $E,F$ be Banach spaces, and $T\in \mathcal{L}(E,F)$. We say that $T$ is a Fredholm operator of index $\text{ind}\,T$ if both $\text{ker}\,T$ and $F/\text{im}\,T$ are finite dimensional and $\text{ind}\,T=\text{ker}\,T-F/\text{im}\,T$.
	\end{defi}

	Let us now make precise the heuristics before about ``inverting the symbol'':
	
	\begin{defi}
		Let $M$ be a compact manifold and let $\textsf{A}:C^\infty(M)\longrightarrow C^\infty(M)$ be a pseudodifferential operator on $M$. Let $c(x,\xi)$ denote its principal symbol. Assume further that it is homogeneous in $\xi$. Then, we say that $\textsf{A}$ is elliptic if $c$ is never zero on $T^\star M\setminus \{0\}$. 
	\end{defi}
	
	Then, due to the fact that elliptic pseudodifferential operators are invertible up to a compact perturbation, c.f. \cite{Treves}, we obtain the following theorem:
	
	\begin{teor}\label{Fredholm}
		Let $\textsf{A}:C^\infty(M)\mapsto C^\infty(M)$ be an elliptic pseudodifferential operator of order $m$. Then, for every $s\in \R$, the bounded operator 
		
		$$\textsf{A}:H^{s-m}(M)\longrightarrow H^{s}(M)$$
		is Fredholm, and has an index independent of $s$.
	\end{teor}
	
	The concept of index is of great importance, as it gives one information about the properties of the operator. For instance, in the course of our work we will find that the equation for the current is given by an equality of the form $\textsf{A}j_0=G,$ where $\textsf{A}$ is an elliptic pseudodifferential operator. Therefore, knowing the index allows one to know whether one can expect this operator to be invertible or not. Furthermore, it gives a criterion of solvability of the equation in terms of the kernel of $\textsf{A}$, like the one that appears in the statement of our main Theorem, \ref{teoremacasos}.
	
	The index of a pseudodifferential operator is a stable concept, in the sense that any two elliptic pseudodifferential operators with the same principal symbol have the same index. Furthermore, since the index of a Fredholm operator is stable under homotopy (c.f. \cite{Hormander}), determining the index is a topological problem. The solution to such problem in full generality is a difficult one that we will not need in the course of this paper. On the contrary, we will employ a simpler version for the case when the pseudodifferential operator we are dealing with is homotopic to a very simple one.
	
	\begin{teor}\label{teorindice}(c.f. \cite{Treves})
		Let $(M,g)$ be a compact Riemannian manifold and  $\textsf{A}:C^\infty(M)\longrightarrow C^\infty(M)$ be a pseudodifferential operator of order $m$. Then, denote by $\sigma$ its principal symbol, and let $K$ denote the set 
		
		$$K:=\{|\xi|^m_{T^\star M}c(x,\xi)\,:\, (x,\xi)\in T^\star\setminus\{0\}\},$$
		which is compact due to the ellipticity of \textsf{A}. Now, assume that there is an open neighborhood $\mathcal{O}$ of $K$ and an homotopy $f:\mathcal{O}\times[0,1]\rightarrow \mathbb{C}\setminus{\{0\}}$ satisfying that $f(z,0)=0$ and $f(z,1)=z_0$ for some $z_0\in\mathbb{C}\setminus\{0\}$. Then, the index of $\textsf{A}$ is zero. 
	\end{teor}
	
	\begin{obs}
		In the course of our work we will deal with symbols for which this proposition trivially holds, where $|\xi|^m_{T^\star M}c(x,\xi)$ takes values in $\{\text{Re}\,z>a\}$ for some $a>0$, where the hypothesis of Theorem \ref{teorindice} are trivially satisfied.
	\end{obs}
	
	In the sequel, we shall use this theory to prove that the operator $A$ in \eqref{operadorprincipal} is Fredholm of index zero. This will allow to characterize the invertibility of the operator $A$ just by checking whether its kernel is trivial or not.
	
	It will be also relevant to study the continuity of of pseudodifferential operators that are not smooth on the $x$ variable. They have already been studied in \cite{Jurgen-87,Taylor-NLPDE}. We recall their definition here: 
	
	\begin{defi}\label{symbol:class:def}
		Let $s\in\mathbb{R}_+\setminus\N$, and $m\in\R$. We define the symbol class $S^{m}(\R^d,s)$ consisting on  functions $a(x,\xi): \R^d\times \R^d\longrightarrow \mathbb{C}$ such that for every multi-index $\gamma$ 
		\begin{equation}\label{symbol}
			\|\partial_\xi^\gamma a(\cdot,\xi)\|_{C^s}\leq C_{\gamma,s}(1+|\xi|)^{m-|\gamma|}.
		\end{equation}
	\end{defi}
	
	Therefore, if $a\in S^{m}(\mathbb{R},s)$ is a symbol, then
	\begin{equation}\label{op}
		\p(a)u(x)=\frac{1}{(2\pi)^d}\int_{\R^d} e^{ix\cdot \xi}a(x,\xi)\reallywidehat{u}(\xi)d\xi, \quad \mbox{for all }  x\in\mathbb{R}^{d}
	\end{equation}
	defines the associated pseudo-differential operator. 

	\begin{defi}
		We can define the formal adjoint of the operator $\p(a)$ via 
		\begin{equation}\label{op:adjoint}
			\p(a)^{\star}v(x)=\frac{1}{(2\pi)^d}\int\int e^{i(x-y)\cdot \xi} \overline{a(y,\xi)}v(y) \ dy d\xi. 
		\end{equation}
	\end{defi}

	It will be also convenient when estimating the norms of pseudo-differential operators \eqref{op} to define the following semi-norms
	\begin{defi}
		Given $s\in\mathbb{R}_+\setminus\N$ and $l\in \N$, we define the following family of semi-norms in the symbol class $S^m(\R^d,s)$:
		
		$$\|a\|_{m,s,l}=\sup_{|\gamma|\leq l}\sup_{\xi\in\R^n}(1+|\xi|)^{|\gamma|-m}\|\partial_\xi^\gamma a(\cdot,\xi)\|_{C^s}.$$
	\end{defi}
	
	Pseudo-differential operators are particularly suited to derive continuity estimates on Besov spaces, even when the symbol belongs to $S^m(\R^d,s)$. The result can be found in \cite[Lemma 4.5]{Jurgen-87} and a thorough proof in the Appendix in \cite{Alo-Velaz-Sanchez-2023}, where the precise estimates are obtained. We recall the definition of the Besov spaces. We begin by introducing a suitable partition of unity:
	
	\begin{prop}[c.f. \cite{Chemin-Danchin-Bahouri-2011}]
		Let $\mathcal{C}$ be the annulus $\{\xi\in \R^d\,:\, 3/4\leq |\xi|\leq 8/3\}.$ There exist radial functions $\chi$ and $\varphi$ such that 
		\begin{itemize}
			\item $\chi\in C^\infty_c(B_{4/3}(0);[0,1])$
			\item $\varphi\in C^\infty_c(\mathcal{C},[0,1]).$
			
		\end{itemize}
		that satisfy

		$$\forall \xi\in \R^d,\quad \chi(\xi)+\sum_{j\geq 0}\varphi(2^{-j}\xi)=1,$$
		and 
		$$\forall\xi\in\R^d\setminus{\{0\}},\quad \sum_{j\in \Z}\varphi(2^{-j}\xi)=1.$$
		
		Furthermore, the support of $\varphi$ and $\chi$ satisfy 
		
		$$|j-j'|\geq 2\Rightarrow \text{supp }\varphi(2^{-j}\cdot)\cap \text{supp }\varphi(2^{-j'}\cdot)=\emptyset,$$
		and 
		$$j\geq 0\Rightarrow \text{supp }\chi\cap \text{supp }\varphi(2^{-j}\cdot)=\emptyset.$$
	\end{prop}
	
	By means of these functions, we can define the Littlewood Paley projections

	\begin{defi}
		We define the  ``Littlewood Paley projections'', $\Delta_j$, as \begin{itemize}
			\item $\Delta_j:=0\quad \forall j\leq -2.$
			\item $\Delta_{-1}:=\chi(D)$
			\item $\Delta_j:=\varphi(2^{-j}D)\quad \forall j\geq 0.$
		\end{itemize}
	\end{defi}
	
	We can now define the Besov spaces as follows:
	
	\begin{defi}
		Let $s\in \R$ and $p,r\in [1,\infty]$. We define the Besov space $B^s_{p,r}(\R^d)$ as the space of distributions $u\in \mathcal{S}'(\R^d)$ satisfying that 
		
		$$\|\left(2^{js}\|\Delta_j u\|_{L^p}\right)_{j\in\N}\|_{L^r}<\infty,$$
		where the operators $\Delta_j$ are the Littlewood Paley projections.
	\end{defi}
	
	\begin{teor}\label{boundedness:Besov}
		Let $a\in S^m(s)$, then $\p(a)$ defined in \eqref{op} admits an extension to a bounded operator 
		\begin{equation}
			\p(a):B^{m-s}_{1,1}(\mathbb{R}^{d})\longrightarrow B^{-s}_{1,1}(\mathbb{R}^{d}).
		\end{equation}
		More precisely, we have that
		\begin{equation}
			\|\p(a)\|_{\mathcal{L}(B^{m-s}_{1,1},\,B^{-s}_{1,1})}\leq C\|a\|_{m,s,2|\gamma|}, \quad 
		\end{equation}
		for all multi-index $\gamma$ with $|\gamma|>d$  
	\end{teor}
	
	Using duality properties of Besov spaces, we find
	\begin{coro}
		Let $a\in S^{m}(\R^d,s)$, then $\p(a)^{\star}$ defined in \eqref{op:adjoint}
		admits an extension to a bounded operator
		\begin{equation}
			\p(a)^{\star}:B^{s}_{\infty,\infty}\longrightarrow B^{s-m}_{\infty,\infty}.
		\end{equation}
		Moreover, 
		\[ \|\p(a)^{\star}\|_{\mathcal{L}(B^{s}_{\infty,\infty}, B^{s-m}_{\infty,\infty})}\leq C\|a\|_{m,s,2|\gamma|}.\]
	\end{coro}
	Therefore, recalling that $B^{s}_{\infty,\infty}(\R^d)\sim C^{s}(\R^d)$, we have obtain that for symbols $a\in S^{m}(\R^d,s)$
	\[ \|\p(a)^{\star}\|_{\mathcal{L}(C^{s}(\R^{d}), C^{s-m}(\R^{d}))}\leq C\|a\|_{m,s,2|\gamma|}. \]

	\subsection{General statement about kernels}\label{sec:kernels}

	It is well known that the Fourier Transform brings convolutions into products. Therefore, it transforms convolution operators into multipliers. It is then a classical topic of Harmonic Analysis to relate the singularity of the convolution kernel with the decay and regularity of the multiplier it gives rise to. In the case of pseudodifferential operators the argument is similar, as any operator of the form 
	
	$$Tu(x):=\int_{\R^d}k(x,x-y)u(y)dy$$
	can be written as 
	
	$$Tu(x):=\frac{1}{(2\pi)^d}\int_{\R^d}a(x,\xi)\widehat{u}(\xi)e^{ix\xi}d\xi\quad \text{for }a(x,\cdot)=\widehat{k}(x,\cdot),$$
	as long as the Fourier Transform is well defined. It is our goal then in this section to obtain which conditions one has to impose on $k$ to obtain that the symbol $a$ is in $S(\R^d,s)$.
	
	This is a classical result that can be found, for instance, in \cite{stein}. We will, however, include a prove here. The reason is twofold: On the one hand, the classical books about pseudodifferential operators deal with the case of $C^\infty$ symbols, whereas in our case we need regularity estimates for symbols that are not that smooth on the $x$ variable. This will not be  a big difficulty, since these classical results also hold for symbols in $S(\R^d,s)$. The second reason is that in our case, there will appear kernels whose singularity is of logarithmic type, a situation that requires an extra cancellation condition to ensure that the decay of the corresponding symbol is the adequate one. This is essential, it will be explained in  Remark \ref{logaritmo}. This will not be the case in higher dimensions because in such situation there will be no logarithmic terms. 
	
	\begin{prop}\label{prop:kernels}
		Let $k(x,z):\R\times(\R\setminus \{0\})\longrightarrow \mathbb{C}$ be compactly supported and assume that it  satisfies the following condition
		
		\begin{equation}\label{singularity} 
			\|\partial_z^\ell k(\cdot, z)\|_{C^{1,\alpha}}\leq C_\ell |z|^{-N-|\ell|}\qquad z\in \R\setminus\{0\},
		\end{equation}
		for any $N\geq 0$ such that $N+|\ell|>0$; and the following cancellation condition 
		
		\begin{equation}\label{cancellation}
			R\int_{\frac{1}{R}\geq |z|\geq \frac{1}{2R}}\|k(\cdot,z)-k(\cdot,-z)\|_{C^{1,\alpha}}dz\leq C_{canc} \qquad z\in \R\setminus\{0\}.
		\end{equation}
		
		Then, the symbol $a(x,\xi)$ defined as 
		
		$$a(x,k):=\frac{1}{2\pi}\int_{\R}k(x,z)e^{i\xi z}dz,$$
		which is well defined because $k$ is integrable, satisfies that, for every $\ell\in \N$, 
		
		\begin{equation}\label{decay}
			\|\partial_{\xi}^\ell a(x,\xi)\|_{C^{1,\alpha}}\leq C\left(C_{canc}+\sum_{j=1}^{\ell+1}C_j\right)|\xi|^{-1-|\ell|}\qquad \xi\in \R.
		\end{equation}
	\end{prop}
	\begin{proof}
		We will focus first on computing the decay of $\|a(\cdot,\xi)\|_{C^{1,\alpha}}$, as the decay of the norm of higher derivatives follows iterating the argument. 
		
		Note that, since $k$ has compact support on both variables, it is clear that $a$ is smooth on the $\xi$ and we can differentiate under the integral sign with respect to $x$, so the only thing remaining to prove is the decay at infinity. Take $|\xi|$ and consider $\eta\in C^\infty_c(\R,[0,1])$ a cutoff satisfying that $\eta\equiv 1$ on $[-1/2,1/2]$ and $\eta\equiv 0$ on $|z|\geq 1$. Moreover, we will take it to be even, so that its derivative is odd. Then, we can define $\eta_\xi(z):=\eta(\xi z)$. Therefore, we can obtain easily that 
		
		\begin{equation} \label{separacion} 
			\partial_x a(x,\xi)=\int_{\R}\partial_x k(x,z)e^{i\xi z}dz=\int_{\R}\partial_xk(x,z)\eta_\xi (z)e^{i\xi z}dz+\int_{\R}\partial_xk(x,z)(1-\eta_\xi(z)) e^{i\xi z}dz. 
		\end{equation} 
		
		Notice that this separates the integral in the regimes in which $|z|\lesssim 1/\xi$ and $|z|\gtrsim 1/|\xi|$. Now we examine each term separatedly.  For the first summand in \eqref{separacion} we use use a standard argument of integration by parts, as well as the fact that $\partial_ze^{i\xi z}=\partial_z \left(e^{i\xi z}-1\right)$ Therefore, 
		
		\begin{multline}\label{regimenzgeqxi}
			\int_\R \partial_xk(x,z)\eta_\xi (z)e^{iz\xi}dz=\frac{1}{i\xi}\int \partial_xk(x,z)\eta_\xi (z)\partial_z\left(e^{iz\xi}-1\right)dz=\\
			-\frac{1}{i\xi}\int_\R\left(\partial_x\partial_zk(x,z)\eta_\xi(z)+\partial_xk(x,z)\xi (\eta')_\xi(x,z)\right)\left(e^{i\xi z}-1\right)dz.
		\end{multline}

		On the other hand, for the second term in \eqref{separacion} we notice that the integrand $\partial_z k(x,z)(1-\eta_\xi(z))$ is supported away from zero. Therefore, we can integrate by parts as usual, so
		
		$$\int_\R \partial_xk(x,z)(1-\eta_\xi(z))e^{iz\xi}dz=-\frac{1}{i\xi}\int_\R \partial_x\partial_z k(x,z)(1-\eta_\xi(z))e^{iz\xi}dz+\frac{1}{i}\int_\R \partial_xk(x,z)(\eta')_\xi (z)e^{iz\xi}dz$$
		
		Notice that the second term cancels with part of the second term in \eqref{regimenzgeqxi}, resulting in 
		
		\begin{equation} \label{identidadsimbolo}
			\begin{split}
				\int_\R \partial_xk(x,z)e^{i\xi z}dz&=-\frac{1}{i\xi}\int_\R \partial_x\partial_z k(x,z)\eta_\xi (z)\left(e^{i\xi z}-1\right)dz+\frac{1}{i}\int_\R \partial_xk(x,z)(\eta ')_\xi(z) dz\\
				&-\frac{1}{i\xi}\int_\R \partial_x\partial_z k(x,z)(1-\eta_\xi (z))e^{iz\xi}dz.
			\end{split}
		\end{equation}
	
		For the first term on the right hand side we use the behaviour of the singularity on $k$ to obtain
	
	$$\left|\int_\R\left(\partial_x\partial_z k(x,z)\eta_\xi(z)\right)\left(e^{i\xi z}-1\right)\,dz\right|\leq C|\xi|\int_{-1/|\xi|}^{1/|\xi|}dz\leq C.$$
	
	As a result, the first term has the decay in \eqref{decay}. For the third term on the right hand side in \eqref{identidadsimbolo} we  integrate by parts once again, so that it reads
		
		\begin{multline}
			\frac{1}{\xi}\int_\R \partial_x\partial_z k(x,z)(1-\eta_\xi (z))e^{iz\xi}dz=\frac{1}{i\xi^2}\int_\R \partial_x\partial_z k(x,z)(1-\eta_\xi (z))\partial_z e^{iz\xi}dz\\
			=-\frac{1}{i\xi^2}\int_\R \partial_x\partial_z^2 k(x,z)(1-\eta_\xi(z))e^{iz\xi}dz+\frac{1}{i\xi}\int_\R \partial_x\partial_z k(x,z)(\eta')_\xi (z)e^{iz\xi}dz.
		\end{multline}
		
		We estimate the first two terms on the right hand side as follows:
		
		$$\left|\int_\R \partial_x\partial_z k(x,z)(1-\eta_\xi(z))e^{iz\xi}dz\right|\leq C\int_{|z|\geq \frac{1}{2|\xi|}}\frac{1}{|z|^2}dz\leq C |\xi|.$$
		
		On the other hand, 
		
		$$\left|\int_\R \partial_x\partial_z k(x,z)(\eta')_\xi (z)e^{iz\xi}dz\right|\leq C\int_{\frac{1}{2|\xi|}\leq |z|\leq \frac{1}{|\xi|}}\frac{1}{|z|}dz\leq C.$$
		
		Therefore, we conclude that 
		
		$$\int_{\R}\partial_xk(x,z)e^{iz\xi}dz\leq O(1/|\xi|) +\int_{\R}\partial_xk(x,z)(\eta')_\xi (z) dz.$$
		
		The last term in \eqref{identidadsimbolo} cannot be estimated by $1/|\xi|$ in general without  extra requirements on $k$. We want to find estimates for kernels that might have a logarithmic singularity. That is why a cancellation condition is imposed. Note that, as $\eta$ is odd, condition \eqref{cancellation} allows us to bound 
		
		$$\left|\int_{\R}\partial_xk(x,z)(\eta')_\xi (z) dz\right|\leq \left|\int_{\R}\partial_xk(x,z)-\partial_xk(x,-z)(\eta')_\xi (z) dz\right|\leq C\cdot C^{canc}/|\xi|,$$
		thus concluding that 
		
		$$\|a(x,\xi)\|_{C^1}\leq C(C^{canc}+C_0+C_1+C_2)|\xi|^{-1}.$$
		
		In order to compute the $C^{\alpha}$-norm of the derivatives we just use the identity \eqref{identidadsimbolo} evaluated in $x_1$ and $x_2$, leading to 
		\small
		\begin{equation} \label{identidadsimbolo2}
			\begin{split}
				\int_\R \left[\partial_xk(\cdot,z)\right]^{x_2}_{x_1}e^{i\xi z}dz&=-\frac{1}{i\xi}\int_\R \left[\partial_x\partial_z k(\cdot,z)\right]^{x_2}_{x_1}\eta_\xi (z)\left(e^{i\xi z}-1\right)dz-i\int_\R \left[\partial_xk(\cdot,z)\right]^{x_2}_{x_1}(\eta ')_\xi(z) dz\\
				&-\frac{1}{i\xi}\int_\R \left[\partial_x\partial_z k(\cdot,z)\right]^{x_2}_{x_1}(1-\eta_\xi (z))e^{iz\xi}dz,
			\end{split}
		\end{equation}
		\normalsize where $\left[f(\cdot)\right]^{x_2}_{x_1}=f(x_1)-f(x_2).$  We now notice that 
		
		$$|\left[\partial^\beta \partial_x k(\cdot,z)\right]^{x_2}_{x_1}|\leq \|\partial^\beta k(\cdot,z)\|_{C^{1,\alpha}}|x_1-x_2|^\alpha,$$
		and that 
		
		$$|\left[\partial^\beta \partial_x k(\cdot,z)\right]^{x_2}_{x_1}-\left[\partial^\beta \partial_x k(\cdot,z)\right]^{x_2}_{x_1}|\leq \|\partial^\beta k(\cdot,z)-k(\cdot,-z)\|_{C^{1,\alpha}}|x_1-x_2|^\alpha,$$
		for every $x_1,x_2$. Therefore, we can employ the same arguments as before. 
		
		One just performs the same reasoning for higher $z$ derivatives of $k$. Notice that, in the estimates for the higher derivatives of $k$, one does not need to use any cancellation condition like \eqref{cancellation} anywhere. 
	\end{proof}
	
	\begin{obs}Notice that we only need the kernel to satisfy the cancellation condition \eqref{cancellation}. No assump- tions on the derivatives are needed.
	\end{obs}
	
	\begin{obs}\label{logaritmo}
		The cancellation condition is necessary, since there are examples of kernels which satisfy \eqref{singularity} but lack the cancellation condition \eqref{cancellation}, and whose symbol fails to have the correct decay at infinity. One simple example is
		
		$$k(z)=\ln(z)\textbf{1}_{\{z>0\}},$$
		which has a Fourier Transform that decays like $\log(|\xi|)/|\xi|$ as $|\xi|\longrightarrow \infty$ (for example, using the method of steepest descents, as done in \cite[Chapter 6.6, Example 1]{Berger}) and then it therefore it is strictly larger than $1/|\xi|$. 
	\end{obs}
	
	In dimension two we obtain a similar statement, but with the difference that we do not need a cancellation condition. The reason is that we do not have logarithmic singularity of the kernel. The corresponding result reads 
	
	\begin{prop}\label{kernel3D}
		Let $a(x,z):\R^2\times (\R^2\setminus\{0\})\longrightarrow \mathbb{C}$ be compactly supported and assume that it  satisfies the following condition
		
		\begin{equation} \label{singularidad3D}
			\|\partial_z^\beta k(\cdot, z)\|_{C^{1,\alpha}}\leq C_\beta |z|^{-1-N-|\beta|} \qquad z\in \R\setminus\{0\}
		\end{equation}
		for any $N\geq 0$ and multiindex $\beta$ .

		Then, the symbol $a(x,\xi)$ defined as
		$$a(x,\xi)=\frac{1}{(2\pi)^2}\int_{\R^2}k(x,\xi)e^{i\xi\cdot  x}dx\qquad \xi\in \R,$$
		satisfies
		
		\begin{equation*}
			\|\partial_{\xi}^\beta a(\cdot,\xi)\|_{C^{1,\alpha}}\leq C\left(\sum_{|\gamma|\leq |\beta|+1} C_\gamma\right)|\xi|^{-1-|\beta|}.
		\end{equation*} 
	\end{prop}
	\begin{proof}
		As in the case of one dimensional kernels, we integrate by parts, which leads to 
		\small
		\begin{equation}  \label{ecuacioncilla} 
			\int_{\R^2}k(z)e^{i\xi z}dz=\frac{-1}{i\xi_k}\int \partial_{z_k} k(z)\eta_\xi(z)(e^{iz\xi}-1)-\frac{1}{i\xi_k}\int\partial_{z_k} k(z)(1-\eta_\xi(z))e^{i\xi z}+\frac{1}{i}\int k(z)(\partial_{z_k}\eta)_\xi(z).
		\end{equation}
		\normalsize
		The  required bounds are essentially the same as in the 1D case. The first two terms are estimated in the same way, where the higher order singularity of $k$ is compensated by the contribution of the jacobian in polar coordinates. The difference relies in the last term. Here, one might think that a cancellation condition is needed. However, the condition \eqref{singularidad3D} rules out a possible logarithm-like behaviour, i.e. the singularity is always dominated by power law in $|z|$, unlike in the 1D case of Theorem \ref{prop:kernels}. Therefore, the last term in the right hand side in \eqref{ecuacioncilla} can be estimated by
		
		$$\left|\int k(z)(\partial_{z_k}\eta)_\xi(z)dz\right|\leq C\int_{1/{2\xi}<|z|<1/\xi}|k(z)|dz\leq C\cdot C^0\int_{1/{2|\xi|}}^{1/|\xi|}dz=C\cdot C^0\frac{1}{|\xi|}.$$
		Once we have estimated this term, the rest of the proof is a straightforward adaptation of the one-dimensional case. 
	\end{proof}
	
	\subsection{Solution of the div-curl system}
	Finding solutions of the div-curl system is a classical problem whose solution can be found in \cite{Cheng_Arthur_Steve}. The first reference to the construction of an integral kernel for the solution, i.e., for the Biot Savart operator, may be found in \cite{Enciso}. In this work, however, we need an explicit expression for the leading term of such integral kernel, so that the Biot Savart kernel consists on this term plus a less singular (i.e. more regular) contribution. In the construction of the kernel we will employ extensively the notions of differential geometry we introduced in the beginning of Section \ref{integralequcurr3d}, and the notations will be the ones discussed there.
	
	\begin{defi}
		Let $j\in C^{1,\alpha}(\overline{\Omega},\R^3)$ be a vector field and consider $\eta$ a smooth cutoff with support in $\mathcal{U}$. Then, we define the vector fields $\mathfrak{J}_1$ and $\mathfrak{J}_2$ as 
		
		\begin{multline}\label{ansatz}
			\frac{1}{4\pi}\int_\Omega \left(\frac{x-y}{|x-y|^3}+\frac{x-y^\star}{|x-y^\star|^3}\right)\wedge j^\rho(y)\eta(y)dy'd\rho\\+\frac{1}{4\pi}\int_\Omega \left(\frac{x-y}{|x-y|^3}-\frac{x-y^\star}{|x-y^\star|^3}\right)\wedge j^{\|}(y)\eta(y)dy'd\rho=\mathfrak{J}_1+\mathfrak{J}_2,
		\end{multline}
		where $y^\star$ corresponds to the reflexion of the point $y$ along the tangent space of $\partial\Omega_-$, i.e. if $y$ is given by $\varphi(y')+\rho\nu(y')$ for some parametrization of the surface $\partial\Omega$, then $y^\star=\varphi(y')-\rho\nu$.
	\end{defi}
	
	The main theorem of this section then reads as follows 
	
	\begin{teor}\label{teor:biotsavart}
		Let $\Omega\subset \R^3$ a simply connected open set with smooth boundary. Assume further that $j\in C^{1,\alpha}(\overline{\Omega},\R^3)$ is a divergence free vector field. Let $f\in C^{2,\alpha}(\Omega)$, and assume that the following compatibility condition holds: 
		
		\begin{equation}\label{compatibilitydivcurl}
			\int_{\Gamma}j\cdot n\, dS=0 \quad \text{and}\quad  \int_{\partial\Omega}fdS=0,
		\end{equation} 
		for every $\Gamma$ connected component of $\partial\Omega$. Then, we can construct a magnetic field $B\in C^{2,\alpha}(\overline{\Omega};\R^3)$ satisfying 
		
		$$\left\{\begin{array}{ll}
			\nabla\cdot B=j & \text{in }\Omega\\
			\nabla\cdot B=0 & \text{in }\Omega\\
			B\cdot n=f & \text{on }\partial\Omega
		\end{array}\right.,$$
		so that if we define $b$ to be  
		
		$$b:=B-\mathfrak{J}_1-\mathfrak{J}_2-\nabla v\qquad \text{on }\partial\Omega,$$
		where $v$ solves the elliptic problem 
		
		$$\left\{\begin{array}{ll}
			\Delta v=0 &\text{in }\Omega\\
			\frac{\partial v}{\partial n}=f & \text{on }\partial\Omega
		\end{array}\right.,$$
		then $b\in C^{2,\tilde{\alpha}}(\Omega)$, and 
		
		$$ \|b\|_{C^{2,\tilde{\alpha}}}\leq C(\|f\|_{C^{2,\alpha}}+\|j\|_{C^{1,\alpha}}).$$
	\end{teor}
	
	\begin{obs}
		The main important fact of this theorem is noticing that if $B$ solves
		
		$$\left\{ \begin{array}{ll}
			\nabla\times B=j & \text{in }\Omega\\
			\nabla\cdot B=0 & \text{in }\Omega\\
			B\cdot n=0&\text{on }\partial\Omega 
		\end{array}\right.,$$
		then $B$ at $\partial\Omega$ is equal to $\mathfrak{J}_1+\mathfrak{J}_2$ plus a compact perturbation. This will be of importance in Sections \ref{Fredholm2D} and \ref{Fredholm3D}, where we compute the index of a given Fredholm operator in which the $BS$ operator is involved. Therefore, since the index does not change under compact perturbations, we only need to find estimates for $\mathfrak{J}_1$ and $\mathfrak{J}_2$.
	\end{obs}
	
	To prove the result, we need a technical lemma regarding the properties of $\mathfrak{J}_1$ and $\mathfrak{J}_2$. Note that, due to the properties of the Newton kernel, c.f. \cite{Gilbarg-Trudinger-2001}, one can readily see that $\mathfrak{J}_1$ and $\mathfrak{J}_2$ are both $C^{2,\alpha}$ functions. Now, if one can consider $\mathfrak{J}_j\cdot n$ on the boundary, with $n$ the normal vector at the boundary, these two new functions are actually more regular than $C^{2,\alpha}(\partial\Omega)$. This can be seen as a consequence of the fact that the integral kernels of $\mathfrak{J}_j$ are the first order approximations of the fundamental solutions for the Dirichlet and the Neumann problem for the Laplace operator, respectively. 
	
	\begin{lema}\label{divcurlconstruction}
		Assume that $j$ is a $C^{1,\alpha}(\overline{\Omega};\R^3)$ vector field. The functions $n\cdot \mathfrak{J}_1$ and $n\cdot \mathfrak{J}_2$, with $\mathfrak{J}_1$ and $\mathfrak{J}_2$ as in \eqref{ansatz}, are $C^{3}$ functions on the boundary $\partial\Omega$.
	\end{lema}
	
	\begin{proof}
		This is a regularity statement, so we just need to check it locally. We may study a parametrization of the boundary given by $\R^{2}\ni x'\mapsto \Gamma(x')=(x',\gamma(x'))$, where $\gamma(x')$ is a suitable smooth function. This is always possible after rotating and translating the boundary, since $\partial\Omega$ is a smooth manifold. We may then assume that $\Omega$ is locally given by $\gamma>0$. 
		
		Let $\mathcal{U}$ be a neighborhood of the boundary $\partial\Omega$ satisfying that the (signed) distance function $x\mapsto d(x,\partial\Omega)$ (i.e., $\rho$ in the parametrization \eqref{feassparametrizaciones}) is a well defined smooth function on $\mathcal{U}$. This neighborhood exists  due to compactness of $\partial \Omega_-$. We can now reduce it further so that $\mathcal{U}\subset \rho^{-1}(-a,a)$ for some $a>0$ small enough. Then, we can define a local parametrization of $\mathcal{U}$ by 
		
		$$\R^{n-1}\times (-a,a)\ni(x',\rho)\mapsto \Psi(x',\rho)=(x',\gamma(x'))+\rho\nu(x'),$$
		where $\nu(x')$ is the inner normal to the boundary, given by 
		
		$$\nu(x')=\frac{(\nabla \gamma(x'),-1)}{\sqrt{|\nabla\gamma(x')|^2+1}}.$$ 
		
		Now, consider our vector field $j$. We know then that $j$ is written as $j(y)=j^\rho(y)\nu(y)+j^{\|}.$ The resulting vector fields $j^\rho\nu$ and $j^\|$ are then also $C^{1,\alpha}$ in a neighborhood of the boundary. Moreover, in this parametrization,  $\nu(y)$ is orthogonal to the vectors
		
		$$E_i(y,\rho)=\frac{\partial \Gamma}{\partial y_i}+\rho\frac{\partial \nu}{\partial y_i}\qquad i\in\{1,2\},$$
		which are the coordinate vector fields tangent to the surfaces $\{\rho=const.\}$, associated to the parametrization $(y,\rho)\mapsto \varphi(y)+\rho\nu(y)$ of $\mathcal{U}$. 
		
		Therefore, we can write $j^{\|}$ simply as 
		
		$$j^{\|}(x',\rho)=j^{\|}_1 E_1(x',\rho)+j^{\|}_2 E_2(x',\rho),$$
		where $j_1^\|$ and $j_2^\|$ are the components of $j^\|$ with respect to the basis $\{E_1,E_2\}$. Now, we study regularity around $x',\rho=0$. We shall focus only on $n\cdot\mathfrak{J}_1$, as the other one is an immediate adaptation of the arguments. Since the only singularity in the integrand defining $\mathfrak{J}_1$ is in the point $x=y$, we can take a smooth cut-off function equal to one in a neighborhood of $x',\rho=0$, and that is supported inside $\mathcal{U}\cap \text{im}\,\Gamma$. Since the remainder will be a smoothing perturbation, we can assume that $j$ is supported in such region. Therefore, we can write
		
		\begin{equation}\label{nucleoentero}
			\begin{split}
				n_x\cdot \mathfrak{J}_1&=\frac{n_x}{4\pi}\cdot \int_{\Omega}\left(\frac{x-y}{|x-y|}+\frac{x-y^\star}{|x-y^\star|}\right)\wedge j^\rho(y)\eta(y)dy\\
				&=\frac{n_x}{4\pi}\cdot\int_{0}^a\int_{U}\left(\frac{\Psi(x',0)-\Psi(y',\rho)}{|\Psi(x',0)-\Psi(y',\rho)|^3}+\frac{\Psi(x',0)-\Psi(y',-\rho)}{|\Psi(x',0)-\Psi(y',-\rho)|^3}\right)\\
				&\qquad \qquad\wedge j^\rho(\Psi(y',\rho))\eta(\Psi(y',\rho))\sqrt{|h(y',\rho)|}dy'd\rho.
			\end{split}
		\end{equation}
		
		Note that the behaviour of the second summand between parenthesis is essentially the same as in the first one. That is, we can approximate the latter by the former up to a term of lower order. More precisely, we can write 
		\small
		\begin{align*}
			&\frac{\Psi(x',0)-\Psi(y',\rho)}{|\Psi(x',0)-\Psi(y',\rho)|^3}+\frac{\Psi(x',0)-\Psi(y',-\rho)}{|\Psi(x',0)-\Psi(y',-\rho)|^3}\\
			&=\frac{1}{|\Psi(x',0)-\Psi(y',\rho)|^3}\Bigg(\Psi(x',0)-\Psi(y',\rho)\left.+\left(\Psi(x',0)-\Psi(y',-\rho)\right)\frac{|\Psi(x',0)-\Psi(y',\rho)|^3}{|\Psi(x',0)-\Psi(y',-\rho)|^3}\right)
		\end{align*}
		\normalsize
		We claim that this expression can be approximated by 
		
		\begin{equation}\label{taylor}
			\frac{1}{|\Psi(x',0)-\Psi(y',\rho)|^3}\left(2(\Psi(x',0)-\Psi(y',0))+H\left(x',y',\rho,\frac{(x'-y',\rho)}{|(x'-y',\rho)|}\right)\right),
		\end{equation}
		where $H$ is a smooth function which equals zero at $y'=x',\, \rho=0$. 
		
		We begin by studying the following quotient:
		\begin{equation}\label{quotient}
			\frac{|\Psi(x',0)-\Psi(y',\rho)|^2}{|\Psi(x',0)-\Psi(y',-\rho)|^2}.
		\end{equation}
		
		Our claim is that it is equal to a smooth function of $x',y',\rho$ that equals one at $x'=y',\rho=0$. This can be proved by means of an elementary observation: notice that the numerator and the denominator of this function vanish both only at $x'=y',\rho=0$. Furthermore, one can prove that the limit of the quotient \eqref{quotient} as $(y',\rho)\rightarrow (x',0)$ is one. Indeed, one can use for instance Taylor's theorem, so since the function 
		
		$$|\Psi(x',0)-\Psi(y',\rho)|^2$$ 
		has a second order zero at $y'=x',\rho=0$, it can be written as 
		
		$$|\Psi(x',0)-\Psi(y',\rho)|^2=(y'-x',\rho)R(x',y',\rho))(y'-x',\rho)^t,$$
		where $R\in C^\infty(U\times U\times(0,a);\R^{n\times n})$. Moreover, at $y'=x'$ and $\rho=0$, it is equal to the Hessian:
		
		\begin{equation}\label{matriz}
			\left(\begin{array}{ccc}
			(1+(\partial _1 \gamma(x'))^2) &  \partial_1 \gamma(x')\partial_2 \gamma(x')  & 0\\
			\partial_1 \gamma(x')\partial_2 \gamma(x')  & (1+(\partial _2 \gamma(x'))^2) & 0 \\
			0						& 0				& 1
		\end{array}\right).
		\end{equation}
		
		A similar argument can be made with the function \eqref{quotient}, with the crucial observation that the Hessians at $x'=y'$ and $\rho=0$ of numerator and denominator coincide. This means that, since this matrix \eqref{matriz} is positive definite, there is no singularity of \eqref{quotient} around $x'=y'$, $\rho=0$, so in a neighborhood of $(x',0)$ the quotient \eqref{quotient} can be written as 
		
		$$\frac{|\Psi(x',0)-\Psi(y',\rho)|^2}{|\Psi(x',0)-\Psi(y',-\rho)|^2}=F\left(x',y',\rho, \frac{(x'-y',\rho)}{|(x'-y',\rho)|}\right),$$
		where $F\in C^\infty(U\times U\times (0,a)\times\S^2).$ Due to the observation we made on the Hessians above, $F(x',y',0,\omega)=1$ for every $\omega\in\mathbb{S}^2$. Furthermore, it is homogeneous of order zero in the $\omega$ variable. We can use this to prove  the regularity estimates for the kernel in \eqref{nucleoentero}:
		
		\begin{equation*}
			\begin{split}
				\frac{\Psi(x',0)-\Psi(y',\rho)}{|\Psi(x',0)-\Psi(y',\rho)|^3}&+\frac{\Psi(x',0)-\Psi(y',-\rho)}{|\Psi(x',0)-\Psi(y',-\rho)|^3}\\
				&=\frac{2(\Psi(x',0)-\Psi(y',0))}{|\Psi(x',0)-\Psi(y',\rho)|^3}+\frac{O(\sqrt{|x'-y'|^2+\rho^2})}{|\Psi(x',\rho)-\Psi(y',\rho)|^3}.
			\end{split}
		\end{equation*}
		
		We shall work now with only the first term on the right hand side, as the second one leads to more regular terms. Using that $a\cdot (b\wedge c)=-b\cdot (a\wedge c)$, we find that the expression \eqref{nucleoentero} can be written, up to the addition  of a term with a lower order singularity, as 
		
		\begin{equation}\label{purokernel}
			-\frac{1}{4\pi}\int_{\R^2}\int_0^a  2\frac{\nu(x')\wedge (\Psi(x',0)-\Psi(y',0))}{|\Psi(x',0)-\Psi(y',\rho)|^3}\cdot (\nu(y'))j^{\rho}(y',\rho)\sqrt{1+|\gamma(y')|^2}\eta(y',\eta) d\rho dy',
		\end{equation}
		where we have written $j^{\rho}(\Psi(y',\rho))$ as $j^\rho(y',\rho)$ for simplicity. The function $\eta$ denotes a compactly supported function in $\R^2\times[0,\infty)$.
		
		We begin by making an observation on the denominator. Using the previous reasoning we can readily check that 
		
		$$\frac{1}{|\Psi(x',0)-\Psi(y',\rho)|^2}=\frac{1}{{|x'-y'|^2+\rho^2}}h\left(x',y',\rho,\frac{(x'-y',\rho)}{|(x'-y',\rho)|}\right),$$
		where $h$ is a smooth function of all of its variables. Furthermore, knowing the precise value of $h$ at $x'=y', \rho=0$, we can estimate the term with the denominator as 
		
		$$\frac{1}{|\Psi(x',0)-\Psi(y',\rho)|^3}=\frac{1}{\left(G(x')_{ij}(x'_{i}-y'_j)(x'_{i}-y'_j)+\rho^2\right)^{3/2}}+O\left(\left(|x'-y'|^2+\rho^2\right)^{-1/2}\right). $$
		
		Note that the convention of summing over repeated indices is assumed. The second term will be a lower order perturbation, so we we will focus on the first one. Note that $\Psi(x',0)-\Psi(y',0)=\partial_i \Psi(x',0)(x'_i-y'_i)+O(|x'-y'|^2)$. On the other hand, the vector $\nu(y')$ behaves like $\nu(y')=\nu(x')+\partial_i\nu(x')(x'-y')+O(|x'-y'|^2).$
		
		Now, due to the properties of the mixed product, we conclude that the leading order term in the kernel defining \eqref{nucleoentero} is
		
		$$\left[\left(\nu(x')\wedge \partial_i \Psi(x',0)\right)\cdot \partial_j\nu(x')\right]\frac{(x'_i-y'_j)(x'_j-y'_j)}{\left(G(x')_{ij}(x'_{i}-y'_i)(x'_{i}-y'_j)+\rho^2\right)^{3/2}}$$
		
		The term in brackets is just a smooth function that depends on $x'$, so it doesn't affect to the regularity estimates. Therefore, we can disregard it. Now, the kernel can be bounded by $(|x'-y'|^2+\rho^2)^{-1/2}$, which is locally integrable in $\R^3$. Since the remaining terms of lower order that we have not written have a milder singularity, we conclude that the operator given by \eqref{nucleoentero}   is well defined on the surface. Moreover, the kernel above can be written as $K(x',x'-y',\rho)$, with $K$ satisfying
		
		\begin{equation}\label{decaykernel}
			\left|\partial_{x'}\partial^{\alpha}_z K(x',z,\rho)\right|\leq C \left(|z|+\rho\right)^{-1-|\alpha|}.
		\end{equation}
		
		The resulting kernel then has locally integrable derivatives as well. Therefore, we obtain that the resulting operator defined by \eqref{nucleoentero} takes $C^0$ functions and maps them continuously into $C^1$ functions. The problem arise when we try to take an extra derivative, i.e. when we try to prove that it actually maps $C^0$ to $C^2$, due to the fact that $(|z|+\rho)^{-3}$ fails to be integrable in $\R^3$ with respect to the variables $(z,\rho)$. However, one can prove that the resulting function is indeed $C^2$ as long as we have some further regularity a priori. The key step is realizing that the support of our $\eta$ is compact and only depends on the boundary. Now, we know that the first derivative of our kernel reads 
		
		$$\int_{\R^3_+}\left((\partial_{x'}K)(x',x'-y',\rho)+(\partial_{z}K)(x',x'-y',\rho)\right)\eta(y',\rho)j^\rho(x',\rho)\sqrt{1+|\nabla\gamma(y')|^2}dy'd\rho.$$
		
		The derivative of the first term is again locally integrable due to \eqref{decaykernel}. On the other side, the second term is the most singular one, so that is why we will focus on it. If we take the incremental quotients of this quantity we find that, if we denote by $f(y',\rho)=j^\rho(y',\rho)\sqrt{1+|\nabla\gamma(y')|^2}$,
		\small
		\begin{equation}\label{chundachunda}
			\begin{split}
				&\phantom{= }\int_{\R^3_+}(\partial_{z}K)(x'+h,x'+h-y',\rho)f(y',\rho)\eta(y',\rho)dy'd\rho-\int_{\R^3_+}(\partial_{z}K)(x',x'-y',\rho)f(y',\rho)\eta(y',\rho)dy'd\rho\\
				&=\int_{\R^3_+}\left((\partial_{z}K)(x'+h,x'+h-y',\rho)-(\partial_{z}K)(x'+h,x'-y',\rho)\right)f(y',\rho)\eta(y',\rho)dy'd\rho\\
				&\quad +\int_{\R^3_+}\left((\partial_{z}K)(x',x'+h-y',\rho)-(\partial_{z}K)(x',x'-y',\rho)\right)\eta(y',\rho)f(y',\rho)dy'd\rho.
			\end{split}
		\end{equation}
		\normalsize
		The integrand in the first summand converges pointwise to $\partial_x\partial_z K(x',x'-z',\rho)$ as $h\rightarrow 0^+$. This is again a locally integrable quantity, so this term is well defined as an absolutely convergent integral. It remains to estimate the second term on the right hand side of the equality in \eqref{chundachunda}. 
		
		\begin{equation}
			\begin{split}
				&\int_{\R^3_+}\left((\partial_{z}K)(x',x'+h-y',\rho)-(\partial_{z}K)(x',x'-y',\rho)\right)\zeta(y')f(y',\rho)dy'd\rho\\
				&=\int_{\R^3_+}\left((\partial_{z}K)(x',x'+h-y',\rho)-(\partial_{z}K)(x',x'-y',\rho)\right)\eta(y',\rho)\left(f(y',\rho)-f(x',\rho)\right)dy'd\rho\\
				&\quad +f(x')\int_{\R^3_+}\left((\partial_{z}K)(x',x'+h-y',\rho)-(\partial_{z}K)(x',x'-y',\rho)\right)\eta(y',\rho)dy'd\rho\\
				&=\int_{\R^3_+}\left((\partial_{z}K)(x',x'+h-y',\rho)-(\partial_{z}K)(x',x'-y',\rho)\right)\eta(y',\rho)\left(f(y',\rho)-f(x',\rho)\right)dy'd\rho\\
				&\quad +f(x')\int_{\R^3_+}(\partial_{z}K)(x',x'-y',\rho)\left(\eta(y'-h,\rho)-\eta(y',\rho)\right)dy'd\rho.=I_1+I_2
			\end{split}
		\end{equation}
		
		If we now divide by $|h|$ and take the limit when $h\rightarrow 0$, we find that the limit of $I_2$ is well defined as the absolutely convergent integral 
		
		$$-f(x')\int_{\R^3_+}(\partial_{z}K)(x',x'-y',\rho)\nabla_{y'}\eta(y',\rho)dy'd\rho.$$
		
		On the other hand, the integrand in $I_1$ tends to
		
		$$\partial_z^2K(x',x'-y',\rho))\eta(y')(f(y',\rho)-f(x',\rho)).$$
		
		This is again absolutely integrable $f$ is taken to be $C^\alpha$. Therefore, we have just proved that if $j$ is $C^\alpha$, then the function defined as 
		
		$$Tf(x')=\int_{\R^3_{+}}K(x',x'-y',\rho)\eta(y')j^\rho(y')dy'$$
		is twice differentiable. Furthermore, we can give a formula for the derivatives as 
		
		$$D_i Tf(x')=\int_{\R^3_+} \left((\nabla_{x'}K)_i(x',x'-y',\rho)+(\nabla_zK)_i(x',x'-y',\rho)\right)\eta(y')j^\rho(y')dy'd\rho.$$
		
		On the other hand, the second derivatives of $Tf$ read 
		
		\begin{equation}
			\begin{split}
				D_jD_kTj^\rho(x')&=\int_{\R^3_+} (\nabla^2_{x'}K)_{kj}(x',x'-y',\rho)\eta(y')j^\rho(y')dy'd\rho\\
				&+2\int_{\R^3}(\nabla_z\nabla_{x'}K)_{kj}(x',x'-y',\rho)\eta(y')j^\rho(y')dy'd\rho\\
				&+\int_{\R^3_+}(\nabla^2_{z}K)_{kj}(x',x'-y',\rho)\eta(y')\left(j^{\rho}(y',\rho)-j^\rho(x',\rho)\right)dy'd\rho\\
				&+j^\rho(x')\int_{\R^3_+}(\nabla_{z}K)_k(x',x'-y',\rho)(\nabla\eta)_ j(y')dy'd\rho
			\end{split}
		\end{equation}
		One readily checks that the derivatives $D_kD_j Tj^\rho$ are continuous. The result is proven. 
	\end{proof}

	The proof of Theorem \ref{teor:biotsavart} is then an easy consequence of Lemma \ref{divcurlconstruction}.
	
	\begin{proof}[Proof of Theorem \ref{teor:biotsavart}]
		The idea  consists on extending the current $j$ suitably to the whole space and find a solution to the div-curl system in the sense of distributions. The issue now is that the extension must remains divergence free. Take a smooth cutoff $\theta\in C^{\infty}_c(\R^3)$ so that $\theta\equiv 1$ in neighborhood of $\overline{\Omega}$, and such that it is supported inside of $\Omega\cup \mathcal{U}$. Now, we can extend $j$ to a vector field in $\tilde{\jmath}$ in  $\Omega\cup \mathcal{U}$ by taking, in a system of coordinates $(x',\rho)$,
		
		$$\tilde{\jmath}^\|(x,\rho)=-\frac{\sqrt{h(x',-\rho)}}{\sqrt{h(x',\rho)}}j^\|(x',-\rho)\qquad \tilde{\jmath}^\rho (x,\rho)=\frac{\sqrt{h(x',-\rho)}}{\sqrt{h(x',\rho)}}j(x',-\rho).$$
		
		Note that this leads to a well defined extension of $j$ on $\mathcal{U}\setminus \Omega$, as the quotient of the determinants is well defined as a function. It is enough to check the behaviour under changes of coordinates. In other words, we extend the tangential component of $j$ oddly and the normal component evenly along the surface that defines the boundary, with a correction imposed to ensure that the divergence free condition is preserved. Using Lemma \ref{divergencia}, one checks readily that this extension is divergence free in the sense of distributions, when restricted to $\Omega\cup\mathcal{U}$, i.e., for any $\varphi\in C^\infty(\Omega\cup\mathcal{U})$, 
		
		$$\int \nabla\varphi(x)\tilde{\jmath}(x)\,dx=0.$$ 
		
		Now, we can extend $\tilde{\jmath}$ to all of the space by multiplying it with the cutoff $\theta$. However, the resulting current $\theta \tilde{\jmath}$ will not be divergence free. Its divergence, though, will be compactly supported in the region where $\theta$ is non constant. More precisely, it will be given by 
		
		$$H:=\text{div}(\tilde{\jmath}\theta)=\tilde{\jmath}\,\nabla\theta.$$
		
		Therefore, 
		
		$$\text{supp}\,H\subseteq \{x\in \R^d\,:\, \nabla\theta\neq 0\}.$$
		
		Since $\theta$ is a cutoff function, we conclude that the support of $H$ is contained in a compact region of $\R^3$ that is at a positive distance of $\Omega$. If $j\in C^{1,\alpha}(\Omega)$, then this divergence is a compactly supported $C^{1,\alpha}$ function. Furthermore, the following compatibility conditions are satisfied:   
		
		$$\int_{\text{supp} H}H(x)dx=\int_{\mathcal{U}\setminus \Omega}\text{div}(\tilde{\jmath}\theta)\,dx=-\int_{\partial\Omega}j\cdot n=0.$$
		
		Actually, as 
		
		$$\int_{\Gamma_j}j\cdot n=0$$ 
		for every connected component $\Gamma_j$ of the boundary $\partial\Omega$, we can solve the problem 
		
		\begin{equation}\label{eq:divergencia}
			\left\{\begin{array}{ll}
				\nabla\cdot G=-H & \text{in } U\\
				G=0 & \text{on }\partial U 
			\end{array}\right.,
		\end{equation}
		for $U$ a smooth bounded domain that contains $\text{supp}\,H$, and that remains away from $\Omega$.This problem is well studied, and it is known that for any such domain $U$, and for $H\in L^p(U)$ with zero integral and with $p\in (1,\infty)$ , there exists a vector field $G\in (W^{1,p}_0(U))^3$ satisfying \eqref{eq:divergencia}. This result can be found in \cite[Theorem 2.6]{divergencia}. The main conclusion is then that the distribution in $\R^3$ given by 
		
		$$\theta\,\tilde{\jmath}+G$$
		is divergence free by construction, and is compactly supported. Therefore, we can construct a solution to the div-curl problem in the whole space:
		
		$$\left\{\begin{array}{ll}
			\nabla\times B=j & \text{in }\R^3\\
			\nabla \cdot B=0 & \text{in }\R^3
		\end{array}\right..$$
		
		Indeed, we can take $\text{curl}\,A,$ where $A$ is the component-wise convolution of $\theta\,\tilde{\jmath}+G$ with the Newton kernel, $1/|x|$. This leads to a solution of the problem above in the whole space in the sense of distributions, as one can check easily by means of the definition of the distributional divergence and curl. This is well defined because, even though the Newton kernel does not map $\mathcal{D}(\R^3)$ into itself, the current we have constructed is compactly supported.  Note that, inside the domain $\Omega$, the distribution $\text{curl}\,A$ coincides with the function given by 
		
		$$\int_{\Omega\cup \mathcal{U}}\frac{x-y}{|x-y|^3}\wedge\tilde{\jmath}(y)dy+\int_{\text{supp} H}\frac{x-y}{|x-y|^3}\wedge G(y)dy.$$
		
		These integrals are well defined as $\tilde{\jmath}$ is a locally bounded function, and the support of $G$ is at a positive distance of $\Omega$. Actually, this implies that the second term is a smooth function in $\Omega$. We can decompose the first term above into two summands
		
		$$\int_{\Omega}\frac{x-y}{|x-y|^3}\wedge j(y)(1-\eta(y))+\int_{\mathcal{U}}\frac{x-y}{|x-y|^3}\wedge\tilde{\jmath}(y)\eta(y),$$ 
		where $\eta$ is yet another cutoff function, which is equal to 1 in a neighborhood of $\partial\Omega$. After a change of variables, we can write the integral above as
		
		\begin{multline}
			\frac{1}{4\pi}\int_\Omega \left(\frac{x-y}{|x-y|^3}+\frac{x-y^\star}{|x-y^\star|^3}\right)\wedge j^\rho(y)\eta(y)dy \\+\frac{1}{4\pi}\int_\Omega \left(\frac{x-y}{|x-y|^3}-\frac{x-y^\star}{|x-y^\star|^3}\right)\wedge j^{\|}(y)\eta(y)dy =\mathfrak{J}_1+\mathfrak{J}_2,
		\end{multline}
		where $y^\star$ corresponds to the reflexion of the point $y$ along the tangent space of $\partial\Omega_-$, i.e. if $y$ is given by $(\varphi(y')+\rho\nu(y'))$ for some parametrization of the surface $\partial\Omega_-$, then $y^\star=\varphi(y')-\rho\nu$. 
		
		We have then constructed a function in $\Omega$ which is divergence free and its curl equals $j$. Note the similarity of the terms above with the terms obtained by applying the method of images to the Dirichlet and Neumann problems, respectively, of the Poisson's equation. It only remains to show that it has the correct regularity, as well as the correct boundary conditions. It is obvious, by the standard properties of the Newton's kernel, that this newly constructed function is $C^{2,\alpha}$ in $\overline{\Omega}$ if $j\in C^{1,\alpha}$ (c.f. \cite{Gilbarg-Trudinger-2001}). Then, we can finish our construction of the magnetic field $B$ as 
		
		$$B:=\text{curl}A+\nabla\phi,$$
		where $\phi$ solves the Neumann problem 
		
		\begin{equation} \label{Neumann} 
			\left\{ 
			\begin{array}{ll}
				\Delta \phi=0&\text{in }\Omega\\
				n\cdot\nabla \phi=-\text{curl}A\cdot n & \text{on }\partial\Omega
			\end{array}
			\right..
		\end{equation}
		Notice that the newly constructed magnetic field $B$ is $C^{2,\alpha}$ up to the boundary due to the basic properties of the Newton kernel, (c.f. \cite{Gilbarg-Trudinger-2001}).
		This regularity is optimal due to the fact that the current $j$ in in $C^{1,\alpha}$.  Now, due to the result in Lemma \ref{divcurlconstruction}, the boundary condition in \eqref{Neumann} is in $C^3$. Therefore, it leads to a solution $\phi\in C^{2,\tilde{\alpha}}$. The result is proven. 
	\end{proof}
	\section{The two dimensional case}\label{2D}

	Once we have the conditions that allow one to solve equation \eqref{integraleqn}, we can state the main perturbation argument that will let us solve the problem. We will prove the following results
	
	\begin{prop}
		Let $B_1$ and $B_2$ be two $C^{2,\alpha}$ vector fields satisfying the assumptions of Section \ref{asunciones1}. Then, 
		
		$$\|A[B_1]-A[B_2]\|_{\mathcal{L}(C^{1,\alpha},C^{2,\alpha})}\leq \frac{\max(\|B_1\|_{C^2},\|B_2\|_{C^{2}})}{\min_{i=1,2}\left(\min_{\omega\in \partial\Omega_-}|B_i(\omega)\cdot n(\omega)|\right)}\|B_1-B_2\|_{C^{2,\alpha}},$$
		where $C$ is bounded by an increasing function of $\|B_1\|_{C^{2,\alpha}}+\|B_2\|_{C^{2,\alpha}}.$
	\end{prop}
	
	It is then easy to see that, if the proposition above holds, and we have an \textit{a priori} magnetic field $B_0$ such that $A[B_0]$ is invertible and with \eqref{condicion} non-vanishing, we can solve equation \eqref{integraleqn} for any $B$ magnetic field close to $B_0$
	
	\begin{prop}
		Assume that $\Omega$ and $B_0\in C^{2,\alpha}(\Omega,\R^2)$ satisfies assumptions in Section \eqref{asunciones1}. Assume further that $A[B_0]$ is invertible and that the integral \eqref{condicion} does not vanish. Then, there exists $\varepsilon>0$ such that, for any magnetic field $B$ satisfying $\|B-B_0\|_{C^{2,\alpha}}\leq \varepsilon$, the operator $A[B]$ is invertible and the quantity in \eqref{condicion} does not vanish.
	\end{prop}
	\begin{proof} 
		
		The proof of this theorem consists on a repeated use of the Neumann's series. Assume that we have an initial magnetic field $B_0$ satisfying that its corresponding operator $A$ is invertible, and that the integral \eqref{condicion} does not vanish. Then, we can solve the equation corresponding to any other magnetic field $B$ which is close enough to $B_0$. Indeed, the new equation reads 
		
		$$g=\lambda B_{mono,\tau}+n\cdot \nabla v|_{\partial\Omega_-}+Jn\cdot \nabla \phi|_{\partial\Omega_-}+A[B]j_0.$$
		
		Here, the only dependence of the new magnetic field $B$ is in the operator $A$, as the rest depend all on the boundary conditions $f$ and $g$. In such case, we can write it as 
		
		$$j_0+A[B]^{-1}(A[B]-A[B_0])j_0=A^{-1}[B_0]\left(g-\lambda B_{mono}-n\cdot\nabla v|_{\partial\Omega_-}-Jn\cdot \nabla\phi|_{\partial\Omega_-}\right).$$ 
		
		As a result, we infer automatically that, if we have a continuity property,
		
		$$\|A[B_0]-A[B]\|_{\mathcal{L}(C^{1,\alpha},C^{2,\alpha})}\leq C\|B_0-B\|_{C^{2,\alpha}},$$
		we can prove via a Neumann's series argument that the equation is solvable if $\|B_0-B\|_{C^{2,\alpha}}$ is small enough. Moreover, we can still give a definition for $J$ that keeps $ p$ univalued. Indeed, as in this case 
		
		$$j_0=(Id+A[B_0]^{-1}(A[B]-A[B_0]))^{-1}\left(A^{-1}[B_0]\left(g-\lambda B_{mono}-n\cdot \nabla v-Jn\cdot\nabla \phi\right)\right),$$
		\eqref{condicion} does not vanish anywhere if
		
		$$\int f(Id+A[B_0]^{-1}(A[B]-A[B_0]))^{-1}\left(A^{-1}[B_0](n\cdot \nabla\phi|_{\partial\Omega_-})\right)\neq 0$$
		
		Now, using the Neumann series, we find that 
		
		\begin{equation}\label{condicion2}
			\begin{split}
				\int f(Id+A[B_0]^{-1}(A[B]&-A[B_0]))^{-1}\left(A^{-1}[B_0](n\cdot \nabla\phi|_{\partial\Omega_-})\right)=\int f(\left(A^{-1}[B_0](n\cdot \nabla\phi|_{\partial\Omega_-})\right)\\
				&+\sum_{k=1}^\infty \int f(A[B_0]^{-1}(A[B]-A[B_0]))^k\left(A^{-1}[B_0](n\cdot \nabla\phi|_{\partial\Omega_-})\right)
			\end{split}
		\end{equation}
		
		We know that the first term is nonzero. On the other hand, the second term on the right hand side is bounded by 
		
		$$C\|B_0-B\|_{C^{2,\alpha}}e^{\|A[B_0]^{-1}\|_{op}\|A[B]-A[B_0]\|_{op}}\|f\|_{\infty}\|A^{-1}[B_0](n\cdot\nabla\phi|_{\partial\Omega_-})\|_{\infty},$$
		where $\|\cdot\|_{op}$ denotes the operator norm. Thus, taking $\|B_0-B\|$ small enough, we conclude that the integral in \eqref{condicion2} does not vanish.
	\end{proof}
	The following subsection \ref{sec:opA} will be devoted to study the properties of the operator $A$.

	\subsection{The operator $A$}\label{sec:opA}
	
	The main results of this subsection consist on proving that $A[B]$ is bounded as an operator from $C^{1,\alpha}(\partial\Omega_-)$ to $C^{2,\alpha}(\partial\Omega_-)$, as well as precise estimates on $\|A[B]-A[B']\|_{\mathcal{L}(C^{1,\alpha},C^{2,\alpha})}$. The statements read 
	
	\begin{teor}\label{mainteor2D2D}
		Let $\Omega$ and $B$ satisfy the assumptions in Section \ref{asunciones1}. Then, the operator $A[B]$ introduced in Definition \ref{defia} satisfies the following bound 
		
		$$\|A[B]\|_{\mathcal{L}(C^{1,\alpha},C^{2,\alpha})}\leq C\frac{\|B\cdot n\|_{C^{2,\alpha}(\partial\Omega)}}{\min_{\omega\in \partial\Omega_-}|B(\omega)\cdot n(\omega)|}\|B\|_{C^{2,\alpha}},$$
		Furthermore, if we have $B_1$ and $B_2$ two magnetic fields satisfying the assumptions in Section \ref{asunciones1}, we find that
		
		\small
		\begin{align*}
			\|A[B_1]-A[B_2]&\|_{\mathcal{L}(C^{1,\alpha},C^{2,\alpha})}\\
			&\leq C\frac{\left(\|(B_1-B_2)\cdot n\|_{C^{2,\alpha}}\|B_1\|_{C^{2,\alpha}}+\|B_1\cdot n\|_{C^{2,\alpha}}\|(B_1-B_2)\|_{C^{2,\alpha}}\right)}{(1-\sup_{\partial\Omega_-}|\cos(\theta)|)\inf_{\omega_{\partial\Omega_-}}|B_1(\omega)|}\\
			&+C\frac{\|B_1\|_{C^{2,\alpha}}\|B_1\cdot n\|_{C^{2,\alpha}}(\|B_1\|_{C^{2,\alpha}}+\|B_2\|_{C^{2,\alpha}} )(\|B_1-B_2\|_{C^{2,\alpha}})}{(1-\sup_{\partial\Omega_-}|\cos(\theta_1)|)(1-\sup_{\partial\Omega_-}|\cos(\theta_2)|)\inf_{\omega_{\partial\Omega_-}}|B_1(\omega)|\inf_{\omega_{\partial\Omega_-}}|B_2(\omega)|}
		\end{align*} 
		\normalsize
	\end{teor}
	
	\begin{obs}
		The boundedness of the operator $A[B]$ is well known just by looking at the building blocks in its definition (c.f. Definition \ref{defia}). What is important about this theorem is the precise dependence of $\|A[B]\|_{\mathcal{L}(C^{1,\alpha},C^{2,\alpha})}$ in terms of $B$, as well of the estimates of $\|A[B_1]-A[B_2]\|_{\mathcal{L}(C^{1,\alpha},C^{2,\alpha})}$. Note that this last statement does not follow immediately out of the building blocks of the definition of $A[B]$. The reason being is that, if we take two solutions $j^\ell$ of the transport system 
		
		$$B^\ell\cdot \nabla j^\ell =0,\qquad \ell\in \{1,2\},$$
		with the same boundary condition $j^\ell|_{\partial\Omega_-}$, and if $B\in C^{2,\alpha}$, the usual regularity theory for transport equations just leads to estimates of the form 
		
		$$\|j^1-j^2\|_{C^{0,\alpha}}\leq C\|j^1_0\|_{C^{1,\alpha}}\|B^1-B^2\|_{C^{2,\alpha}},$$
		i.e. the estimates we obtain must be taken in a space of lower regularity. That is why a finer study must be carried out to derive estimates on the norm in $\mathcal{L}(C^{1,\alpha},C^{2,\alpha}).$ 
	\end{obs}
	
	In order to prove this theorem, we study each of the building blocks that lead to the construction of the operator. More precisely, we have the following lemma: 
	
	\begin{lema}\label{aaaa} The operator $A$ introduced in Definition \ref{defia} is given by the following integral operator 
		\begin{equation}\label{formula}
			Aj_0 (\omega)=\int_{\Omega}K(z,\omega)j(\Phi^{-1}(z))dz=\int_{\partial\Omega_-}j_0(\omega')\int_{0}^{L(\omega')}K(\Phi(\omega',s),\omega)|J(\omega',s)|dsd\omega',
		\end{equation}
		where $K$ is the Poisson kernel of the domain $\Omega$, $\Phi(\omega',s)$ is the characteristic line of $B$ given by the equation
		
		\begin{equation}\label{characteristic}
			\left\{\begin{array}{ll}
				\frac{\partial\Phi}{\partial s}(\omega,s)=B(\Phi(\omega,s)) & \omega\in\partial\Omega_-\\
				\Phi(\omega,0)=\omega.
			\end{array}
			\right.,
		\end{equation}
		$L(\omega)$ is the maximal time of existence of this ODE, and $J(\omega',s)$ is the Jacobian of the change of variables given by $\Phi$.
	\end{lema}
	\begin{obs}
		The maximal time $L(\omega)$ corresponds to the time taken by the integral curve of $B$ starting in $\omega\in \partial\Omega_-$ to reach the other component of the boundary $\partial\Omega_+$.
	\end{obs}
	\begin{proof}[Proof of Lemma \ref{aaaa}]
		First of all, the current $j$ must satisfy the transport equation 
		
		$$(B\cdot \nabla)j=0$$ 
		with boundary data $j|_{\partial\Omega_-}=j_0$. We can then use the method of characteristics, so that $j_0$ is given by 
		
		$$j(x)=j_0\circ \Phi^{-1}(x),$$
		where the function $\Phi$ is a diffeomorphism $\Phi:U\longrightarrow \Omega$, with 
		
		$$U:=\{(\omega,s)\in  \partial\Omega_-\times \R,\, 0\leq s\leq L(\omega)\},$$
		where $\Phi$ is defined via the flow map 
		
		$$\left\{\begin{array}{ll}
			\frac{\partial\Phi}{\partial s}(\omega,s)=B(\Phi(\omega,s)) & \omega\in\partial\Omega_-\\
			\Phi(\omega,0)=\omega.
		\end{array}
		\right..$$
		It is easy to see that, as $B$ is never vanishing, and due to the classical existence and uniqueness theorems for ODEs, the characteristics of $B$ do not cross, and fill all of $\Omega$. Therefore, the map $\Phi$ defined above is a well defined diffeomorphism.
		
		Now, recall that $Aj_0$ was given by $n\cdot \nabla \varphi|_{\partial\Omega_-}$, where $\varphi$ satisfies \eqref{scheme}. Then, we can use the Green function of the domain $\Omega$, that we denote by $G(x,z)$ so that the function $\varphi$ reads 
		
		$$\varphi(x)=\int_{\Omega}G(x,z)j_0(\Phi^{-1}(z))dz,$$
		where $G$ satisfies 
		
		$$\left\{\begin{array}{ll}
			\Delta_z G(x,z)=\delta(x-z) & (x,z)\text{ in }\Omega\times \Omega\\
			G(x,z)=0 & z\,\text{on }\partial\Omega
		\end{array}\right..$$
		
		Finally, to study $Aj_0$ we have to take $\lim_{x\rightarrow \omega}n_x\cdot \nabla\varphi(x)$ for each $\omega\in \partial\Omega_-$. At this point we recall the definition of the Poisson kernel. This is defined as $n_y\cdot \nabla_y G(x,y)|_{\partial\Omega}$. Therefore, due to the symmetry of the Green's function we can write the operator $A$ as in \eqref{formula}, so the proof is finished.
	\end{proof}
	
	The expression for $A$ can be simplified a bit further, due to the fact that $B$ is divergence free. This is a consequence of the following lemma
	\begin{lema}
		The determinant $J(\omega,s)$ in the expression \eqref{formula} is independent of $s$.
	\end{lema}
	\begin{proof}
		To check it, we recall that the expression of $J(\omega,s)$ in terms of a parametrization $\gamma(t)$ of the curve $\partial\Omega_-$ is
		
		$$b_1(\Phi(\omega',s))\frac{\partial (\Phi_2(\gamma(t),s))}{\partial t}-b_2(\Phi(\omega',s))\frac{\partial (\Phi_1(\gamma(t),s))}{\partial t}$$

		We now compute the derivative with respect to $s$. $\frac{\partial J_s}{\partial s}=I+II+III+IV$ given by (we assume that every time that $B$ appears, it is evaluated in the point $\Phi(\gamma(t),s)$)
		
		\begin{align*}
			I=\frac{\partial b_1}{\partial s}\frac{\partial (\Phi_2(\gamma(t),s))}{\partial t}&=b_1\frac{\partial b_1}{\partial x}\frac{\partial (\Phi_2(\gamma(t),s))}{\partial t}+b_2\frac{\partial b_1}{\partial y}\frac{\partial (\Phi_2(\gamma(t),s))}{\partial t}\\
			&=-b_1\frac{\partial b_2}{\partial y}\frac{\partial (\Phi_2(\gamma(t),s))}{\partial t}+b_2\frac{\partial b_1}{\partial y}\frac{\partial (\Phi_2(\gamma(t),s))}{\partial t},
		\end{align*}
		
		\begin{align*}
			II=b_1\frac{\partial^2 (\Phi_2(\gamma(t),s))}{\partial s\partial t}
			&=b_1\frac{\partial}{\partial t}b_2\\
			&=b_1\frac{\partial b_2}{\partial x}\frac{\partial \Phi_1(\gamma(t),s))}{\partial t}+b_1\frac{\partial b_2}{\partial y}\frac{\partial \Phi_2(\gamma(t),s))}{\partial t},
		\end{align*}
		
		\begin{align*}
			III=-\frac{\partial b_2}{\partial s}\frac{\partial (\Phi_1(\gamma(t),s))}{\partial t}&=-b_1\frac{\partial b_2}{\partial x}\frac{\partial (\Phi_1(\gamma(t),s))}{\partial t}-b_2\frac{\partial b_2}{\partial y}\frac{\partial (\Phi_1(\gamma(t),s))}{\partial t}\\
			&=-b_1\frac{\partial b_2}{\partial x}\frac{\partial (\Phi_1(\gamma(t),s))}{\partial t}+b_2\frac{\partial b_1}{\partial x}\frac{\partial (\Phi_1(\gamma(t),s))}{\partial t},
		\end{align*}
		
		\begin{align*}
			IV=-b_2\frac{\partial^2 (\Phi_1(\gamma(t),s))}{\partial s\partial t}
			&=-b_2\frac{\partial}{\partial t}b_1\\
			&=-b_2\frac{\partial b_1}{\partial x}\frac{\partial \Phi_1(\gamma(t),s))}{\partial t}-b_2\frac{\partial b_1}{\partial y}\frac{\partial \Phi_1(\gamma(t),s))}{\partial t},
		\end{align*}
		where we have used the divergence free condition on $B$ as well as the chain rule. It is then obvious that the sum of all four terms equals zero. Therefore, the Jacobian is independent of $s$, so $J(\gamma(t),s)=J(\gamma(t),0)$. At such point, the function $\Phi$ is the identity, so 
		
		$$\frac{\partial\Phi_{a}(\gamma(t),0)}{\partial t}=\gamma_a'(t)\quad \text{for }a\in\{1,2\}.$$
		
		As a result, we can compute the value of the jacobian at every point in $\Omega$, just by noticing that $\gamma'(t)$ equals a tangent vector given by the parametrization of $\partial\Omega_-$ we have chosen. This is 
		
		\begin{equation} \label{eq:Bnprod}
			J(\gamma(t),s)=B(\gamma(t))\cdot \textsf{n}(\gamma(t)).
		\end{equation}
	\end{proof}
	
	We then obtain that the expression for $A$ is simplified, as $J(\omega,s)$ has no dependence on $s$.  We know that this is a map from $C^{1,\alpha}$ to $C^{2,\alpha}$, just by analyzing the different building blocks of the operator (see the steps of Definition \ref{defia}). We can go even further, and prove precise estimates on the norm of $A$ in terms on the $C^{1,\alpha}$ norm of the magnetic field $B$. We first give a precise expression for the Poisson kernel for domains with smooth boundary.
	
	\begin{prop}(c.f. \cite{poisson}) \label{poisson}Let $\Omega$ be a domain with $C^\infty$ boundary. Then, its Poisson kernel reads 
		
		$$K(x,\omega)=\frac{1}{\pi}\frac{d(x)}{|x-\omega|^2}F\left(\omega,|x-\omega|,\frac{x-\omega}{|x-\omega|}\right),\quad x\in\Omega,\, \omega\in\partial\Omega,$$
		where $d\in C^\infty(\Omega)$ satisfies $d(x)=d(x,\partial\Omega)$ for $d(x,\partial\Omega)$ small enough, and $F\in C^{\infty}(\partial\Omega,\overline{\R}_+,\S^1)$ satisfies $F(\omega,\nu,0)=1$ for every $\omega\in \Omega$ and $\nu\in\mathbb{S}^1$. 
	\end{prop}
	
	This helps us to give a precise expression for the operator $A$: 
	
	\begin{coro}
		The operator $A$ can be written as 
		
		$$Aj_0(\omega)=\int_{\partial\Omega_-}\mathsf{K}(\omega,\omega')j_0(\omega')d\omega',$$
		with kernel $\mathsf{K}$ given by 
		
		$$\mathsf{K}(\omega,\omega')=\frac{B(\omega)\cdot n(\omega)}{\pi}\int_0^{L(\omega')}\frac{d(\Phi(\omega',s))}{|\Phi(\omega',s)-\omega|^2}F\left(\omega,|\Phi(\omega',s)-\omega|,\frac{\Phi(\omega',s)-\omega}{|\Phi(\omega',s)-\omega|}\right),$$
		where $F$ is as in Proposition \ref{poisson}.
	\end{coro}
	
	We are then in a position to prove Theorem \ref{mainteor2D2D}.
	
	\begin{proof}[Proof of Theorem \ref{mainteor2D2D}]
		In order to derive an estimate on the operator norm of $A$, we may choose an arbitrary point $\omega\in \partial\Omega_-$ and perform regularity estimates around it. First of all, it is convenient to notice that the kernel $\mathsf{K}$ is not well defined at the diagonal $\omega'=\omega$.
		
		We define $\varepsilon>0$ to be such that  
		
		\begin{equation}\label{epsilon}
			\varepsilon\leq \frac{1}{4}\min\left\{\frac{\inf_{\omega\in\partial\Omega_-} |B(\omega)||\sin(\theta)|}{\|B\|_{C^1}^2},\min_{\omega\in\partial\Omega_-}L(\omega)\right\},
		\end{equation}
		where $\theta$ is the angle that $B$ makes with the boundary $\partial\Omega_-$. This is a well defined quantity, since on $\partial\Omega_-$ the magnetic field $B$ is never tangent and $\partial\Omega_-$ is compact. 
		
		Then, we can split $A$ as $A:=A_1+A_2$ where $A_1$ and $A_2$ are given by the kernels $\mathsf{K}_1$ and $\mathsf{K}_2$, respectively, with` 
		
		\begin{equation}\label{maincontribution} \mathsf{K}_1(\omega,\omega')=\frac{1}{\pi}\int_0^{\varepsilon}\frac{d(\Phi(\omega',s))}{|\Phi(\omega',s)-\omega|^2}F\left(\omega,|\Phi(\omega',s)-\omega|,\frac{\Phi(\omega',s)-\omega}{|\Phi(\omega',s)-\omega|}\right).
		\end{equation} 
		$$\mathsf{K}_2(\omega,\omega')=\frac{1}{\pi}\int_\varepsilon^{L(\omega')}\frac{d(\Phi(\omega',s))}{|\Phi(\omega',s)-\omega|^2}F\left(\omega,|\Phi(\omega',s)-\omega|,\frac{\Phi(\omega',s)-\omega}{|\Phi(\omega',s)-\omega|}\right).$$
		
		Since $\text{dist}(\Phi(\omega',s),\partial\Omega_-))>0$ for every $s\geq \varepsilon$, $A_2$ is a smoothing operator. Therefore, it is easier to study. We then focus on $A_1$, as it is the one where a singularity takes place.  We prove that, in a neighborhood any point $\omega\in\partial\Omega$, $A_1j_0$  is a $C^{2,\alpha}$ function. To that end, take a parametrization $\gamma:(-a,a)\longrightarrow\partial\Omega_-$ of a neighborhood $V_{\omega}$ of $\omega$ so that $\gamma$ is parametrized by arc length and such that $a$ satisfies 
		
		\begin{equation}\label{definiciona}
			a\leq \frac{1}{4}\min\left\{\frac{\inf_{\omega'\in\partial\Omega_-} |\sin(\theta)|}{\max\kappa},\text{length}(\partial\Omega_-)\right\},
		\end{equation}
		where $\kappa$ is the curvature. We can now study $A_1j_0(\gamma(t))$ and check that it is indeed a function in $C^{2,\alpha}(-a,a)$, and estimate its norm. To do so, we consider the open cover of $\partial\Omega_-$ given by $\{V_\omega,W_{\omega}\}$, where $W_{\omega}:=\partial\Omega_-\setminus \gamma[-a/2,a/2]$. We now take a partition of unity $\alpha,\beta$ subordinated to this cover. Then, we can write the operator $A_1$ as 
		
		$$A_1j_0(\gamma(t))=A_{1}(\alpha j_0)+A_{1}(\beta j_0)$$
		in the region $[-a/4,a/4]$. Again, $A_{1}(\beta(j_0))$ is as smooth function in this region. On the other hand, $\beta\cdot A_1(\alpha j_0)$ is again smooth. Therefore, it is enough to study the function $\alpha A_{1}(\alpha j_0)$. Employing the parametrization $\gamma$, we can write it as 
		
		\begin{equation}\label{A1}
		\alpha\cdot A_{1}(\alpha j_0)=\frac{\alpha(\gamma(t))}{\pi}\int_{-a/2}^{a/2}\mathsf{K}(\gamma(t),\gamma(t'))\alpha(\gamma(t'))j_0(\gamma(t'))|B(\gamma(t'))\cdot n||\gamma'(t')|dt',
		\end{equation}
		
		This allows us to study this operator as an operator from $C^{1,\alpha}(\R)$ to $C^{2,\alpha}(\R))$, where we can use the results of Theorem \ref{boundedness:Besov}. We then see that this is an operator with kernel
		
		\begin{multline}\label{kernel2D}
			\mathscr{K}(t,t')=\alpha(\gamma(t))\alpha(\gamma(t'))|B(\gamma(t'))\cdot n|\\
			\times \int_0^{\varepsilon} \frac{d(\Phi(\gamma(t'),s))}{|\Phi(\gamma(t'),s)-\gamma(t)|^2}F\left(\gamma(t),|\Phi(\gamma(t'),s)-\gamma(t)|,\frac{\Phi(\gamma(t'),s)-\gamma(t)}{|\Phi(\gamma(t'),s)-\gamma(t)|}\right)ds
		\end{multline}
		
		We shall now prove that the adjoint of this kernel, i.e. $\tilde{\mathscr{K}}(t,t'):=\mathscr{K}(t',t)$ is the kernel of a pseudo- differential operator with limited regularity of degree $-1$, in the sense we described in Section \ref{sec:kernels}. As we know, the kernel of a pseudodifferential operator is of the form $k(x,x-y)$. We then write our kernel $\tilde{\mathscr{K}}(t,t')$   as a function of the form $f(t,t-t')$:
		\begin{equation}\label{efe}
			f(t,z)=\tilde{\mathscr{K}}(t,t-z)=\mathscr{K}(t-z,t).
		\end{equation}
		
		We claim now that this leads to a pseudodifferential operator of order $-1$. Note that, once this claim is proven, the result follows.
	\end{proof}
	
	We have then seen that the proof of Theorem \ref{mainteor2D2D} reduces to see that the kernel $f(t,\zeta)$ is the kernel of a pseudodifferential operator of order $1$. Before proving this, we shall state an important lemma that, in spite of it simplicity, will allow us to obtain the regularity estimates more efficiently.
	
	\begin{lema}\label{estimacion}
		There exists a positive constant C, only depending on the domain, and a positive $\delta<1$ such that for $s\leq\varepsilon$ and $|t-t'|\leq a$
		
		$$\frac{1}{|\Phi(\gamma(t),s)-\gamma(t')|^2}\leq \frac{C}{1-\delta}\frac{1}{|sb(\gamma(t))|^2+|t-t'|^2}.$$
	\end{lema}
	\begin{proof}
		This is a direct consequence of Taylor's theorem. Recall that, due to the definition of the flow map $\Phi$, and Taylor's theorem, 
		
		$$\Phi(\gamma(t),s)=\gamma(t)+sb(\gamma(t))-\int_0^s\tau \nabla B(\Phi(\gamma(t),\tau))\cdot B(\Phi(\gamma(t),\tau))d\tau.$$
		
		On the other hand, we can compute the next order of the Taylor polynomial for $\gamma$ to approximate the difference $\gamma(t)-\gamma(t')$, so 
		
		\begin{align*}
			\gamma(t')-\gamma(t)&=(t'-t)\int_0^1\gamma'(\tau t'+(1-\tau)t)d\tau\\
			&=(t'-t)\int_0^1\gamma'(\tau t+(1-\tau)t')d\tau\\
			&=(t'-t)\int_0^1(\tau)'\gamma'(\tau t+(1-\tau)t')d\tau\\
			&=\gamma'(t)(t'-t)+(t-t')^2\int_0^1\tau \gamma''(\tau t+(1-\tau)t')d\tau
		\end{align*}
		
		Therefore, 
		
		\begin{align*}
			\Phi(\gamma(t),s)-\gamma(t')&=\Phi(\gamma(t),s)-\gamma(t)+\gamma(t)-\gamma(t')\\
			&=sb(\gamma(t))-(t'-t)\gamma'(t)\\
			&\qquad -\int_0^s\tau \nabla B(\Phi(\gamma(t),\tau))\cdot B(\Phi(\gamma(t),\tau))d\tau+(t-t')^2\int_0^1\tau \gamma''(\tau t+(1-\tau)t')d\tau.
		\end{align*}
		
		Hence, if we denote $g:=\int_0^1\gamma'(\tau t'+(1-\tau)t)d\tau$,
		
		\begin{equation*}
			\frac{\sqrt{|sb(\gamma(t))|^2+|t-t'|^2}}{|\Phi(\gamma(t),s)-\gamma(t)+(t-t')g(t,t-t')|}\leq \frac{\sqrt{|sb(\gamma(t))|^2+|t-t'|^2}}{|sb(\gamma(t))-\gamma'(t)(t-t')|-\frac{s^2}{2}\|B\|_{C^1}-\frac{(t-t')^2}{2}\max{\kappa}}
		\end{equation*}
		
		Now, we just need to notice that 
		
		$$\frac{s^2\|B\|^2_{C^1}+(t-t')^2\max\kappa}{2\sqrt{|sb(\gamma(t))|^2+|t-t'|^2}}\leq \frac{1}{2}\left(s\|B\|_{C^2}^2+|t-t'|\max\kappa\right).$$
		
	 	Since  $s<\varepsilon$ with $\varepsilon$ as in \eqref{epsilon}, and $|t-t'|<a$ with $a$ as in \eqref{defia}, we conclude that there exists some constant $C$ satisfying
		
		$$\frac{\sqrt{|sb(\gamma(t))|^2+|t-t'|^2}}{|\Phi(\gamma(t),s)-\gamma(t)+(t-t')g(t,t-t')|}\leq C.$$
		
		Now, since $B\cdot n$ is never vanishing on the boundary, we know that it is never tangent to the boundary region $\partial\Omega_-$. Moreover, since $\partial\Omega_-$ is compact, there must exist some $\delta>0$ satisfying
		
		\begin{equation*} 
			\begin{split}
				|sb(\gamma(t))-\gamma'(t)(t-t')|^2&=|sb(\gamma(t))|^2+(t-t')^2-2s(t-t')B(\gamma(t))\cdot \gamma'(t)\\
				&\geq (1-\delta)(|sb(\gamma(t))|^2+|t-t'|^2).
			\end{split}
		\end{equation*}
		
		The result then follows. 
	\end{proof}
	
	As a result, we can now prove that $f$ is the kernel of a pseudodifferential operator of order $-1$, thus concluding the proof of Theorem \ref{mainteor2D2D}.
	
	\begin{prop}\label{pseudodifferential2D}
		Consider the function $f(t,\zeta)$ defined in \eqref{efe}. Then, for every $t$, $f(t,\cdot)$ is absolutely integrable and its Fourier transform $a(t,\cdot)$ belongs to $S^{-1}(\R^d,1+\alpha)$.
	\end{prop}
	\begin{proof}
		
		We make use of the Proposition \ref{prop:kernels}, that relates the singularity of a kernel with the decay of the Fourier Transform. 
		
		In order to simplify the notation, we shall denote $\alpha(\gamma(t)):=\sigma(t)$. Now, we firstly need to prove that the function $f(t,\zeta)$ is locally integrable for every $t$. To that end, we may employ Lemma \ref{estimacion} to check that the singularity of the kernel is an integrable one.
		
		To do so, using the fact that $F$ is smooth, and since when the functions $\alpha$ and $\sigma$ are nonzero, Lemma \ref{estimacion} above applies, we conclude that in the support of $\sigma(t)\sigma(t-\zeta)$, 
		
		$$|f(t,\zeta)|\leq C\|B\cdot n\|_{C^0(\partial\Omega_-)}\int_0^\varepsilon \frac{d(\Phi(\gamma(t),s))}{s|B(\gamma(t))|^2+\zeta^2}ds.$$
		
		Since $d(x)$ vanishes for $x\in \partial\Omega$,  $\partial\Omega$ is a $C^\infty$ manifold, and $\Phi(\omega,0)=\omega$ for every $\omega\in\partial\Omega_-$, we can estimate 
		
		$$\left|\frac{d(\Phi(\gamma(t),s))}{(s|B(\gamma(t))|)^2+\zeta^2}\right|\leq \frac{Cs\|B\|_{C^0}}{(s|B(\gamma(t))|)^2+\zeta^2}.$$
		
		If we now integrate from $0$ to $\varepsilon$ with respect to $s$, we find that the kernel is bounded by 
		
		$$\frac{C|B(\gamma(t))\cdot n|}{1-\delta}\frac{\|B\|_{C^0}}{|B(\gamma(t))|}\ln\left(1+\frac{\varepsilon}{\zeta}\right),$$
		where the constant only depends on the domain. The term above is less singular than any power law, so \eqref{singularity} holds for $\ell=0$. We now need to estimate the different derivatives and the Hölder norms, as well as the cancellation condition \eqref{cancellation}, that ensures that this is the kernel of a pseudodifferential operator of order $-1$.

		We now write an asymptotic expansion for this term. First, notice that the function $F$ is defined so that $F(\omega,\nu,0)=1$ for every $\omega\in \partial\Omega_-$ and $\nu\in \mathbb{S}^2$. Therefore, we can employ its Taylor expansion to write 
		
		\begin{multline}\label{fgammazeta}
			f(\gamma(t),\zeta)=\frac{1}{\pi}\int_0^\varepsilon \frac{d(\Phi(\gamma(t),s))|B(\gamma(t))\cdot n|\cdot \sigma(t)\sigma(t-\zeta)}{|\Phi(\gamma(t),s)-\gamma(t)+(t-t')g(t,t-t')|^2}ds\\
			+\sum_{j=1}^N \frac{1}{\pi}\int_0^\varepsilon\frac{d(\Phi(\gamma(t),s))|B(\gamma(t))\cdot n|\cdot \sigma(t)\sigma(t-\zeta)}{|\Phi(\gamma(t),s)-\gamma(t)+(t-t')g(t,t-t')|^2}\\
			\times \frac{(\partial_rF)\left(\Phi(\gamma(t),s),\frac{\Phi(\gamma(t),s)-\gamma(t)+(t-t')g(t,t-t')}{|\Phi(\gamma(t),s)-\gamma(t)+(t-t')g(t,t-t')|},0\right)}{j!}ds+R_N
		\end{multline}
		
		It will become clear after we prove that the Fourier Transform of the first summand on the right hand side in \eqref{fgammazeta} is in $S^{-1}(\R^d,1+\alpha)$, the same arguments allow us to check that remainder terms belong to $S^{-1-j}(\R^d,1+\alpha)$. 

		\textbf{$L^\infty$ estimates on the $\zeta$ derivatives}
		
		We now derive the estimates on the $\zeta$ derivatives. This can be done by noticing that when we perform a derivative with respect to $\zeta$, we obtain  always contributions from the derivatives acting on $\sigma(t-\zeta)$ and contributions from the derivatives acting on 
		
		$$\frac{1}{|\Phi(\gamma(t),s)-\gamma(t)+\zeta g(t,\zeta)|^2}.$$
		
		When the derivative acts on $\sigma$, the term does not become more singular, so the estimates are straightforward and can be performed in the same way as in the case above. On the other hand, every time a $\zeta$ derivative acts on the denominator, we obtain a more singular term. By elementary computations, and since the only explicit dependence on $\zeta$ is given by the term $\zeta g(t,\zeta)$, one infers readily that the $\ell$-th derivative with respect to $\zeta$ acting on the denominator can be bounded by 
		
		\begin{equation} \label{aaaa´}
			\frac{C}{|\Phi(\gamma(t),s)-\gamma(t)+\zeta g(t,\zeta)|^{2+\ell}},
		\end{equation}
		where the constant $C$ depends only  on the domain $\Omega$.		
		
		Therefore, we can estimate the singularity by means of Lemma \ref{estimacion}, so the $\ell$-th derivative with respect to $\zeta$ of kernel is bounded by 
		
		\begin{equation}\label{estimatezeta} C\frac{|B\cdot n|_{C^0}\|B\|_{C^0}}{(1-\delta)^{2+\ell}}\int_0^\varepsilon\frac{s}{((s|B(\gamma(t))|)^2+\zeta^2)^{(2+\ell)/2}}ds,
		\end{equation} 
		where $C$ only depends on the domain. This is an integral that can be easily computed, so we derive the estimate
		
		\begin{equation*} 
			\begin{split} 
				|\partial_\zeta f(t,\zeta)|&\leq C|B\cdot n|_{C^0}\|B\|_{C^0}\int_0^\varepsilon \frac{sds}{((s|B(\gamma(t))|)^2+\zeta^2)^{(2+\ell)/2}}\\
				&\leq \frac{C|B\cdot n|_{C^0}\|B\|_{C^0}}{(1-\delta)^{(2+\ell)/2}\min_{\omega\in\partial\Omega_-}|B(\omega)|}\frac{1}{|\zeta|^\ell}\int_0^\infty \frac{s}{(s^2+1)^{(2+\ell)/2}}ds.
			\end{split}
		\end{equation*}
		
		\textbf{$L^\infty$ estimates for the $t$ derivatives}
		
		We are now interested in obtaining estimates on the $t$ derivative. In this case there are three main contributions, so the derivative of $f$ with respect to $t$ reads 
		
		\begin{equation}\label{tderiv}
			\begin{split}
				\partial_t f(t,\zeta)&=\frac{1}{\pi} \partial_t (|B(\gamma(t)\cdot n|\cdot \sigma(t)\sigma(t-\zeta))\int_0^\varepsilon \frac{d(\Phi(\gamma(t),s))}{|\Phi(\gamma(t),s)-\gamma(t)+\zeta g(t,\zeta)|^2}ds\\
				&+\frac{1}{\pi} |B(\gamma(t)\cdot n|\cdot \sigma(t)\sigma(t-\zeta)\int_0^\varepsilon \frac{\nabla d(\Phi(\gamma(t),s))\cdot \partial_t\Phi(\gamma(t),s)}{|\Phi(\gamma(t),s)-\gamma(t)+\zeta g(t,\zeta)|^2}ds\\
				&+\frac{1}{\pi} |B(\gamma(t)\cdot n|\cdot \sigma(t)\sigma(t-\zeta)\int_0^\varepsilon \frac{d(\Phi(\gamma(t),s))}{|\Phi(\gamma(t),s)-\gamma(t)+\zeta g(t,\zeta)|^4}\\
				&\qquad \cdot (\partial_t \Phi(\gamma(t),s)-\gamma'(t)+\zeta \partial_tg(t,\zeta))\cdot \Phi(\gamma(t),s)-\gamma(t)+\zeta g(t,\zeta).\\
				&=J_1+J_2+J_3
			\end{split}
		\end{equation}
		
		The first summand can be estimated exactly as in the case before, leading to 
		
		$$|J_1(t,\zeta)|\leq \frac{\|B\cdot n\|_{C^{1}}}{(1-\delta)\min_{\omega\in\partial\Omega_-}|B(\omega)|}\|B\|_{C^0}\ln\left(1+\frac{\varepsilon}{\zeta}\right),$$
		and 
		$$|\partial_\zeta^\ell J_1(t,\zeta)|\leq \frac{\|B\cdot n\|_{C^{1}}}{(1-\delta)\min_{\omega\in\partial\Omega_-}|B(\omega)|}\|B\|_{C^0}\frac{1}{|\zeta|^\ell},$$
		i.e., we obtain a control in terms of the $C^1$ norm of the boundary data $B\cdot n$. For the second summand in \eqref{tderiv}, we just need to notice that the numerator is also bounded by $s$. Indeed, at $s=0$, the flow map $\Phi(\gamma(t),0)=\gamma(t)$ due to the boundary condition. Since the distance function $d$ is zero along the boundary $\partial\Omega_-$, its gradient at $\gamma(t)$ has to be normal to the boundary at such point. But $\partial_t \Phi(\gamma(t),0)=\gamma'(t)$, so the numerator in the integrand of the second term vanishes at $s=0$ as well. Therefore, we can repeat  the arguments that we used to bound $\|\partial_\zeta^\ell f(\cdot,\zeta)\|_{L^\infty}$, so we find that 
		
		$$|J_2(t,\zeta)|\leq\frac{\|B\cdot n\|_{C^{0}}}{\min_{\omega\in\partial\Omega_-}|B(\omega)|}\|B\|_{C^1}\ln\left(1+\frac{\varepsilon}{\zeta}\right),$$
		and 
		$$|\partial_\zeta^\ell J_2(t,\zeta)|\leq\frac{\|B\cdot n\|_{C^{0}}}{\min_{\omega\in\partial\Omega_-}|B(\omega)|}\|B\|_{C^1}\frac{1}{|\zeta|^\ell},$$
		i.e., the estimate this time is controlled by the $C^1$ norm of the magnetic field $B$. This follows from the fact that, this time, the cancellation condition leads to an estimate depending on the $C^1$ norm of the function $\nabla d(\Phi(\gamma(t),s))\cdot \partial_t\Phi(\gamma(t),s)$, and thus leading to the $\|B\|_{C^1}$ factor instead. Finally, the third summand in \eqref{tderiv} requires a bit more of care. If we now take $\partial^\ell_\zeta J_3$, the only relevant contributions are those of the integrand. We then notice that the most singular contribution is given by 
		
		\begin{equation}\label{estoyhastalasnaricesya} \partial_\zeta ^{\ell}\left(\frac{\Phi(\gamma(t),s)-\gamma(t)+\zeta g(t,\zeta)}{|\Phi(\gamma(t),s)-\gamma(t)+\zeta g(t,\zeta)|^4}\right)\partial_\zeta^{j-\ell}\left(\partial_t \Phi(\gamma(t),s)-\gamma'(t)+\zeta \partial_tg(t,\zeta)\right).
		\end{equation}
		
		Using the same arguments that led us to obtain \eqref{aaaa´}, i.e. applying Lemma \ref{estimacion}, the first factor in the formula above can be estimated by 
		
		$$\frac{C}{((s|B(\gamma(t)|)^2+\zeta^2)^{\frac{3}{2}+\ell}}.$$
		
		On the other hand, if $j-\ell=0$, we derive a bound of the form 
		
		$$C(\Omega)(\|B\|_{C^1}s+\zeta),$$
		for the second factor in \eqref{estoyhastalasnaricesya} and if $j-\ell$ is greater than zero, then it is bounded by a constant. Therefore, depending on the value of $j-\ell$, we obtain that the third summand in \eqref{tderiv} is bounded from above by 
		
		\begin{equation*}
			\begin{split}
				C(\Omega)\frac{\|B\|_{C^1}^2}{(1-\delta)^{\frac{3}{2}+j}\min_{\omega\in\partial\Omega_-} |B(\omega)|}&\int_0^\infty \frac{s(s+\zeta)}{(s^2+\zeta^2)^{\frac{3}{2}+j}}ds\\
				&\leq \frac{\|B\|_{C^1}^2}{(1-\delta)^{\frac{3}{2}+j}\min_{\omega\in\partial\Omega_-} |B(\omega)|}\frac{1}{|\zeta|^j}\int_0^\infty \frac{u(u+1)}{(u^2+1)^{3/2+j}}\quad \text{if }\ell=j>0,
			\end{split}
		\end{equation*}
		\\
		by 
		
		\begin{equation*}
			\begin{split}
				C(\Omega)\frac{\|B\|_{C^1}^2}{(1-\delta)^{\frac{3}{2}}\min_{\omega\in\partial\Omega_-} |B(\omega)|}\int_0^\varepsilon \frac{s(s+\zeta)}{(s^2+\zeta^2)^{\frac{3}{2}}}ds&\leq \frac{\|B\|_{C^1}^2}{(1-\delta)^{\frac{3}{2}}\min_{\omega\in\partial\Omega_-} |B(\omega)|}\int_0^{1/\zeta} \frac{u(u+1)}{(u^2+1)^{3/2}}\\
				&\leq C\frac{\|B\|_{C^1}^2}{\min_{\omega\in\partial\Omega_-} |B(\omega)|}\ln\left(1+\frac{\varepsilon}{\zeta}\right)\quad \text{if } j=\ell=0,
			\end{split}
		\end{equation*}
		and by 
		
		\begin{equation*}
			\begin{split} 
			C(\Omega)\frac{\|B\|_{C^1}^2}{(1-\delta)^{\frac{3}{2}+\ell}\min_{\omega\in\partial\Omega_-} |B(\omega)|}&\int_0^\infty \frac{s}{(s^2+\zeta^2)^{\frac{3}{2}+\ell}}ds\\
			&\leq \frac{\|B\|_{C^1}^2}{(1-\delta)^{\frac{3}{2}+j}\min_{\omega\in\partial\Omega_-} |B(\omega)|}\frac{1}{|\zeta|^\ell}\int_0^\infty \frac{u}{(u^2+1)^{3/2+j}}du\quad j\neq \ell
			\end{split}
		\end{equation*}
		 Therefore, we inmediately establish the $L^\infty$ estimate for the kernel. We now have to find the $C^{1,\alpha}$ estimates. This is just a rutinary repetition of the computation above. Notice that, actually, we can actually obtain estimates for the second $t-$derivative of $f$, so we actually can obtain the following bounds for the kernel 
		
		$$\|\partial^{\alpha}_\zeta k(\cdot,\zeta)\|_{C^{2}}\leq C|\zeta|^{-|\alpha|}.$$

		\textbf{The cancellation condition}
		
		We check now the cancellation condition \eqref{cancellation},so that we have to estimate $f(\gamma(t),\zeta)-f(\gamma(t),-\zeta)$. This difference can be written as 
		\small
		\begin{equation*}
			\begin{split}
				\frac{1}{\pi}\int_0^\varepsilon d(\Phi(\gamma(t),s))&|B(\gamma(t))\cdot n|\alpha(\gamma(t))\left(\frac{\sigma(t-\zeta)-\sigma(t+\zeta)}{|\Phi(\gamma(t),s)-\gamma(t)+\zeta g(t,\zeta)|^2}\right. \\
				&\left.+\sigma(t+\zeta)\left(\frac{1}{|\Phi(\gamma(t),s)-\gamma(t)+\zeta g(t,\zeta)|^2}-\frac{1}{|\Phi(\gamma(t),s)-\gamma(t)-\zeta g(t,-\zeta)|^2}\right)\right)\\
				&=I_1+I_2
			\end{split}
		\end{equation*}
		
		\normalsize
		We now have to take the integral between $1/2\xi$ and $1/\xi$. For the first one, we notice that the difference $\sigma(t-\zeta)-\sigma(t+\zeta))$ has no dependence on $s$. Therefore, we can repeat the same estimates we had for $f$ to conclude that it is bounded by 
		
		$$C\frac{|B(\gamma(t))\cdot n|}{1-\delta}\frac{\|B\|_{C^0}}{|B(\gamma(t))|}|\zeta|\ln\left(1+\frac{\varepsilon}{\zeta}\right)$$
		
		Now, since 
		
		$$\int s\ln(s)ds=\frac{s^2}{4}(2\log(s)-1)+C,$$
		we conclude that the integral between $1/2\xi$ and $1/\xi$ is bounded by a constant. For the other term, the content of the big parenthesis can be computed explicitly, so that it equals to 
		
		$$\frac{\zeta^2|(g(t,-\zeta)|^2-|g(t,\zeta)|^2)+2\zeta(\Phi(\gamma(t),s)-\gamma(t))\cdot(g(t,-\zeta)-g(t,\zeta))}{|\Phi(\gamma(t),s)-\gamma(t)+\zeta g(t,\zeta)|^2|\Phi(\gamma(t),s)-\gamma(t)-\zeta g(t,-\zeta)|^2}$$
		
		This can be divided into two contributions, each of which can be estimated by 
		
		$$\frac{C}{(1-\delta)^2}\frac{\zeta^2}{|\left(s|B(\gamma(t))|)^2+|\zeta|^2\right)^{2}}\quad \text{and}\quad \frac{C}{(1-\delta)^2}\frac{s\zeta}{|\left(s|B(\gamma(t))|)^2+|\zeta|^2\right)^{2}}$$
		
		Therefore, $I_2$ be bounded by 
		
		$$C\|B\|_{C^0}\frac{|B(\gamma(t))\cdot n|}{1-\delta}\int_0^\varepsilon\frac{s\zeta(s+\zeta)}{|\left(s|B(\gamma(t))|)^2+|\zeta|^2\right)^{2}}ds.$$

		We now estimate the integral with respect to $\zeta$ between $1/2\xi$ and $1/\xi$. We can perform a change of variables given the homogeneity of the integral, resulting in 
		
		$$C\|B\|_{C^0}\frac{|B(\gamma(t))\cdot n|}{1-\delta}\frac{1}{\xi}\int_{1/2}^1\int_0^{\xi \varepsilon/z}\frac{u(u+1)}{(u|B(\gamma(t)|)^2+1)^2}dudz.$$
		
		This can be split into two contributions. On the one hand, 
		
		\begin{equation*}
			\begin{split}
				\frac{1}{\xi}\int_{1/2}^1\int_0^{\xi \varepsilon/z}\frac{u^2}{(u|B(\gamma(t)|)^2+1)^2}dudz&\leq 2\varepsilon\int_0^{\xi \varepsilon/z}\frac{u}{(u|B(\gamma(t)|)^2+1)^2}dudz\\
				&\leq\frac{2\varepsilon}{|B(\gamma(t))|^2}\int_0^{\infty}\frac{u}{(u^2+1)^2}dudz.
			\end{split}
		\end{equation*}
		
		On the other,
		
		\begin{equation*}
			\begin{split}
				\frac{1}{\xi}\int_{1/2}^1\int_0^{\xi \varepsilon/z}\frac{u}{(u|B(\gamma(t)|)^2+1)^2}dudz&= \frac{2}{|B(\gamma(t))|^2}	\frac{1}{\xi}\int_{1/2}^1\left(1-\frac{z^2}{(\xi \varepsilon)^2+z^2}\right)dz\\
				&=\frac{2\varepsilon^2}{|B(\gamma(t))|^2}\int_{1/2}^1\left(\frac{1}{(\xi \varepsilon)^2+z^2}\right)dz\\
				&=\frac{C\varepsilon^2}{|B(\gamma(t))|^2}.
			\end{split}
		\end{equation*}
		
		Therefore, the cancellation condition \eqref{cancellation} for the kernel $f$ holds.

		In order to check the cancellation condition \eqref{cancellation} for $\partial_tf$, we can again collect some of the estimates we already performed. Indeed, if we look at \eqref{tderiv}, the cancellation condition for the $\partial_t J_1$ and $\partial_t J_2$ can be estimated in the same way as the cancellation condition for $f$. The only term that requires a different argument is $J_3$. If we compute $J_3(t,\zeta)-J_3(t,-\zeta)$, we obtain  a sum of the following four terms:
		
		\begin{equation*}
			\begin{split}
				J_{31}(t,\zeta)&=-\frac{2}{\pi}|B(\gamma(t)\cdot n)|\alpha(\gamma(t))(\sigma(t-\zeta)-\sigma(t+\zeta))\int_0^\varepsilon \frac{d(\Phi(\gamma(t),s))}{|\Phi(\gamma(t),s)-\gamma(t)+\zeta g(t,\zeta)|^4}\\
				&\times \left(\Phi(\gamma(t),s)-\gamma(t)+\zeta g(t,\zeta)\right)\cdot\left(\frac{\partial}{\partial t}\Phi(\gamma(t),s)-\gamma'(t)+\zeta\partial_t g(t,\zeta)\right)ds
			\end{split}
		\end{equation*}
		
		\begin{equation*}
			\begin{split}
				J_{32}(t,\zeta)&=-\frac{2}{\pi}|B(\gamma(t)\cdot n)|\alpha(\gamma(t))\sigma(t+\zeta)\\
				&\times \int_0^\varepsilon d(\Phi(\gamma(t),s))\left(\frac{1}{|\Phi(\gamma(t),s)-\gamma(t)+\zeta g(t,\zeta)|^4}-\frac{1}{|\Phi(\gamma(t),s)-\gamma(t)-\zeta g(t,-\zeta)|^4}\right)\\
				&\times \left(\Phi(\gamma(t),s)-\gamma(t)+\zeta g(t,\zeta)\right)\cdot\left(\frac{\partial}{\partial t}\Phi(\gamma(t),s)-\gamma'(t)+\zeta\partial_t g(t,\zeta)\right)ds
			\end{split}
		\end{equation*}
		
		\begin{equation*}
			\begin{split}
				J_{33}(t,\zeta)&=-\frac{2}{\pi}|B(\gamma(t)\cdot n)|\alpha(\gamma(t))\sigma(t+\zeta)\int_0^\varepsilon \frac{d(\Phi(\gamma(t),s))}{|\Phi(\gamma(t),s)-\gamma(t)-\zeta g(t,-\zeta)|^4}\\
				&\times \zeta\left(g(t,\zeta)-g(t,-\zeta)\right)\cdot\left(\frac{\partial}{\partial t}\Phi(\gamma(t),s)-\gamma'(t)+\zeta\partial_t g(t,\zeta)\right)ds
			\end{split}
		\end{equation*}
		\begin{equation*}
			\begin{split}
				f_{134}(t,\zeta)&=-\frac{2}{\pi}|B(\gamma(t)\cdot n)|\alpha(\gamma(t))\sigma(t+\zeta)\int_0^\varepsilon \frac{d(\Phi(\gamma(t),s))}{|\Phi(\gamma(t),s)-\gamma(t)-\zeta g(t,-\zeta)|^4}\\
				&\times \zeta\left(\Phi(\gamma(t),s)-\gamma(t)+\zeta g(t,\zeta)\right)\cdot\left(\partial_t g(t,\zeta)-\partial_t g(t,-\zeta)\right)ds
			\end{split}
		\end{equation*}
		
		Estimates for $J_{21}$ are obtained arguing as with  $f(t,\zeta)$, with the difference that the term $\sigma(t-\zeta)-\sigma(t+\zeta)$ can be estimated from above by $C\zeta$. Therefore,  $J_{21}$ is bounded by $\zeta\ln(\zeta)$ for $\zeta$ near $0$. Therefore, $J_{21}$ is bounded in a neighborhood of zero. The estimates for $J_{23}$ and $J_{24}$  can be respectively bounded by 
		
		\begin{equation} \label{pff2}C\|B\cdot n\|_{C^0}\frac{s\zeta^2(s\|B\|_{C^1}+\zeta)}{((s|B(\gamma(t))|)^2+\zeta ^2)^2} \quad \text{and}\quad C\|B\cdot n\|_{C^0}\frac{s\zeta^2}{((s|B(\gamma(t))|)^2+\zeta ^2)^{3/2}}.
		\end{equation}

		For $f_{122}$ we have to estimate the difference  given by 
		
		$$\frac{1}{|\Phi(\gamma(t),s)-\gamma(t)+\zeta g(t,\zeta)|^4}-\frac{1}{|\Phi(\gamma(t),s)-\gamma(t)-\zeta g(t,-\zeta)|^4}.$$
		
		One just has to notice that this is given by 
		\small
		\begin{equation}\label{pff3}
			\begin{split}
				&\frac{1}{|\Phi(\gamma(t),s)-\gamma(t)+\zeta g(t,\zeta)|^2}\left(\frac{1}{|\Phi(\gamma(t),s)-\gamma(t)+\zeta g(t,\zeta)|^2}-\frac{1}{|\Phi(\gamma(t),s)-\gamma(t)-\zeta g(t,-\zeta)|^2}\right)\\
				&\frac{1}{|\Phi(\gamma(t),s)-\gamma(t)-\zeta g(t-\zeta)|^2}\left(\frac{1}{|\Phi(\gamma(t),s)-\gamma(t)+\zeta g(t,\zeta)|^2}-\frac{1}{|\Phi(\gamma(t),s)-\gamma(t)-\zeta g(t,-\zeta)|^2}\right)
			\end{split}
		\end{equation}
		\normalsize
		and employ same estimates we have used so far, leading to a bound for $f_{122}$ of the form 
		
		$$C(\Omega)\|B\cdot n\|_{C^0}\|B\|_{C^0}\int_0^s\frac{s(\|B\|_{C^1}s+\zeta)}{((sB(\gamma(t))^2+\zeta^2)^6}(\zeta^3+s\zeta^2)ds$$
		
		Combining  \eqref{pff2} and \eqref{pff3}, it is now straightforward to see that the contribution to the integral between $1/{2\xi}$ and $1/\xi$ is bounded by a constant. 
		
		We can extend the arguments here to estimate the operator norm of $A[B_1]-A[B_2]$.
	\end{proof}

	\subsubsection{Principal symbol of the operator $A$}\label{calculosimbolo2D}
	
	In the previous section we studied how to prove, if $B\in C^{2,\alpha}(\Omega)$ is non-tangential to $\partial\Omega_-$, that the operator $A[B]$ introduced in Definition \ref{defia} is bounded as an operator from $C^{1,\alpha}(\partial\Omega_-)$ to $C^{2,\alpha}(\partial\Omega_-)$. Moreover, we checked that, due to the form of the singularities, $A[B]$ was given by the sum of an operator $A_1$, given by kernel \eqref{maincontribution} which is the kernel of (the adjoint of) a pseudodifferential operator plus a compact perturbation. 
	
	Notice that, if $B$ was not only $C^{2,\alpha}$, but $C^\infty$, then we could apply the usual theory for pseudodifferential operators. In this section we will prove that our operator equals an \textit{elliptic} pseudodifferential operator of order $-1$, in the sense we introduced in Section \ref{auxiliaryresults}. We will also employ Theorems \ref{Fredholm} and \ref{teorindice} to prove that our operator is Fredholm of index zero. We begin by noticing that, if $B$ is smooth, as well as the boundary, then the map given by $A$ equals the adjoint of a classical pseudodifferential operator of order $-1$, since we can repeat the computations we made in Proposition \ref{pseudodifferential2D} . Due to the observation done in Section \ref{sec:opA}, since $A$ is given by the sum of $A_1$, defined in \eqref{maincontribution}  plus contributions of lower order, we can compute the principal symbol of $A$ just by studying the one of $A_1$. To this end, we notice that the singularity comes from the point $\zeta=0$, as expected. In particular, due to the fact that the integral involved reaches $s=0$, where the denominator vanishes.
	
	Summarizing, the main proposition of this subsection reads as follows
	
	\begin{prop}\label{Fredholm2D}
		Consider the operator $A$  in Definition \ref{defia}, and assume that $B$ is $C^\infty$. Then, its principal symbol equals 
		
		$$a(x,\xi)=\frac{1}{|\xi||B(x)|}\frac{n(x)\cdot B(x)}{i\text{sgn}(\xi)\cos(\theta)+|\sin(\theta)|},$$
		where $\theta$ is the angle formed by $B$ and the boundary $\partial\Omega_-$ and $n(x)$ is the inner normal to the surface. Furthermore, the operator $A$ is Fredholm and of index zero. 
	\end{prop}
	\begin{proof}
		The denominator in the kernel defining $A_1$ in \eqref{A1} is given by the function 
		
		$$\mathscr{F}(\gamma(t),\zeta,s):=|\Phi(\gamma(t),s)-\gamma(t)+\zeta g(t,\zeta)|^2.$$
		
		This is a  smooth function that satisfies 
		
		$$\mathscr{F}(t,0,0)=0$$
		\begin{equation*}
			\partial_\zeta\mathscr{F}(w,\zeta,s)=2\left(\Phi(\gamma(t),s)-\gamma(t)+\zeta g(t,\zeta)\right)\cdot \left(g(t,\zeta)+\zeta \partial\zeta g(t,\zeta))\right)
		\end{equation*}
		\begin{equation*}
			\partial_s\mathscr{F}(t,\zeta,s)=2\left(\Phi(\gamma(t),s)-\gamma(t)+\zeta g(t,\zeta)\right)\cdot B(\Phi(\gamma(t),s))
		\end{equation*}
		where we have used the definition of $\Phi$ in terms of the ODE. Now, this means that $\nabla_{\zeta,s}\mathscr{F}(w,0,0)$ equals zero once again. Therefore, due to Taylor's theorem, the function $\mathscr{F}(w,\zeta,s)$ equals 
		
		$$\left(\begin{array}{cc}
			\zeta & s\\
		\end{array}\right)
		\left(\begin{array}{cc}
			R_{11}(t,\zeta,s) & R_{12}(t,\zeta,s)\\
			R_{21}(t,\zeta,s) & R_{22}(t,\zeta,s)
		\end{array}\right)
		\left(\begin{array}{cc}
			\zeta \\
			s\\
		\end{array}\right),$$
		i.e. it is a quadratic function of $(\zeta,s)$ where the matrix of the quadratic form also depends on $t$, $\zeta$ and $s$. Therefore,
		
		\begin{equation}\label{hessian}
			\mathscr{F}(w,\zeta,s)=(\zeta^2+s^2)\left(\begin{array}{cc}
				\omega_\zeta & \omega_s\\
			\end{array}\right)
			\left(\begin{array}{cc}
				R_{11}(t,\zeta,s) & R_{12}(t,\zeta,s)\\
				R_{21}(t,\zeta,s) & R_{22}(t,\zeta,s)
			\end{array}\right)
			\left(\begin{array}{cc}
				\omega_\zeta \\
				\omega_s\\
			\end{array}\right),
		\end{equation}
		where $\omega=(\omega_\zeta,\omega_s):=\frac{(\zeta,s)}{\sqrt{\zeta^2+s^2}}\in \mathbb{S}.$ We can then write the kernel \ref{kernel2D} as 
		
		$$f(\gamma(t),\zeta)=\int_0^\varepsilon \frac{1}{\zeta^2+s^2}\textsf{F}(t,\zeta,s,\omega)ds.$$
		
		The function $\textsf{F}:\R\times \R\times\R\times \mathbb{S}\rightarrow \R$ is smooth in each of its entries, for a small enough choice of the partition of unity. Indeed, at $(s,\zeta)=0$, the matrix we defined in \eqref{hessian} equals
		
		$$\left(\begin{array}{cc}
			|B(\gamma(t))|^2 & B(\gamma(t))\textsf{t}(\gamma(t))\\
			B(\gamma(t))\textsf{t}(\gamma(t)) & |\textsf{t}(\gamma(t))|
		\end{array}\right).$$
		
		This matrix is positive definite, so in a small neighborhood of $s=0$ (controlled by $\varepsilon$) and $\zeta$ (controlled by $a$), it is positive definite and, therefore, the function $\textsf{T}$ is smooth in each of its variables. This means that we can perform a Taylor expansion around $s=\zeta=0$. This leads to an expansion of the form 
		
		$$\textsf{F}(w,s,\zeta,\omega)=\textsf{F}(w,0,0,\omega)+s\partial_s\textsf{F}(w,0,0,\omega)+\zeta\partial_\zeta \textsf{F}(w,0,0,\omega)+o(s^2+\zeta^2) \text{ as }(s,\zeta)\rightarrow 0.$$
		
		The first term on the right hand side is zero since the function $dist(x,\partial\Omega)$ vanishes for $x\in \partial\Omega$. The leading order is then given by 
		
		\begin{equation}\label{Furieretransforme}
			\alpha(\gamma(t)) \alpha(\gamma(t))\nabla d(\gamma(t))\cdot B(\gamma(t))|B(\gamma(t)\cdot n)|\int_0^\varepsilon \frac{s}{(s|B(\gamma(t))|)^2+2s\zeta \textsf{t}\cdot B(\gamma(t))+(\zeta|\textsf{t}|)^2}ds.
		\end{equation}
		
		Completing squares we rewrite the denominator as:
		
		\begin{equation*}
			(s|B(\gamma(t))|)^2+2s\zeta \textsf{t}\cdot B(\gamma(t))+(\zeta|\textsf{t}|)^2=s^2|B(\gamma(t))|^2\sin^2(\theta)+\left(|\textsf{t}|\zeta+s\frac{B(\gamma(t))\cdot \textsf{t}}{|\textsf{t}|}\right)^2,
		\end{equation*}
		where $\theta$ is the angle between $B$ and the tangent vector to the point $\gamma(t)$.
		
		We can compute the Fourier transform of \eqref{Furieretransforme}. We first perform a change of variables $u=|\textsf{t}|\zeta+s\frac{B(\gamma(t))\cdot \textsf{t}}{|\textsf{t}|}$, so
		
		\begin{multline}
			\int_\R e^{-i\zeta\xi}\int_0^\varepsilon \frac{s}{s^2|B(\gamma(t))|^2\sin^2(\theta)+\left(|\textsf{t}|\zeta+s\frac{B(\gamma(t))\cdot \textsf{t}}{|\textsf{t}|}\right)^2}ds\,d\zeta\\
			=\frac{1}{|\textsf{t}|}\int_0^\varepsilon \int_\R e^{-i\frac{u\xi}{|\textsf{t}|}} \frac{se^{is\xi\frac{B\cdot \textsf{t}}{|\textsf{t}|^2}}}{s^2|B(\gamma(t))|^2\sin^2(\theta)+u^2}du\, ds.
		\end{multline}
		
		We now employ some identities of the Fourier Transform. Notice that, for $\alpha>0$,
		
		$$\int_\R e^{-\alpha|x|}e^{-i\xi x}dx=
		\frac{2\alpha}{\alpha^2+\xi^2}.$$
		
		If we now choose $\alpha=s|B(\psi^{-1})||\sin(\theta)|$, we obtain 
		
		\begin{multline}
			\int_\R e^{-i\zeta\xi}\int_0^\varepsilon \frac{s}{s^2|B(\gamma(t))|^2\sin^2(\theta)+\left(|\textsf{t}|\zeta+s\frac{B(\gamma(t))\cdot \textsf{t}}{|\textsf{t}|}\right)^2}ds\,d\zeta\\
			=\frac{\pi}{|B(\gamma(t)|\textsf{t}|^2|\sin(\theta)|}\int_{0}^\varepsilon e^{is\xi\frac{B\cdot \textsf{t}}{|\textsf{t}|^2}}e^{s|B(\psi^{-1})||\sin(\theta)|\frac{|\xi|}{|\textsf{t}|}}ds.
		\end{multline}
		
		Integrating we conclude that the principal symbol of the operator $A$ reads 
		
		$$a(x,\xi)=\frac{ n(\gamma(t))\cdot B(\gamma(t))}{|\xi| |B(\gamma(t))|\left(i\text{sgn}(\xi)\cos(\theta)+|\sin(\theta)|\right)}.$$
		
		This proves ellipticity of the symbol. We notice that the symbol has very convenient properties. Indeed, one sees readily that, if one studies the range of the function $|\xi|\cdot a(x,\xi),$
		then inmediately notices that it is contained in the region 
		
		$$\{z\in\mathbb{C}\,:\, \text{Re }z>0\}.$$
		
		One can even be more precise and, since the magnetic field is never tangent to $\partial\Omega_-$, one readily finds that there exists some positive constant $a$ so that the image of $|\xi|\cdot a(x,\xi)$ lies in 
		
		$$\{z\in\mathbb{C}\,:\, \text{Re }z>c\}.$$
		
		Therefore, we conclude that the operator has index equal to zero, by means of Theorem \ref{teorindice}
	\end{proof}
	
	\subsection{Application to nearly annular domains}\label{asunciones2d}
	
	Now we will employ the techniques described above to particular choices of domains $\Omega$ and $B_0$. We begin by proving the statement we made in the introduction, checking that our requirements on the magnetic field $B_0$ are non trivial, as we can construct a large class magnetic fields $B_0$ for many domains where the assumptions in Section \ref{asunciones1} are satisfied. 
	
	Recall that, in two dimensions, any domain $\Omega:=\Omega_1\setminus \overline{\Omega}_0$ with both open, simply connected, with smooth boundary, and with $\overline{\Omega}_0\subset \Omega_1$ admit potential vector fields $\nabla u$ that never vanish . This is a simple application of the argument principle: 
	
	\begin{prop}\label{proposition2D}
		Let $\Omega$ be as described above. Then, if $\varphi\in C^{\infty}(\Omega)$ satisfies 
		
		$$\left\{
		\begin{array}{ll}
			\Delta u=0 & \text{in }\Omega\\
			u=a & \text{on } \partial\Omega_+\\
			u=b<a & \text{on }\partial\Omega_-\\
		\end{array}
		\right.,$$
		the divergence free vector field $B:=\nabla u$ never vanishes. 
	\end{prop}
	\begin{proof}
		We argue as in \cite{Alessandrini-Magnanini}. By Hopf's lemma, $n\cdot B<0$ (resp. $>0$) on $\partial\Omega_-$ (resp. $\partial\Omega_+$). Moreover, due to the fact that $f(z)=\partial_x u(z)+i\partial_y u(y)$ is harmonic, we can make use of the argument principle. That, is the number of zeros of $f$ (i.e. the number of vanishing points of $B$) equals 
		
		$$\int_{\partial\Omega}\frac{f'}{f}dz.$$
		As $u$ is constant on the boundary, its derivative is parallel to the normal vector. Therefore, the change of argument equals the change of argument of the normal vector. As a result, as the boundary has two connected components, we conclude that the number of zeros of $f$ must be zero. 
	\end{proof}
	
	In the previous sections we have proven that the operator $A$ is stable with respect to perturbations of the magnetic field, which leads to solvability of the equation as long as we know an initial magnetic field $B_0$ that makes $A$ invertible. We also proved that, under suitable smoothness assumptions on the magnetic field $B_0$, the operator $A$ is a Fredholm operator with index zero, meaning that invertibility of the operator can be proven simply by checking injectivity. Here we will prove that the operator is invertible for a particular geometry, where the domain is an annulus: 
	
	$$\{(x,y)\in \R^2\,:\, 1\leq x^2+y^2\leq L\}$$
	for some $L>0$ and the magnetic field is the monopolar field given by 
	
	$$B=\frac{(x,y)}{x^2+y^2}.$$
	
	\begin{teor}\label{mainteor2D}
		Let $L>1$ and $\alpha \in (0,1)$. Then, if $\Omega=\{(x,y)\in\R^2 \,:\, 1<x^2+y^2<L^2\}$, there exists a (small) constant $M=M(\alpha,L)>0$ such that, if $\in C^{2,\alpha}(\partial\Omega)$ and $g\in C^{2,\alpha}(\partial\Omega_-)$ satisfy 
		
		$$\int_{\partial\Omega_-}f = 0,$$
		and 
		
		$$\|f\|_{C^{2,\alpha}}+\|g\|_{C^{2,\alpha}}\leq M,$$
		then there exists a solution $(B,p)\in C^{2,\alpha}(\Omega)\times C^{2,\alpha}(\Omega)$ of the problem 
		
		$$\left\{\begin{array}{ll}
			(\nabla\times B)\cdot B=\nabla p & \text{in }\Omega\\
			\nabla\cdot B=0 & \text{in }\Omega\\
			B\cdot n=f+B_0\cdot n & \text{on }\partial\Omega\\
			B_\tau =g & \text{on }\partial\Omega_-
		\end{array} \right.,$$
		where $B_0=\frac{x}{|x|^2}$, and $n$ is the outer normal vector to $\Omega$. Moreover, this is the unique solution satisfying 
		
		$$\|B-B_0\|_{C^{2,\alpha}}\leq M.$$
	\end{teor} 
	\begin{proof}
		
		We just need to prove that the operator $A$ is injective. Then, Theorem \ref{teoremacasos} applies, and the result follows.
		In this case, checking injectivity of the operator $A$ can be done more easily by means of the use of Fourier series. Since they are a base for $L^2(\S^1)$, they are an useful tool to prove injectivity of the operator $A$. As we know, we have to begin with a current $j_0$ defined on $\partial\Omega_-=\{(x,y)\in\R^2\,:\, x^2+y^2=1\}$. By means of the usual transformation formulas for the gradient one readily finds that the transport equation reads, in polar coordinates,
		
		$$\frac{1}{\rho}\frac{\partial j_0}{\partial\rho}=0.$$
		
		As a result, the vector field $j_0$ is constant along the radial lines. Therefore, employing the formula of the Laplacian operator in polar coordinates we find that $u$ satisfies the equation 
		
		$$\left\{\begin{array}{ll}
			\frac{1}{\rho}\frac{\partial}{\partial\rho}\left(\rho\frac{\partial u}{\partial \rho}\right)+\frac{1}{\rho^2}\frac{\partial^2u}{\partial \varphi^2}=j_0\\
			u(r=0)=u(r=a)=0
		\end{array}.
		\right.$$
		
		Due to the symmetry of the current $j_0$, we can look for solutions of the form $\rho^\alpha e^{ik\varphi}$, leading to a solution of the form 
		
		$$j=\sum_{k\in \Z}A_k(\rho)e^{ik\rho},$$
		with $A_k(\rho)$ given by

		\begin{empheq}[left=\empheqlbrace]{alignat*=3}
			-\frac{1}{4}\widehat{j}_0(k)+\left(\frac{1}{4}\widehat{j}_0(k)-\frac{a^2}{4}\widehat{j}_0(k)\right)\frac{\ln\rho}{\ln a}+\frac{\rho^2}{4}\widehat{j}_0(k)& & k=0\\
			\frac{\widehat{j}_0(k)}{4(a^{-2}-a^2)}(\rho a)^2\ln a-\frac{\widehat{j}_0(k)}{4(a^{-2}-a^2)}\left(\frac{a}{\rho}\right)^2\ln a+\frac{\widehat{j}_0(k)}{4}\rho^2\ln\rho & & k=2\\
			\frac{a^{-k}-a^2}{a^{k}-a^{-k}}\frac{\widehat{j}(k)}{4-k^2}\rho^k+\frac{a^2-a^{k}}{a^{k}-a^{-k}}\frac{\widehat{j}(k)}{4-k^2}\rho^{-k}+\frac{\rho^2}{4-k^2}\widehat{j}_0(k) & & \phantom{asdf}k\neq 0,2.
		\end{empheq}
		
		To see whether $A$ is invertible, we have to take the $\rho$ derivative at $\rho=1$, leading to 
		
		\begin{empheq}[left=\empheqlbrace]{alignat*=3}
			\frac{\widehat{j}_0(k)}{4\ln(a)}\left(1-a^2+2\ln a \right) & &k=0\\
			\frac{\widehat{j}_0(k)}{4(a^{-2}-a^2)}\left(4a^2\ln a+(a^{-2}-a^2)\right) & &k=2\\ 
			\frac{\widehat{j}_0(k)}{4-k^2}\frac{1}{a^k-a^{-k}}\left((k-2)a^{-k}+(k+2)a^k-2ka^2\right) & &\phantom{asdf}k\neq 0,2
		\end{empheq}
		
		The term $\left.\frac{\partial u}{\partial\rho}\right|_{\rho=1}$ vanishes if and only if all of its Fourier coefficients are zero. However, due to the form of the multiplier functions that appear above, the only way for this to happen is by taking $\widehat{j}_0(k)=0$ for every $k\in\Z$. Therefore, the kernel of the operator is trivial. As a result, due to Theorem \ref{Fredholm2D}, since $A$ has index $0$, the operator is invertible.
		
		The other ingredient in the proof consists on checking whether the integral \eqref{condicion} is nonvanishing. To do so, we can employ the expression for the inverse in terms of the Fourier series. We can apply it to the function $n\cdot \nabla \chi$, where $\chi$ is the solution of 
		
		$$\left\{\begin{array}{ll}
			\Delta \chi=0\\
			\chi|_{\partial\Omega_-}=0\\
			\chi|_{\partial\Omega_+}=1
		\end{array}\right..$$
		
		Due to the properties of our domain, $n\cdot\nabla \chi=\frac{1}{\ln a}$ at $\rho=1$. Therefore, it is a constant, and due to the structure of the operator we computed above, $A^{-1}(n\cdot \nabla\chi)$ is another constant. Therefore, the equation of the magnetic field $B=\frac{1}{\rho}e_\rho$ is solvable, and thus the result is proven. 
	\end{proof}

	The result we have just presented generalizes that of \cite{Alo-Velaz-Sanchez-2023,Alo-Velaz-2022} to the case of the space between two circles. In this case, the radial vector field $B_0=e_\rho$ plays the role of $(0,1)$ in \cite{Alo-Velaz-2022}. 
	
	One of the main improvements of the present article with respect to \cite{Alo-Velaz-Sanchez-2023,Alo-Velaz-2021} is the fact that the domains dealt with here are physically more realistic. On the other hand, the refinement of the method in \cite{Alo-Velaz-Sanchez-2023,Alo-Velaz-2022} using the operator $A$ allows one to obtain similar results for domains close (in a suitable sense) to the space between two annulus. 
	
	\begin{teor}\label{dominiosgenerales2D}
		Consider $\Omega\subset \R^2$ and $B_0:\Omega\rightarrow \R^2$ as in Theorem \ref{mainteor2D}. Assume that there is an open set $U$ and an orientation preserving diffeomorphism $\gamma:\Omega\longrightarrow U$.
		
		Then, there is a (small)constant $M$ so that for every $f\in C^{2,\alpha}(\partial U)$, $g\in C^{2,\alpha}(\partial U_-)$  satisfying 
		
		$$\int_{\partial U} f=0\qquad\text{and }\qquad \|f\|_{C^{2,\alpha}}+\|g\|_{C^{2,\alpha}}\leq M,$$
		and if $g_{ij}$ satisfies 
		
		\begin{equation}\label{cercaisometria} \|\partial_{j}\gamma_i-\delta_{ij}\|_{C^{3,\alpha}},\,\|\partial_j(\gamma^{-1})_i-\delta^{ij}\|_{C^{3,\alpha}}\leq M,
		\end{equation} 
		there is a solution $(B,p)\in C^{2,\alpha}(U)\times C^{2,\alpha}(U)$ of the boundary value problem
		
		$$\left\{\begin{array}{ll}
			j\times B=\nabla p & \text{in }U\\
			\nabla\times B=j & \text{in }U\\
			\nabla\cdot B=0 & \text{in }U\\
			B\cdot n=f+n\cdot \tilde{B}_0& \text{on }\partial U\\
			B_\tau=g+\tilde{B}_{0,\tau} & \text{on }\partial U_-
		\end{array}\right.,$$
		with $\tilde{B}_0$ defined as $\nabla u$, with $u$ defined as 
		
		$$\left\{\begin{array}{ll}
			\Delta u=0 & \text{in }U\\
			u=0 &\text{on }\partial U_-=\gamma(\partial\Omega_-)\\
			u=\ln(L) & \text{on }\partial U_+=\gamma(\partial \Omega_+)
		\end{array}\right.$$
		
		Furthermore, it is the only solution that satisfies 
		
		$$\|B-\tilde{B}_0\|_{C^{2,\alpha}}\leq M,$$
		
	\end{teor}
	
	\begin{proof}
		The main idea consists on relating the construction of the operator $A[\tilde{B}_0]$ of the domain $U$ (that we will denote explicitly by $A^U[\tilde{B}_0]$) with that of $\Omega$, $A^\Omega[B_0]$. Note that, in order to define the operators $A$, we do not need $B$ to be a solution of the MHS system. Indeed, we only need it to be a never vanishing vector field with the correct sign of $B\cdot n$ on $\partial \Omega_+$ and $\partial\Omega_-$, so that the characteristic lines fill the domain $U$. First of all, notice that $\tilde{B}_0$ is a smooth magnetic field solution of the MHS system with zero current, i.e. a potential field.
		
		Now, as an intermediate step, we prove that $A^U[\gamma_\star B_0]$ and $A^\Omega[B_0]$ are close to each other, in a suitable sense. We begin by relating the flow map of both $\tilde{B}_0$ and $\gamma_\star B_0$. Note that, if we define
		
		$$\tilde{\Phi}(\varpi,s)=\gamma \circ {\Phi}(\gamma^{-1}(\varpi),s)\quad \text{for }\varpi\in\gamma({\partial\Omega_-})=\partial U_-,$$
		we find
		
		\begin{equation} 
			\begin{split}
				\frac{\partial \tilde{\Phi}_{\ell}}{\partial s}&=\left.\frac{\partial \gamma_\ell}{\partial x_j}\right|_{{\Phi}(\gamma^{-1}(\varpi),s)}\left.\frac{\partial \Phi_j}{\partial s}\right|_{(\gamma^{-1}(\varpi),s)}\\
				&=\left.\frac{\partial \gamma_\ell}{\partial x_j}\right|_{{\Phi}(\gamma^{-1}(\varpi),s)} B_{0,j}(\Phi(\gamma^{-1}(\varpi),s))\\
				&=\left.\frac{\partial \gamma_\ell}{\partial x_j}\right|_{\gamma^{-1}(\tilde{\Phi}(\varpi,s))} B_{0,j}(\gamma^{-1}(\tilde{\Phi}(\varpi,s)))\\
				&=(d_{\gamma^{-1}(\tilde{\Phi}(\varpi,s))}\gamma)(B_0(\gamma^{-1}(\tilde{\Phi}(\varpi,s))))\\
				&=\gamma_\star B_0(\tilde{\Phi}(\varpi,s)).
			\end{split}
		\end{equation}
		
		As $\tilde{\Phi}=\varpi$ for $\varpi\in\partial U_-$, $s=0$, we conclude that $\tilde{\Phi}$ equals the flow map of $\gamma_\star B_0$. As a result, we find that the solution of the system 
		
		\begin{equation*}
			\left\{\begin{array}{ll}
				\left(\left(\gamma_\star B_0\right) \cdot \nabla\right) j=0 & \text{in } U\\
				j=j_0 & \text{on }\partial U_-
			\end{array}\right.,
		\end{equation*}
		is given by
		
		$$j=j_0(\varpi(y))=j_0(\gamma(\omega(\gamma^{-1}(y)))).$$ 
		
		Recall the notation $\tilde{\Phi}^{-1}(y)=(\varpi(y),s(y))$ and $\Phi^{-1}(x)=(\omega(x),s(x))$, with $\varpi\in \partial U_-$ and $\omega\in\partial\Omega_-$. Now, the next step in defining $A^U$ consists on solving the PDE system 
		
		$$\left\{\begin{array}{ll}
			\Delta u=j & \text{in }U\\
			u|_{\partial U}=0 & \text{on }\partial U
		\end{array}\right..$$
		
		Now, this is written in terms of the coordinates $(y_1,y_2)\in U$. We can rewrite the equation in terms of the variables $(x_1,x_2)\in \Omega$, leading to 
		
		$$\left\{\begin{array}{ll}
			L v=\tilde{\jmath} & \text{in }\Omega\\
			v|_{\partial U}=0 & \text{on }\partial \Omega
		\end{array}\right.,$$
		where $\tilde{\jmath}$ is the solution of 
		
		\begin{equation*}
			\left\{\begin{array}{ll}
				B_0 \cdot \nabla \tilde{\jmath}=0 & \text{in } \Omega\\
				\tilde{\jmath}=j_0\circ \gamma & \text{on }\partial \Omega_-
			\end{array}\right.,
		\end{equation*}
		$v=u\circ \gamma$ and the differential operator $L$ equals 
		
		$$Lv=\frac{\partial}{\partial x^k}\left(g^{jk}\frac{\partial v}{\partial x^j}\right),$$
		with the matrix $g_{ij}(x)$  defined via $\gamma$ and the canonical basis $\{e_1,e_2\}$ of $\R^2$ via 
		
		$$g_{ij}(x)=\langle d_x\gamma(e_i),\,d_x\gamma(e_j)\rangle .$$
		
		Now, notice that, since $u=v\circ \gamma^{-1}$, we can write 
		
		$$A^U j_0(y)=n^U_y\cdot \nabla u(y)=d_yu(n^U_y)=d_{\gamma^{-1}(y)}v(d_y\gamma^{-1}(n^U_y)).$$
		
		Our claim now is that, for $\varepsilon>0$ small enough, the operator $A^U$ and the operator $\textsf{A}$ defined as 
		
		$$\mathsf{A}j_0:=A^\Omega(j_0\circ \gamma)\circ \gamma^{-1}$$
		are close enough. Note that, as $\gamma$ is a diffeomorphism, invertibility of $A^\Omega$ implies that $\mathsf{A}$ is invertible, which leads to invertibility of $A^U[\gamma_\star{B}]$ by means of a Neumann series argument. We first check that in the case of $\gamma$ an isometry, the considerations above imply immediately that $A^U[\gamma_\star B]$ equals $\textsf{A}$. Thus, the claim follows immediately in such case. The goal is then proving that, if $\gamma$ is close enough to an isometry, in the sense of the statement of the theorem, then $\textsf{A}$ stays near to $A^U[\gamma_\star B_0]$. We notice that the main difference lies in the fact that the operator $L$ is not the Laplacian if $g^{ij}\neq \delta^{ij}$. As a result, it is natural to relate $v$ with the solution of 
		
		$$\left\{\begin{array}{ll}
			\Delta w=\tilde{\jmath} & \text{in }\Omega\\
			w=0 & \text{on }\partial\Omega
		\end{array}
		\right. $$
		
		Then, the function $v-w$ satisfies the boundary value problem 
		
		$$\left\{\begin{array}{ll}
			\Delta (v-w)=\left(\delta^{jk}-g^{jk}\right)\frac{\partial^2v}{\partial x_j\partial x_k}-\frac{\partial g^{jk}}{\partial x_k}\frac{\partial v}{\partial x_j} & \text{in }\Omega\\
			w=0 & \text{on }\partial\Omega
		\end{array}
		\right. $$
		
		As a result, due to the classical Schauder estimates (c.f. \cite{Gilbarg-Trudinger-2001}), there is a $C>0$,  depending only on $\Omega$ such that 
		
		\begin{equation} \label{estimacionvw}
			\|v-w\|_{C^{3,\alpha}}\leq C\|\delta^{ij}-g^{ij}\|_{C^{2,\alpha}}\|v\|_{C^{3,\alpha}}.
		\end{equation}
		
		We can now estimate the difference between $\textsf{A}$ and $A^U[\gamma_\star B_0]$ by noticing that 
		
		\begin{equation} 
			\begin{split}
				(A^U[\gamma_\star B_0]-\textsf{A})j_0&=n^U_y\cdot \nabla u(y)-n^{\Omega}_{\gamma^{-1}(y)}\cdot \nabla w (\gamma^{-1}(y))\\
				&=(d_{\gamma^{-1}(y)}v)(d_y\gamma^{-1}(n^U_y))-(d_{\gamma^{-1}(y)}w)(n^{\Omega}_{\gamma^{-1}(y)})\\
				&=(d_{\gamma^{-1}(y)}v)(d_y\gamma^{-1}(n^U_y))-(d_{\gamma^{-1}(y)}w)(n^{\Omega}_{\gamma^{-1}(y)})\\
				&=(d_{\gamma^{-1}(y)}v)(d_y\gamma^{-1}(n^U_y)-n_{\gamma^{-1}(y)}^\Omega )\\
				&\phantom{ramaladingdon}+(d_{\gamma^{-1}(y)}(v-w))(n^{\Omega}_{\gamma^{-1}(y)})=I_1+I_2
			\end{split}
		\end{equation}
		
		Due to estimate \eqref{estimacionvw}, we can bound $I_2$ by 
		
		\begin{equation*} 
			\begin{split}
				\|I_2\|_{C^{2,\alpha}}&\leq C\|v-w\|_{C^{3,\alpha}}\|\gamma\|_{C^{3,\alpha}}\\
				&\leq C\varepsilon \|v\|_{C^{3,\alpha}}\|\gamma\|_{C^{3,\alpha}}\\
				&\leq C\varepsilon\|\gamma\|_{C^{3,\alpha}}\|\gamma^{-1}\|_{C^{3,\alpha}}\|u\|_{C^{2,\alpha}}\\
				&\leq C\varepsilon\|\gamma\|_{C^{3,\alpha}}\|\gamma^{-1}\|_{C^{3,\alpha}}\|j_0\|_{C^{2,\alpha}}.
			\end{split}
		\end{equation*}
		
		On the other hand, as $\gamma$ is orientation preserving and $\|g^{ij}-\delta^{ij}\|_{C^{3,\alpha}}\leq \varepsilon$, we conclude that 
		
		$$\|I_{1}\|_{C^{1,\alpha}}\leq C \varepsilon \|v\|_{C^{3,\alpha}}\leq C\varepsilon\|\gamma\|_{C^{3,\alpha}}\|j_0\|_{C^{1,\alpha}}.$$
		
		All in all, we have proved that 
		
		$$\|A^U[\gamma_\star B_0]-\textsf{A}\|_{\mathcal{L}(C^{1,\alpha},C^{2,\alpha})}\leq C \varepsilon.$$
		
		Note that, due to the assumptions on the diffeomorphism, the constant $C$ does not depend on $\gamma$ nor $U$, only on $\Omega$. As a result, we conclude that, for $\varepsilon>0$ small enough, the operator $A^U$ is bijective. 
		
		Finally, to prove that $A^U[\tilde{B}_0]$ is invertible, we just need to prove that $\tilde{B}_0$ is close to $\gamma_\star B_0$ with respect to the $C^{2,\alpha}$ norm. This follows after repeating the arguments before, as $\tilde{B}_0$ and $\gamma_\star{B}$ both satisfy two elliptic equations with coefficients that are $C^{2,\alpha}$ close due to the condition \eqref{cercaisometria} on the diffeomorphism $\gamma$. Due to the choice of the boundary condition, one can prove that $\|\tilde{B}_{0}-\gamma_\star B_0\|_{C^{2,\alpha}}\leq C\varepsilon.$
		
		In order to finish the proof we need to check that the integral in \eqref{condicion} does not vanish. This is again a straightforward consequence of the fact that if the functions $\varphi^\Omega$ and $\varphi^U$ satisfy 
		
		$$\left\{\begin{array}{ll}
			\Delta \phi^\Omega & \text{in }\Omega\\
			\phi^{\Omega}=0 & \text{on }\partial\Omega_+\\
			\phi^\Omega=1 & \text{on }\partial\Omega_-
		\end{array} \right. \qquad \text{and}\qquad \left\{\begin{array}{ll}
			\Delta \phi^U & \text{in }\Omega\\
			\phi^{U}=0 & \text{on }\partial\Omega_+\\
			\phi^U=1 & \text{on }\partial\Omega_-
		\end{array}\right.,$$
		
		then $$\|\phi^U-\phi^\Omega\circ \gamma^{-1}\|_{C^{2,\alpha}}\leq C\|g^{ij}-\delta^{ij}\|_{C^{0,\alpha}}.$$
		
		Now, one readily sees that $\textsf{A}^{-1}(n\cdot \nabla\varphi^U)$ is given by 
		
		$$\textsf{A}^{-1}(n\cdot \nabla\varphi^U)=A^{\Omega}[B_0](n\cdot \nabla\varphi^U\circ \gamma)\circ \gamma^{-1}.$$
		
		Therefore, 
		
		\begin{equation}\label{integralchunga}
			\int_{\partial U_-}(n\cdot \tilde{B}_0)\cdot \textsf{A}^{-1}(n\cdot \nabla\varphi^U)d\ell=\int_{\partial \Omega_-} \left[(n\cdot \tilde{B}_0)\circ \gamma\right] A^{\Omega}[B_0]^{-1}(n\cdot \nabla\varphi^U\circ \gamma)J_\gamma d\ell,
		\end{equation}
		where $J_\gamma$ is the jacobian of the diffeomorphism $\gamma$. On the other hand, the integral 
		
		$$\int_{\partial\Omega_-}A^{\Omega}[B_0]^{-1}(n\cdot \phi^\Omega)d\ell$$
		does not vanish. Furthermore, we also know that 
		
		$$\|\phi^U\circ \gamma-\phi^\Omega\|_{C^{2,\alpha}},\, \|\tilde{B}_0\circ \gamma-B_0\|_{C^{2,\alpha}},\, \|J_\gamma-1\|_{C^{1,\alpha}}\leq C\varepsilon.$$
		
		As a result, taking $\varepsilon$ small enough, we conclude that the integral \eqref{integralchunga} does not vanish. Finally, using again that $\|\textsf{A}-A^U[\tilde{B}_0]\|_{\mathcal{L}(C^{1,\alpha},C^{2,\alpha})}$ is close to zero, a simple use of the Neumann series implies that, reducing $\varepsilon$ further, the integral \eqref{condicion} for $A^{U}[\tilde{B}_0]$ is not zero.

	\end{proof}
	\section{The three dimensional case}\label{3D}

	\subsection{The operator $A$}
	\subsubsection{Continuity of the operator $A$ with respect to the field $B$}
	
	We will prove now that the norms of the operators $A$ and $E$ can be estimated in terms of the $C^{1,\alpha}$ norms of the magnetic field $B$. This uses the same techniques as in the 2D case that we studied in \eqref{mainteor2D2D}, but in the three dimensional situation. Since the computations required are entirely similar, we shall give a sketch of the proof, only outlining the places where differences with the 2D case might arise. Note that, actually, due to Proposition \ref{kernel3D}, we do not need to check the cancellation condition \eqref{cancellation}, so the calculations turn out to be simpler than in Section \ref{2D}. We will  study only the operator $A$, as the operator $E$ is a straightforward adaptation of the arguments. The main theorem of this section is, therefore, 
	\begin{teor}\label{teoremadiferencias3D}
		Let $B$ be a $C^{2,\alpha}$ vector field on $\Omega$. Then, the operator $A$ is bounded from $C^{2,\alpha}$ to $C^{1,\alpha}$. Moreover, there exists a constant satisfying 
		
		$$\|A[B]\|_{\mathcal{L}(C^{1,\alpha},C^{2,\alpha})}\leq C\frac{\|B\cdot n\|_{C^{2,\alpha}(\partial\Omega)}}{\min_{\omega\in\partial\Omega_-}|B(\omega)\cdot n(\omega)|^2}\|B\|_{C^{2,\alpha}},$$
		
		Furthermore, if we have $B_1$ and $B_2$ two magnetic fields satisfying the assumptions in Section \ref{asunciones1}, we find that
		
		\small
		\begin{align*}
			\|A[B_1]-A[B_2]&\|_{\mathcal{L}(C^{1,\alpha},C^{2,\alpha})}\\
			&\leq C\frac{\left(\|(B_1-B_2)\cdot n\|_{C^{2,\alpha}}\|B_1\|_{C^{2,\alpha}}+\|B_1\cdot n\|_{C^{2,\alpha}}\|(B_1-B_2)\|_{C^{2,\alpha}}\right)}{\min_{\omega\in\partial\Omega_-}|B_1(\omega)\cdot n(\omega)|^2}\\
			&+C\frac{\|B_1\|_{C^{2,\alpha}}\|B_1\cdot n\|_{C^{2,\alpha}}(\|B_1\|_{C^{2,\alpha}}+\|B_2\|_{C^{2,\alpha}} )(\|B_1-B_2\|_{C^{2,\alpha}})}{\inf_{\omega_{\partial\Omega_-}}|B_1(\omega)\cdot n(\omega)|^2\inf_{\omega_{\partial\Omega_-}}|B_2(\omega)\cdot n(\omega)|}
		\end{align*} 
	\end{teor}
	\begin{proof}
		We will argue by dividing the task in different substeps. Our operator was constructed by taking a function $\psi$ on the surface $\partial\Omega_-$, constructing $\phi$ as a solution of a given PDE on $\partial\Omega_-$, and then solving the div-curl problem where the current $j$ is given by the extension of $j^\|_0=n\times \grad_{\partial\Omega_-}\psi+\grad_{\partial\Omega_-}\phi$ via the transport equation. Therefore, we will proceed by proving suitable regularity estimates for each of the different building blocks of our operator. 
		
		Take $j^\|_0$ be a current on $\partial\Omega_-$. The corresponding leading order of the magnetic field is given by the operator kernel 
		
		\begin{equation}\label{nucleo}
			\frac{1}{4\pi}\int_{\Omega}\left(\frac{x-y}{|x-y|^3}\sqrt{|h(y)|}-\frac{x-y^\star}{|x-y^\star|^3}\sqrt{|h(y^\star)|}\right)\wedge j^{\|}(y)\eta(y)dy,
		\end{equation}
		where $j$ satisfies the transport problem \eqref{transportexp}.
		
		We recall that $y^\star$ equals the reflected point of $y$ along the surface $\partial\Omega$, and $\eta$ equals a smooth cutoff function equal to one in a neighborhood of $\partial\Omega_-$ and whose support lies inside the tubular neighborhood $\mathcal{U}$. Now, since $j$ satisfies the equation \eqref{transportexp}, it must satisfy, for every parametrization of the surface, 
		
		\begin{equation}
			\left\{\begin{array}{l}
				B^\rho\frac{\partial j^\rho}{\partial\rho}+B^1\frac{\partial j^\rho}{\partial x^1}+B^2\frac{\partial j^\rho}{\partial x^2}=j^\rho\frac{\partial B^\rho}{\partial\rho}+j^1\frac{\partial B^\rho}{\partial x^1}+j^2\frac{\partial B^\rho}{\partial x^2}\\
				B^\rho\frac{\partial j^{\|}_1}{\partial\rho}+B^1\frac{\partial j^{\|}_1}{\partial x^1}+B^2\frac{\partial j^{\|}_1}{\partial x^2}=j^\rho\frac{\partial B^{\|}_1}{\partial\rho}+j^1\frac{\partial B^{\|}_1}{\partial x^1}+j^2\frac{\partial B^{\|}_1}{\partial x^2}\\
				B^\rho\frac{\partial j^\|_2}{\partial\rho}+B^1\frac{\partial j^\|_2}{\partial x^1}+B^2\frac{\partial j^\|_2}{\partial x^2}=j^\rho\frac{\partial B^\|_2}{\partial\rho}+j^1\frac{\partial B^\|_2}{\partial x^1}+j^2\frac{\partial B^\|_2}{\partial x^2}.
			\end{array}\right.
		\end{equation}
		
		Therefore, in this coordinates, the vector field $j$ will be given by $(j^\|_1,j^{\|}_2,j^\rho)$ satisfying 
		
		\begin{equation}\label{jota}
			\left(\begin{array}{l}
				j^\rho\\
				j^\|_1\\
				j^\|_2
			\end{array}\right)=F(\Phi^{-1}(y))\left(\begin{array}{c}
				0\\
				j^\|_{0,1}(\Phi^{-1}(y))\\
				j^\|_{0,2}(\Phi^{-1}(y))
			\end{array}\right),
		\end{equation}
		where $\Phi:\partial\Omega_-\times (0,\alpha)\rightarrow \Omega$ is the change of variables given by the equation 
		
		$$\frac{\partial}{\partial\rho}\Phi^{\|}(\omega,\rho)=\frac{B^\|}{B^\rho}\qquad \frac{\partial}{\partial\rho}\Phi^{\rho}(\omega,\rho)=1,$$
		with initial condition $\Phi(\omega,0)=\omega$ for every $\omega\in \partial\Omega_-$, and $F$ is a fundamental matrix of the ODE. Notice that, by definition, $B^\rho$ is never vanishing in $\partial\Omega_-$. Therefore, in a neighborhood of $\partial\Omega_-$, the argument above is valid. Notice, further, that using an adequate partition of unity, we can write the integral in \eqref{nucleo} as 
		
		$$\sum_{i}\frac{1}{4\pi}\int_{\Omega}\left(\frac{x-y}{|x-y|^3}-\frac{x-y^\star}{|x-y^\star|^3}\right)\wedge j^\|(y)\eta(y)\alpha_i(y)dy,$$
		where the sum is finite due to the compactness of $\partial\Omega_-$, and the function $\alpha_i$ is supported inside a coordinate patch.
		
		After taking a differentiable partition of unity, and taking a parametrization in terms of a graph, $\Gamma(x')=(x',\gamma(x'))$, the operator $A$ is given, up to lower order terms, by a sum of operators with kernels of the form 
		\small
		\begin{equation} \label{coordpatch}
			\begin{split}
				\mathsf{K}^{a}_{kl}(\omega,\omega')=&\int_0^\varepsilon \frac{\sqrt{|h(\omega',\rho)|}}{4\pi}\frac{ d(y(\Phi_\rho(\omega')))}{|\omega-y(\Phi_\rho(\omega'))|^3}\frac{\partial}{\partial x_i}\\
				&\cdot\left(\nu(\omega')\wedge\left.\frac{\partial \Gamma}{\partial y_k}\right|_{y(\Phi_\rho(\omega'))}\right)F_{kl}(\omega',\rho)\mathcal{F}\left(x,\Phi_\rho(\omega')\right)\alpha^a(\omega')d\rho,
			\end{split}
		\end{equation}
		\normalsize
		for $\varepsilon$ small enough. This operator acts on the vector $j^\|_{0}(\omega',\rho)$. Therefore, analogously as we did in Section \ref{2D}, the regularity estimates of the operator boil down to study the properties of a kernel operator defined in $\R^2$ by parametrizing the surface. Again, the rest of the integral terms are smoothing, and thus it is easier to derive estimates on then, as we do not have to deal with singularities, and they are regularizing.
		
		Again, the proof of the theorem is based on identifying the kernel above as the kernel of the \textit{adjoint} of a pseudodifferential operator with limited regularity. If we take the parametrization as $x'=(x_1,x_2)\mapsto \omega(x')=(x_1,x_2,\gamma(x_1,x_2))$, with $\nabla\gamma(0,0)=\gamma(0,0)=0$, the aforementioned kernel reads
		
		\begin{equation*}
			f(x',\zeta)=\int_0^\varepsilon \mathfrak{F}(x',\rho)\cdot \mathfrak{G}(x',\zeta,\rho)\cdot \mathfrak{D}(x',\zeta)d\rho,
		\end{equation*}
		where the functions $\mathfrak{F},\,\mathfrak{G},\,\mathfrak{D}$ are defined as 
		
		\begin{equation}\label{frakF}
			\mathfrak{F}(x',\rho)=\frac{\sqrt{|h(\omega(x'),\rho)|}}{4\pi} d(\Phi_\rho(\omega(x')))F_{kl}(\omega',\rho)\mathcal{F}(\Phi_\rho(\omega(x')))\alpha_a(\Phi_\rho(\omega(x'))),
		\end{equation}
		\begin{equation}\label{frakG}
			\mathfrak{G}(x',\zeta,\rho)=\left(\left.\frac{\partial \omega}{\partial x_l}\right|_{x'-\zeta}\right)\left(\nu(\Phi_{\rho}(\omega(x')))\wedge \left.\frac{\partial \Gamma}{\partial y_k}\right|_{\Phi_\rho(\omega(x'))}\right)
		\end{equation}
		\small
		\begin{equation}
			\mathfrak{D}(x',\zeta,\rho)=\frac{1}{\left((\zeta_1+x'_1-y(\Phi_\rho(x'))_1)^2+(\zeta_2+x'_2-y(\Phi_\rho(x'))_2)^2+(g(x',\zeta)+\gamma(x'_1,x'_2)-y(\Phi_\rho(x'))_3)^2\right)^{3/2}},
		\end{equation}
		\normalsize
		where the function $g$ is defined by the difference between $\gamma(x)$ and $\gamma(x-\zeta)$ via the fundamental theorem of calculus, in the same fashion as in Lemma \ref{estimacion}.
		We will derive now regularity estimates for these functions. The case of $\mathfrak{F}$ is the simplest one, as we just need to repeat  the arguments in Theorem \ref{mainteor2D2D}. First, notice that
		
		$$\|d(\Phi_\rho(\omega(\cdot )))\|_{C^{1,\alpha}}\leq C\rho\|B\|_{C^{1,\alpha}}.$$
		
		On the other hand, the other three terms in \eqref{frakF} can be estimated easily as well. Note that $h(\omega(x'),\rho)$ is a smooth function depending only on the domain, as well as $\alpha$. The function $\mathcal{F}$ is the jacobian of the transformation given by the flow map $\Phi$. Therefore, we can find an estimate of the form 
		
		$$\|\mathcal{F}\|_{C^{1,\alpha}}\leq C\|B\|_{C^{2,\alpha}}.$$
		
		Finally, the term $F_{jk}$ is a matrix term that equals the identity at $\rho=0$. Therefore, by means of the mean value theorem, we find that 
		
		$$\|F_{jk}\|_{C^{1,\alpha}}\leq C\left(1+\rho\|B\|_{C^{1,\alpha}}\right).$$
		
		All in all, we find that 
		
		$$\|\mathfrak{F}(\cdot,\rho)\|_{C^{1,\alpha}}\leq C\rho\|B\|_{C^{2,\alpha}}^2.$$	
		
		For $\mathfrak{G}$, the first factor in its definition is obtained easily, as 
		
		$$\left.\frac{\partial \omega}{\partial x_l}\right|_{x'-\zeta}=(e_l,\partial_l\gamma(x'-\zeta)),$$
		where $e_1=(1,0)$ and $e_2=(0,1)$ form the canonical basis in $\R^2$. Then, we can estimate it as 
		
		$$\left\|\partial_\zeta^\alpha \left(\left.\frac{\partial \omega}{\partial x_l}\right|_{(\cdot)-\zeta}\right)\right\|_{C^{1,\alpha}}\leq C.$$
		
		On the other hand, the other two factors in the definition of $\mathfrak{G}$ in \eqref{frakG} can be estimated easily as well just by the employ of the mean value theorem, so that 
		
		$$\|\partial_\zeta^\beta \mathfrak{G}(\cdot,\zeta,\rho)\|_{C^{1,\alpha}}\leq C(1+\rho\|B\|_{C^{1,\alpha}}).$$
		
		We finally want to study the main source of the singularity, i.e. the term $\mathfrak{D}$. Here the estimates are derived in the same way as in Theorem \eqref{mainteor2D2D}, by means of a variation of Lemma \ref{estimacion}, but adapted to the two dimensional setting. Therefore, we find that 
		
		$$\|\mathfrak{D}(x',\zeta,\rho)\|_{C^{1,\alpha}}\leq C\frac{\|B\|_{C^{1,\alpha}}}{\left((\inf_{\omega\in \partial\Omega_-}|B(\omega)|\rho)^2+|B^\rho|^2\zeta^2\right)^{3/2}}$$
		
		As a result, the kernel can be estimated by 
		\small
		$$\|K_{jl}(\cdot,\zeta)\|_{C^{1,\alpha}}\leq C\frac{\|B\|^2_{C^{1,\alpha}}}{\inf_{\omega\in\partial\Omega_-}|B(\omega)|}\int_0^\varepsilon\frac{\rho}{\left(\rho^2+|\zeta|^2\right)^{3/2}}d\rho\leq C\frac{\|B\|^2_{C^{1,\alpha}}}{\inf_{\omega\in\partial\Omega_-}|B(\omega)|}\frac{1}{|\zeta|}\int_0^\infty\frac{u}{(u^2+1)^{3/2}}du.$$
		\normalsize
		
		We then conclude that the kernel satisfy the estimates of Theorem \ref{kernel3D} for $|\beta|=0$. The estimates for the higher derivatives can be obtained similarly. 
	\end{proof}
	\subsubsection{Fredholm and index properties for the case of $B$ smooth}\label{Fredholm3D}
	
	Here, as in Section \ref{calculosimbolo2D}, we provide a computation of the principal symbol of the operator $A$.

	Recall that the operator $A$ consists on the composition of three different operators, according to the following scheme: 
\begin{equation*}	
	\begin{tikzcd}
	\psi && {n\times\text{grad}_{\partial\Omega_-}\,\psi+\text{grad}_{\partial\Omega_-}\,\varphi } && j && W_\tau && {\text{div}^\|_{\partial\Omega_-}W_\tau}
	\arrow["{\textcircled{1}}", from=1-1, to=1-3]
	\arrow["{\textcircled{2}}", from=1-3, to=1-5]
	\arrow["{\textcircled{3}}", from=1-5, to=1-7]
	\arrow["{\textcircled{4}}", from=1-7, to=1-9]
	\end{tikzcd}
\end{equation*}
	
	In step $\textcircled{1}$, the function $\phi$ is obtained by solving the equation \eqref{eq:elipticaphi} with $j^\rho=0$. In \textcircled{2}, $j$ is obtained after solving the transport system \eqref{transport} with $j_0=n\times \text{grad }\psi+\text{grad }\varphi$, and in \textcircled{3}, $W$ is obtained by solving the div-curl system.  In order to prove that the full operator is Fredholm of index zero, we can disregard the contribution of compact operators. Now, the step \textcircled{3}, due to Theorem \ref{teor:biotsavart}, equals to the operator $j_0\mapsto (\mathfrak{J}_2)|_{\tau}$, plus a compact perturbation. On the other hand, step \textcircled{1} equals to the identity plus a compact perturbation, since $\varphi\in C^{3,\alpha}$. Therefore, the operator $A$ equals, up to a compact perturbation, to the following composition of operators
	\begin{equation}\label{diagrama}
	\begin{tikzcd}
		\psi && {n\times\text{grad}_{\partial\Omega_-}\,\psi} && j && \left(\mathfrak{J}_2\right)_\tau && {\text{div}^\|_{\partial\Omega_-}} \left(\mathfrak{J}_2\right)_\tau
		\arrow["{\textcircled{1}}'", from=1-1, to=1-3]
		\arrow["{\textcircled{2}}'", from=1-3, to=1-5]
		\arrow["{\textcircled{3}}'", from=1-5, to=1-7]
		\arrow["{\textcircled{4}}'", from=1-7, to=1-9]
	\end{tikzcd}
	\end{equation}
	We shall begin by describing the symbol of the operator given by the composition $\textcircled{3}'\circ\textcircled{2}'$.
	\begin{prop}
		Let $\Omega\subset \R^3$ and $B:\Omega\longrightarrow \R^3$ satisfying the assumptions in Section \ref{asunciones1}. Consider a surface current $j^\|_0$ defined on $\partial\Omega_-$. Consider now the operator given by 
		
		$$j_0\mapsto (\mathfrak{J}_2)_{\tau},$$
		where $\mathfrak{J}_2$ is defined in \eqref{ansatz}. Then, this is a pseudodifferential operator of order $-1$ whose principal symbol is given by
		
		$$\frac{1}{4\pi}\frac{n\times (\cdot)}{|\xi|_g-i\xi(B^\|)/B^\rho},
		$$
		where $n$ equals the outer normal to $\partial\Omega_-$.
	\end{prop}
	\begin{proof}
		 Assume that (after rotation and translation) each of coordinate patches in \eqref{coordpatch} is given by a graph, so that 
		
		$$y=\Gamma(y'_1,y'_2)=(y'_1,y'_2,\gamma(y'_1,y'_2))+\rho\nu(y'),$$
		where $\gamma$ is a smooth function satisfying $\gamma(0)=0=\nabla\gamma(0)$. Therefore, the basis of the tangent space is given by 
		
		$$\frac{\partial\Gamma}{\partial y'_1}=(1,0,\partial_1 \gamma)+\rho\partial_1\nu(y')\qquad \,\frac{\partial\Gamma}{\partial y'_2}=(0,1,\partial_2\gamma)+\rho\partial_2\nu(y').$$
		
		Now, we can take the change of variables given by $\Phi$, resulting in \eqref{coordpatch} being given by
		
		\begin{equation}\label{localization}
			\begin{split}
				\sum_{i}\frac{1}{4\pi}\int_{\partial\Omega_-}\int_0^\varepsilon\left(\frac{x-\Phi(\omega,\rho)}{|x-\Phi(\omega,\rho)|^3}\right.&\left.-\frac{x-\Phi(\omega,\rho)^\star}{|x-\Phi(\omega,\rho)^\star|^3}\right)\wedge j^\|(\Phi(\omega,\rho))\eta(\Phi(\omega,\rho))\alpha_i(\Phi(\omega,\rho))J_\rho(\omega)d\rho\,d\omega,
			\end{split}
		\end{equation}
		
		Using formula \eqref{jota}, in a system of coordinates $(y_1,y_2,\rho)$, 
		
		$$j^\|(\Phi(\omega,\rho))=j^\|_1(\Phi(\omega,\rho))\left.\frac{\partial\Gamma}{\partial y_1}\right|_{\Phi(\omega,\rho)}+j^\|_2(\Phi(\omega,\rho))\left.\frac{\partial\Gamma}{\partial y_2}\right|_{\Phi(\omega,\rho)},$$
		where 
		
		$$j^\|_1(\Phi(\omega,\rho))=F_{11}(\omega,\rho)j^\|_{1,0}(\omega)+F_{12}(\omega,\rho)j^\|_{2,0}(\omega),$$
		$$j^\|_2(\Phi(\omega,\rho))=F_{21}(\omega,\rho)j^\|_{1,0}(\omega)+F_{22}(\omega,\rho)j^\|_{2,0}(\omega).$$
		
		We are now interested in studying each of the tangential components of this new integral operator. To that end, take $x\in \partial\Omega_-$ and choose a coordinate patch around it. Without loss of generality we can assume (after translation and rotation, if needed) that the parametrization of a neighborhood of $x$ is given by a graph, as before. Now, if we study the operator around $x$, we have two different situations depending on the image of $\alpha_i$. On the one hand, if the support of $\alpha$ does not intersect the support of the coordinate patch for $x$, then this term is trivially smoothing and, as such, does not interfere with the principal symbol. Therefore, everything reduces to study the operators of the form 
		
		\small
		\begin{equation}
			\frac{\partial\Gamma}{\partial x'_i}\cdot\frac{1}{4\pi}\int_{\Omega}\left(\frac{x-y(\Phi(\omega,\rho))}{|x-y(\Phi(\omega,\rho))|^3}-\frac{x-y(\Phi(\omega,\rho))^\star}{|x-y(\Phi(\omega,\rho))^\star|^3}\right)\wedge j^\|(\Phi(\omega,\rho))\eta(\Phi(\omega,\rho))\alpha_i(\Phi(\omega,\rho))\alpha_i(x)J_\rho(\omega)d\rho\,d\omega,
		\end{equation}
		
		\normalsize
		
		Now, via the parametrization of $\partial\Omega_-$, this operator can be now seen as an operator on the plane, just by noticing that it is given by the kernel 
		
		\small
		\begin{equation}
			\begin{split}
				\sum_{k,j=1}^2\frac{\sqrt{1+|\nabla\gamma|^2}}{4\pi}\int_{0}^a\frac{\partial}{\partial x'_i}\cdot&\left(\frac{x-y(\Phi(\Gamma(y',0),\rho))}{|x-y(\Phi(\Gamma(y',0),\rho)),\rho)|^3}-\frac{x-y(\Phi(\Gamma(y',0),\rho)),\rho)^\star}{|x-y(\Phi(\Gamma(y',0),\rho)),\rho)^\star|^3}\right)\wedge \left. \frac{\partial\Gamma}{\partial y'_j}\right|_{y(\Phi(\Gamma(y',0),\rho)),\rho)}\\
				&\times F_{jk}(\Gamma(y',\rho)) \eta(\Phi(\Gamma(y',0),\rho))\alpha_i(\Phi(\Gamma(y',0),\rho))\alpha_i(x)J_\rho(\Gamma(y',0))d\rho\,dy',\\
			\end{split}
		\end{equation}
		\normalsize
		where this kernel acts on the vector $(j^\|_{0,1}(\Gamma(y',0)),j^\|_{0,2}(\Gamma(y',0))).$ Using the same argument we employed when found the regularity estimates for the method of images, we can write the kernel above, up to lower order terms, as 
		\small
		\begin{equation}
			\begin{split}
				\sum_{k,j=1}^2\frac{\sqrt{1+|\nabla\gamma|^2}}{2\pi}&\int_{0}^\varepsilon\frac{F\left(\Gamma(x'),y(\Phi_\rho(\Gamma(y',0)),\frac{\Gamma(x')-y(\Phi_\rho(\Gamma(y',0))}{|\Gamma(x')-y(\Phi_\rho(\Gamma(y',0))|}\right)}{|\Gamma(x')-y(\Phi_\rho(\Gamma(y',0))|^3}\frac{\partial\Gamma}{\partial x'_i}\cdot\left(\rho \nu(\Gamma(y',0))\wedge \left. \frac{\partial\Gamma}{\partial y'_j}\right|_{y(\Phi_\rho(\Gamma(y',0)))}\right)\\
				&\times F_{jk}(\Gamma(y',\rho)) \eta(y(\Phi_\rho(\Gamma(y',0)))\alpha_i(\Gamma(y',0))\alpha_i(\Gamma(x'))J_\rho(\Gamma(y',0))d\rho\,d\omega.\\
			\end{split}
		\end{equation}
		\normalsize
		
		Notice that, anywhere we have a difference $\Gamma(x')-y(\Phi_\rho(\Gamma(y',0))$, we can write
		\small
		\begin{equation}
			\begin{split}
				&(x'_1-y(\Phi_\rho(\Gamma(y',0))_1,x'_2-y(\Phi_\rho(\Gamma(y',0))_2,\gamma(x')-y(\Phi_\rho(\Gamma(y',0))_3)=\\
				&(x'_1-y'_1+y'_1-y(\Phi_\rho(\Gamma(y',0))_1,x'_2-y'_2+y'_2-y(\Phi_\rho(\Gamma(y',0))_2,\gamma(x')-\gamma(y')+\gamma(y')-y(\Phi_\rho(\Gamma(y',0))_3)
			\end{split}
		\end{equation}
		\normalsize
		The denominator can be written quite explicitly, by noticing that, due to Taylor's theorem, it can be as expressed as 
		
		$$|x-y(\Phi_\rho(\Gamma(y',0))|^2=(|x'-y'|^2+\rho^2) h\left(x',y',\rho, \frac{(x'-y',\rho)}{|(x'-y',\rho)|}\right),$$
		where $h$ is a smooth function in $\R\times\R\times\R\times \S^2$, that satisfies that on $x'=y',\rho=0$, it equals 
		
		$$\omega^T\left(\begin{array}{ccc}
			1+(\partial_1\gamma)^2 & (\partial_1\gamma)(\partial_2\gamma) & \frac{B}{B^\rho}\cdot \frac{\partial}{\partial x_1}\\
			(\partial_1\gamma)(\partial_2\gamma) & 1+(\partial_2\gamma)^2 &\frac{B}{B^\rho}\cdot \frac{\partial}{\partial x_2}\\
			\frac{B}{B^\rho}\cdot \frac{\partial}{\partial x_1} &\frac{B}{B^\rho}\cdot \frac{\partial}{\partial x_2} & \frac{|B|^2}{(B^\rho)^2}
		\end{array}\right)\omega=\omega^T R(x') \omega$$
		
		All in all, we have just proved that the kernel can be written up to lower order terms, as 
		\small
		
		$$K(x',y',x'-y')=\int_0^a \frac{g\left(x',y', \rho,\frac{(x'-y')}{|x'-y'|},\frac{\Gamma(x')-y(\Phi_\rho(\Gamma(y',0))}{|\Gamma(x')-y(\Phi_\rho(\Gamma(y',0))|}\right)}{(|x'-y'|^2+\rho^2)^{3/2}}\frac{\partial}{\partial x'_i}\cdot\left(\rho \nu(\Gamma(y',0))\wedge \left. \frac{\partial}{\partial y'_j}\right|_{\Phi(\Gamma(y',0),\rho)}\right)d\rho.$$
		\normalsize
		Out of here we infer that, by means of a Taylor expansion around $x'=y'$ and $\rho=0$ we can write the kernel as 
		
		$$K(x',y',x'-y')=\sum_{j=1}^N K_{-j}(x',y',x'-y')+R_N$$
		where 
		\small
		\begin{equation}
			\begin{split}
				K_{-j}(x',y',&x'-y')\\
				&=\sum_{|\alpha|=j}\int_0^a\frac{(x'-y')^{\alpha_1}\rho^{\alpha_2}}{((x'-y')^2+\rho^2)^{3/2-j}}\frac{1}{\alpha!}\partial^{\alpha_1}_{y'}\partial^{\alpha_2}_\rho g\left(x',x',0,\frac{(x'-y')}{|x'-y'|},\frac{\Gamma(x')-\Phi(\Gamma(y'),\rho)}{|\Gamma(x')-\Phi(\Gamma(y'),\rho)|}\right)\\
				&\qquad\times \frac{\partial\Gamma}{\partial x'_i}\cdot\left(\rho \nu(\Gamma(y',0))\right)\wedge \left. \frac{\partial\Gamma}{\partial y'_j}\right|_{\Phi(\Gamma(y',0),\rho)}d\rho
			\end{split}
		\end{equation}
		\normalsize
		We then obtain that $K$ is the kernel of a pseudodifferential operator whose higher order is given by $K_0$. This procedure actually gives us an asymptotic expansion of the operator. We can now compute the principal symbol by studying the leading order term. This is given by 
		\small
		$$K_0(x',y',z)=\frac{\delta_{jk}\alpha_i(x')\alpha_i(x')}{4\pi}\int_0^a \rho\frac{\sqrt{1+|\nabla\gamma|^2}}{(|x'-y'|^2+\rho^2)^{3/2}(\omega R(x')\omega)^{3/2}}\frac{\partial\Gamma}{\partial x'_i}\cdot\left( \nu(\Gamma(y',0))\right)\wedge \left. \frac{\partial\Gamma}{\partial y'_j}\right|_{\Phi(\Gamma(y',0),\rho)}$$
		\normalsize
	 
	 	Since
		
		$$B\cdot \frac{\partial}{\partial x_1}=B^\|_1(1+(\partial_1\gamma)^2)+B^{\|}_2 (\partial_1\gamma)(\partial_2\gamma),$$
		$$B\cdot \frac{\partial}{\partial x_2}=B^\|_1(\partial_1\gamma)(\partial_2\gamma)+B^{\|}_2 (1+(\partial_1\gamma)^2),$$
		we can write the denominator as 
		\small
		\begin{equation*}
			\begin{split}
				&(1+(\partial_1\gamma)^2)\left(z_1+\frac{B_1^\|}{B^\rho}\rho\right)^2+(1+(\partial_2\gamma)^2)\left(z_2+\frac{B_2^\|}{B^\rho}\rho\right)^2+2(\partial_1\gamma)(\partial_2\gamma)\left(z_1+\frac{B^\|_1}{B^\rho}\rho\right)\left(z_2+\frac{B^\|_2}{B^\rho}\rho\right)+\rho^2,
			\end{split}
		\end{equation*}
		\normalsize
		where $z=x'-y'$. We can now perform the Fourier transform with respect to the $z$ variable. Note that the quadratic form above is nothing else than the contraction of the vector $z+B^\|\rho/B^\rho$. Therefore, if $g_{ij}$ denotes the matrix of the metric, the relevant part when performing the Fourier Transform equals then 
		\small
		\begin{equation*}
			\int_0^a e^{-i\xi_1\rho B^\|_1/B^\rho-i\xi_2\rho B^\|_2/B^\rho}\int_{\R^2}\frac{e^{-iz\cdot\xi}}{((z_i g_{ij}z_j)^2+\rho^2)^{3/2}}=\int_0^a \frac{e^{-i\xi_1\rho B^\|_1/B^\rho-i\xi_2\rho B^\|_2/B^\rho}}{\sqrt{1+|\nabla\gamma|^2}}\int_{\R^2}\frac{e^{-iz\cdot g^{-1/2}\xi}}{(|z|^2+\rho^2)^{3/2}}.
		\end{equation*}
		\normalsize
		The Fourier transform inside can be computed easily. Notice that the function $(|z|^2+\rho^2)^{-3/2}$ is a radial function and, as such, its Fourier Transform will be a radial function as well. Now, if we take $\xi=(1,0)$, and since
		
		$$\int_{-\infty}^\infty \frac{dz_2}{(z_1^2+z_2^2+\rho^2)^{3/2}}=\frac{2}{z_1^2+\rho^2},$$
		we can conclude that 
		
		$$\int_{\R^2}\frac{e^{-iz_1\xi_1}}{(|z|^2+\rho^2)^{3/2}}dz_1dz_2=\int_{-\infty}^\infty 2\frac{e^{-i\xi_1 z}}{z_1^2+\rho^2}dz_1=\frac{1}{\rho}e^{-\rho |\xi_1|}.$$
		
		Since the Fourier transform of a radial function is again radial, we obtain 
		
		$$\int_{\R^2}\frac{e^{-iz_1\xi_1}}{(|z|^2+\rho^2)^{3/2}}dz_1dz_2=\frac{1}{\rho}e^{-\rho|\xi|}.$$
		Therefore, the matrix of the principal symbol, with respect to the basis of $TM$ given locally by $\{E_1,E_2\}$, is 

		$$\sigma(x,\zeta)=\frac{1}{4\pi}\frac{(1+|\nabla\gamma|^2)^{-1/2}}{|\xi|_g-i\xi(B^\|)/B^\rho}\left(
		\begin{array}{ll}
			-(\partial_1\gamma)(\partial_2\gamma)  & -(1+(\partial_2\gamma)^2)\\
			1+(\partial_1\gamma)^2 & (\partial_1\gamma)(\partial_2\gamma)
		\end{array}\right)$$
	
	Now, we just note that the matrix 
	
	$$(1+|\nabla\gamma|^2)^{-1/2}\left(
	\begin{array}{ll}
		-(\partial_1\gamma)(\partial_2\gamma)  & -(1+(\partial_2\gamma)^2)\\
		1+(\partial_1\gamma)^2 & (\partial_1\gamma)(\partial_2\gamma)
	\end{array}\right)$$
is the matrix of the linear map given by $v\mapsto n\times(v)$ if we express it in the basis of the tangent space given by the parametrization $\Gamma$. Thus, the result is proved.
	\end{proof}
	Now, due to the composition rule for the principal symbol (c.f. Proposition \ref{composimbolos}), we need to compute principal symbols of the operators ${\textcircled{1}}'$ and ${\textcircled{4}}'$ from \eqref{diagrama} to find the principal symbol of $A$.
	
	\begin{lema}
		Let $(M,g)$ be a Riemannian manifold. Then, the operator 
		
		$$ f\mapsto n\times\grad_M{f}$$
		defines a differential operator from $M$ to $TM$. Moreover, is symbol is given by 
		
		$$\sigma_{\grad}(x,\xi)_{\ell}=i\varepsilon_{\ell jk}n_j(x)g^{kl}\xi_l,$$
		where $\varepsilon_{\ell jk}$ is the Levi-Civita symbol.
	\end{lema}
	
	\begin{lema}
		Let $(M,g)$ be a Riemannian manifold. Then, the operator 
		
		$$X\mapsto \text{div}_M X$$
		defines a differential operator from $TM$ to $M$. Moreover, its symbol is given by 
		
		$$\sigma_{\text{div}}=i\xi(\cdot) .$$
	\end{lema}
	
	After these two easy lemmas we can compute the symbol of the operator $A$, which equals $\sigma_{div}\cdot \sigma\cdot\sigma_{grad}$.
	
	\begin{coro}\label{purocorolario}
		Let $\Omega\subset \R^3$ and $B:\Omega\longrightarrow \R^3$ be as in Section \ref{asunciones1}. Then, consider the operator $A$ defined in \ref{defia3D}. Then, up to a compact perturbation, it is a pseudodifferential operator with symbol 
		$$\frac{1}{4\pi}\frac{|\xi|_g^2}{|\xi|_g-i\xi(B^\|)/B^\rho}.$$
	\end{coro}
	\begin{obs}
		Note that this is an expression which is invariant under changes of coordinates, so we see that our symbol correctly describes a function in $T^\star\partial\Omega_-\setminus\{0\}$. We see again that this is an elliptic operator of order 1. Therefore, our operator is Fredholm. Moreover, it has (up to the term $|\xi|_g$, which is does not play a role when applying Theorem \ref{teorindice}) the same symbol as the operator in the 2D situation. Therefore, it is of index zero. 
	\end{obs}
	\begin{proof}[Proof of Corollolary \ref{purocorolario}]
		As indicated before, one just has to perform a product of matrices, due to Proposition \ref{composimbolos}, and to make sure that the rest of terms are of lower order. This is, however, an easy consequence of Proposition \ref{divcurlconstruction}, that tells us that the solution to the div-curl system is given by $\mathfrak{J}_1+\mathfrak{J}_2$ plus compact perturbations. Furthermore, the fact that in $\mathfrak{J}_1$ only $j^\rho$ contributes, and that the matrix $F_{jk}=\delta_{jk}$ makes the contribution of $\mathfrak{J}_1$ of lower order.  
	\end{proof}

	\subsection{Applications to nearly annular domains}\label{asunciones3d}

	Regarding how general our assumptions on the admissible domains $\Omega$ and magnetic fields $B$ are, the situation is less clear as the one in 2D (c.f. Theorem \ref{proposition2D}), as we do not have the tools of complex analysis. However, there is still a class of reasonably general domains where we can find such fields. This is the case $\Omega_1$ and $\Omega_0$ star shaped convex domains around $0$, and their boundaries are smooth and nonintersecting. We define $\Omega:=\Omega_1\setminus \Omega_0$. Then, if we construct $B_0$ again as $\nabla \varphi$, we find that $B_0$ cannot have vanishing points. This argument can be found in \cite{Evans}. It is based on the fact that the function $x\mapsto u(\mu x)$ is harmonic for every $\mu>0$. Therefore, so is 
	
	$$v:=\left.\frac{d}{d\mu}\right|_{\mu=1}u(\mu x)=x\cdot \nabla u(x).$$
	
	Due to the geometry of our domain, $v$ is nonpositive on the boundary of $\Omega$ and, therefore, by the maximum principle, it is never vanishing in the interior. Therefore, we find that $\nabla u(x)$ can never vanish in $\Omega$.

	As in the 2D case, the reasonings above can be applied to solve the problem in the case of the space between two spheres. Due to the reasonings we made before, we just need to prove that the equation is solvable for our initial magnetic field $B_0$. In this case, we take it to be 
	
	$$B_0=\frac{x}{|x|^3},$$
	i.e. the magnetic monopole centered at $x=0$. We have to begin by solving the transport equation. Due to the formula of the material derivative in spherical coordinates $(r,\varphi,\theta)$, the transport equation reads
	
	$$\left\{\begin{array}{ll}
		\frac{\partial j_r}{\partial r}=-2\frac{j_r}{r},\\\\
		\frac{\partial j_\theta}{\partial r}=2\frac{j_\theta}{r},\\\\
		\frac{\partial j_\varphi}{\partial r}=2\frac{j_\varphi}{r},
	\end{array}
	\right.$$
	
	Therefore, we find that 
	
	$$j(r,\theta,\varphi)=\left(\frac{1}{r}\right)^2j_0^r+rj^\|_0.$$
	
	Now, to make the computations explicitly, we introduce the vector spherical harmonics, defined as 
	
	$$Y_{lm}=\frac{x}{|x|^3}\mathcal{Y}_{lm},$$
	$$\Phi_{lm}=x\times \nabla\mathcal{Y}_{lm},$$
	$$\Psi_{lm}=|x|\nabla\mathcal{Y}_{lm}.$$
	
	The properties of these functions can be found in \cite{barrera}. Due to the epression of $j$, in which radial and tangential components have an identifiable behaviour, we can write 
	
	$$j=\sum_{l=0}^\infty \sum_{m=-l}^lj_{lm}^r\frac{1}{r^2}Y_{lm}+j_{lm}^{(1)}r\Psi_{lm}+j^{(2)}_{lm}r\Phi_{lm}.$$
	
	The divergence free condition then reads 
	
	\begin{equation}
		0=\nabla\cdot j=l(l+1)j_{lm}^{(1)}.
	\end{equation}
	
	On the other hand, the div-curl problem for $b=B-B_0$ reads 
	
	\begin{equation}\label{divcurlesfera}
		\left\{\begin{array}{ll}
			-\frac{l(l+1)}{r}b^{(2)}_{lm}=j^r_{lm}\frac{1}{r^2} \\
			-\frac{1}{r}\frac{d}{dr}\left(r b_{lm}^{(2)}\right)=0\\
			-\frac{b^r_{lm}}{r}+\frac{1}{r}\frac{d}{dr}\left(rb^{(1)}_{lm}\right)=j^{(2)}_{lm}r\\
			\frac{1}{r^2}\frac{d}{dr}\left(r^2b^r_{lm}\right)-\frac{l(l+1)}{r}b^{(1)}_{lm}=0
		\end{array}\right.,
	\end{equation}
	where the first three equations correspond to the equation for the curl, and the last one is the divergence free condition. The term $l=0$ leads to the set of conditions 
	
	$$j_{00}^r=0 \qquad b^{r}_{00}\frac{f^{-}_{00}}{r^2}.$$
	
	On the other hand, for the higher $l'$s, we obtain the equations for $j^r$ given by 
	
	$$j^r_{lm}=-\frac{g^{(2)}}{r_0}l(l+1).$$
	
	Note that this corresponds precisely with our observation in \eqref{jrho}. Now, for $j^{(2)}_{lm}$ we can find, after some computations, that $b^r_{lm}$ is written as 
	
	$$r^2\frac{d^2}{dr^2}b^r_{lm}+4r\frac{d}{dr}b^r_{lm}+(2-l(l+1))b^r_{lm}=j^{(2)}_{lm}r^2l(l+1).$$
	
	This, as in the 2D case, is an Euler Cauchy whose solutions can be found explicitly. We find that the solution reads 
	
	\begin{empheq}[left=b_{lm}^r{=}\empheqlbrace]{alignat*=2}	
		&C_1^{lm}r^{l-1}+\frac{C_2^{lm}}{r^{2+l}}+\frac{l(l+1)}{12-l(l+1)}r^2j^{(2)}_{lm}\phantom{asd}&  l\neq 3\\
		&C_1^{3m}r^2+\frac{C^{3m}_2}{r^5}+\frac{12}{7}j^{(2)}_{lm}r^2\ln(r)&  l=3
	\end{empheq}
	
	We finally impose boundary conditions $b^r_{lm}=0$ for $r=1,a$, resulting in 
	
	\begin{empheq}[left=b_{lm}^r{=}\empheqlbrace]{alignat*=2}
		&\frac{l(l+1)}{12-l(l+1)}j^{(2)}_{lm}\left(\frac{L^2-L^{-2-l}}{L^{-2-l}-L^{l-1}}r^{l-1}+\frac{L^{l-1}-L^2}{L^{-2-l}-L^{l-1}}r^{-2-l}+r^2\right)&  \phantom{asd}l\neq 1,3\\
		&\frac{12}{7}j^{(2)}_{lm}\left(\frac{L^2\ln L}{L^{-5}-L^2}r^2-\frac{L^2\ln L}{L^{-5}-L^2}r^{-5}+r^2\ln r\right) & r=3
	\end{empheq}
	
	Finally, out of here we can find the formula for $b^{(1)}_{lm}$, given by \eqref{divcurlesfera}
	
	\begin{empheq}[left=b_{lm}^r{=}\empheqlbrace]{alignat*=2}
		&\frac{j^{(2)}_{lm}}{7}\left(4\frac{L^2\ln L}{L^{-5}-L^2}+3\frac{L^2\ln L}{L^{-5}-L^2}+1\right) &\phantom{asdf} l=3\\
		&\frac{j^{(2)}_{lm}}{12-l(l+1)}\left((l+1)\frac{L^2-L^{-2-l}}{L^{-2-l}-L^{l-1}}-l\frac{L^{l-1}-L^2}{L^{-2-l}-L^{l-1}}+4\right) & l\neq 3
	\end{empheq}
	
	We find that the equation is invertible, as the coefficients multiplying $j_{lm}^{(2)}$ do not vanish for any $L>0$. Since the domain is simply connected, we do not have to worry in this case about any well definiteness of the presssure, so we conclude the 3D analogous of the Theorem \ref{mainteor2D}.
	
	\begin{teor}\label{mainteor3D}
		Let $L>1$ and $\alpha \in (0,1)$. Then, if $\Omega=\{(x,y,z)\in\R^3 \,:\, 1<x^2+y^2+z^2<L^2\}$, there exists a (small) constant $M=M(\alpha,L)>0$ such that, if $\in C^{2,\alpha}(\partial\Omega)$ and $g\in C^{2,\alpha}(\partial\Omega_-)$ satisfy 
		
		$$\int_{\partial\Omega_-}f =  \int_{\partial\Omega_+}f,$$
		and 
		
		$$\|f\|_{C^{2,\alpha}}+\|g\|_{C^{2,\alpha}}\leq M,$$
		then there exists a solution $(B,p)\in C^{2,\alpha}(\Omega)\times C^{2,\alpha}(\Omega)$ of the problem 
		
		$$\left\{\begin{array}{ll}
			(\nabla\times B)\cdot B=\nabla p & \text{in }\Omega\\
			\nabla\cdot B=0 & \text{in }\Omega\\
			B\cdot n=f+1& \text{on }\partial\Omega\\
			B_\tau =g & \text{on }\partial\Omega_-
		\end{array} \right.,$$
		where $B_0=\frac{x}{|x|^3}$, and $n$ is the outer normal vector to $\Omega$. Moreover, this is the unique solution satisfying 
		
		$$\|B-B_0\|_{C^{2,\alpha}}\leq M.$$
	\end{teor} 
	
	Again, the techniques in this paper are robust enough to generalize these results to domains ``near'' a reference domain. Thus, we obtain a generalization of Theorem \ref{dominiosgenerales2D}: 
	
	\begin{teor}\label{dominiosgenerales3D}
		Consider $\Omega\subset \R^3$ and $B_0:\Omega\rightarrow \R^3$ as in Theorem \ref{mainteor3D}. Assume that there is an open set $U$ and an orientation preserving diffeomorphism $\gamma:\Omega\longrightarrow U$. Then, there exists a (small) $\varepsilon>0$ such that, if $\gamma$ satisfies
		\begin{itemize}
			\item $\det\left(\nabla \gamma\right)$ constant.
			\item $\|\partial_j \gamma_i-\delta_{ij}\|_{C^{4,\alpha}},\,\|\partial_j\gamma^{-1}_i-\delta^{ij}\|_{C^{4,\alpha}}\leq \varepsilon.$
			
		\end{itemize}
		
		Then, there is a constant $M$ so that for every $f\in C^{2,\alpha}(\partial U)$, $g\in C^{2,\alpha}(\partial U_-)$  satisfying 
		
		$$\int_{\partial U} f=0\qquad\text{and }\qquad \|f\|_{C^{2,\alpha}}+\|g\|_{C^{2,\alpha}}\leq M,$$
		there is a solution $(B,p)\in C^{2,\alpha}(U)\times C^{2,\alpha}(U)$ of the boundary value problem
		
		\begin{equation}\label{sistemamhs3D} \left\{\begin{array}{ll}
				j\times B=\nabla p & \text{in }U\\
				\nabla\times B=j & \text{in }U\\
				\nabla\cdot B=0 & \text{in }U\\
				B\cdot n=f+n\cdot \tilde{B}_0& \text{on }\partial U\\
				B_\tau=g+(\tilde{B}_0)_\tau & \text{on }\partial U_-
			\end{array}\right.,
		\end{equation}
		where $\tilde{B}_0=\nabla \tilde{u}$, with $\tilde{u}$ defined as 
	
		$$\left\{\begin{array}{ll}
			\Delta \tilde{u}=0 & \text{in }U\\
			\tilde{u}=-1 & \text{on }\partial U_-\\
			\tilde{u}=-1/L & \text{on }\partial U_+
		\end{array}\right..$$
		
		Furthermore, $B$ is the unique solution of the MHS \eqref{sistemamhs3D} satisfying 
		
		$$\|B-\tilde{B}_0\|_{C^{2,\alpha}}\leq \varepsilon,$$
	
	\end{teor}
	\begin{proof}
		The arguments are similar as those in Theorem \ref{dominiosgenerales2D}, with some further computational difficulties arising from the fact that the system is three dimensional. Note that here defining $A[B]$ requires the magnetic field $B$ to be divergence free. Otherwise, our computations do not make sense. That is why we have imposed the extra assumption on the determinant of the jacobian $\nabla \gamma$. Therefore, we must be more careful when relating the operators $A^\Omega$ and $A^U$.  Again, the goal consists on proving that, due to the fact that $\gamma$ is ``almost'' an isometry, the operators $A^U$ and $A^{\Omega}$ look alike, so the injectivity of $A^\Omega$ implies that of $A^U$. Note first that, due to the condition on $\text{det}\,(\nabla\gamma)$, it makes sense to study $A^U[\gamma_\star B_0]$. Furthermore, due to the fact that both are defined in terms of gradients of solutions of elliptic equations, and due to the transformation formula of the differential operators in terms of a change of coordinates, it is easy to prove that 
		
		$$\|\gamma_\star B_0-\tilde{B}_0\|_{C^{2,\alpha}}\leq C\varepsilon.$$
		
		As a result, taking $\varepsilon$ small enough, we obtain that $A^U[\tilde{B}_0]$ is invertible if $A^U[\gamma_\star B_0]$ is invertible as well. Therefore, everything boils down to study the invertibility of the operator $A^U[\gamma_\star B_0]$. This will be done, as in the case of Theorem \ref{dominiosgenerales2D}, by proving that $A^U[\gamma_\star B_0]$ is close to an operator related to $A^{\Omega}[B_0]$, defined as 
		
		$$\mathsf{A}\tilde{\psi}:=A^\Omega[B_0](\tilde{\psi}\circ \gamma)\circ \gamma^{-1}.$$
		
		We denote $\psi:=\tilde{\psi}\circ\gamma$. Recall that, when constructing $A^U$ and $A^\Omega$, the first step consists on defining tangential currents $\tilde{\jmath}_0^\|$ and $j_0^\|$ on $\partial U_-$ and $\partial\Omega_-$, respectively, so that 
		
		$$\tilde{\jmath}_0^\|=\tilde{n}\times \text{grad}_{\partial U_-} \tilde{\psi}+\text{grad}_{\partial U_-}\tilde{\varphi},$$
		and 
		$${j}_0^\|={n}\times \text{grad}_{\partial \Omega_-} {\psi}+\text{grad}_{\partial \Omega_-}{\varphi},$$
		where $\varphi$ and $\tilde{\varphi}$ satisfy
		
		\begin{equation} \label{elipticasuperficie}
			(\gamma_\star B_0)^\rho \Delta_{\partial U_-}\tilde{\varphi}-\mathcal{L}_{\text{grad}_{\partial U_-}\tilde{\varphi}} (\gamma_\star B_0)^\rho=\mathcal{L}_{n\times\grad_{\partial U_-}\tilde{\psi}}(\gamma_\star B_0)^\rho.
		\end{equation}
		and 
		
		\begin{equation} \label{elipticasuperficie2}
			B_0^\rho \Delta_{ \partial\Omega_-}{\varphi}-\mathcal{L}_{\text{grad}_{\partial \Omega_-}{\varphi}} B_0^\rho=\mathcal{L}_{n\times\grad_{\partial \Omega_-}{\psi}}B_0^\rho.
		\end{equation}
		
		We can see the tangent vector fields on $\partial \Omega$ and $\partial U$ as functions with values in $\R^3$, just by using the canonical embedding $T_x\partial \Omega \,(\text{or }T_x\partial U)\hookrightarrow \R^3$. In that case, it is easy to see that, if the components of $\text{grad}_{\partial\Omega_-}\,\psi$ and $\text{grad}_{\partial U_-}\,\tilde{\psi}$ as vectors in $\R^3$ are given by $(\text{grad}_{\partial\Omega_-}\,\psi)_i$ and $(\text{grad}_{\partial U_-}\,\tilde{\psi})_i$, respectively, then 
		
		\begin{equation}\label{ec:formulagradientecambiovar}
			(\text{grad}_{\partial U_-}\,\tilde{\psi})_i(\gamma(x))=\frac{\partial\gamma_j}{\partial x_i}(x)(\text{grad}_{\partial\Omega_-}\,\psi)_j(x).
		\end{equation}
		
		As a result, due to the assumptions on $\gamma$ it is easy to see that 
		
		$$\|\tilde{n}\times \text{grad}_{\partial U_-}\,\tilde{\psi}-\gamma_\star (n\times\text{grad}_{\partial \Omega_-}\,{\psi})\|_{C^{1,\alpha}}\leq C\varepsilon\|\psi\|_{C^{2,\alpha}}.$$
		
		Due to formula \eqref{ec:formulagradientecambiovar} and the conditions on $\gamma$, we can also find that 
		
		$$\|\gamma_\star(\text{grad}_{\partial\Omega_-}\varphi)-\text{grad}_{\partial U_-}(\varphi\circ \gamma^{-1})\|_{C^{1,\alpha}}\leq C\varepsilon\|\psi\|_{C^{2,\alpha}}.$$
		
		As a result, due to the Helmholtz decomposition, we find that 
		
		$$\|\tilde{\jmath}^\|_0-\gamma_\star j^\|_0\|_{C^{1,\alpha}}\leq C\varepsilon\|\psi\|_{C^{2,\alpha}}.$$
		
		The next step when defining both $A^U[\gamma_\star B_0]$ and $A^\Omega[B_0]$ consists on constructing bulk currents $\tilde{\jmath}$ and $j$ given by 
		
		\begin{equation}\label{firstorderpdetransformacion}
			\left\{\begin{array}{ll}
				(B_0\cdot \nabla ){j}=({j}\cdot \nabla) B_0& \text{in }\Omega\\
				j={j}_0^\| & \text{on }\partial \Omega_-
			\end{array}\right. .
		\end{equation}
		and
		\begin{equation}\label{firstorderpdetransformacion2}
			\left\{\begin{array}{ll}
				(\gamma_\star B_0\cdot \nabla )\tilde{\jmath}=(\tilde{\jmath}\cdot \nabla)\gamma_\star B_0& \text{in }U\\
				j=\tilde{\jmath}_0^\| & \text{on }\partial U_-
			\end{array}\right. .
		\end{equation}

		One could try by means of the method of characteristics to relate the solution of \eqref{firstorderpdetransformacion} and its counterpart in $U$. However, a slightly more direct method relies of the fact, already mentioned in Section \ref{integralequcurr3d}, that equation \eqref{firstorderpdetransformacion2} is equivalent to write
		
		$$[\gamma_\star{B},\tilde{\jmath}\,]=0.$$
		
		Therefore, it is a \textit{differential-topological} quantity, whose behaviour under changes of coordinates (i.e., diffeomorphisms) is well known. As a result, if $j$ solves \eqref{firstorderpdetransformacion}, then $\gamma_\star j$ solves 
		
			\begin{equation*}
			\left\{\begin{array}{ll}
				(\gamma_\star B_0\cdot \nabla )\gamma_\star j=(\gamma_\star j\cdot \nabla)\gamma_\star B_0& \text{in }U\\
				j=\gamma_\star j_0^\| & \text{on }\partial U_-
			\end{array}\right. .
		\end{equation*}
		
		As a result, due to basic estimates on linear hyperbolic equations, we find that

		$$\|\tilde{\jmath}-\gamma_\star j\|_{C^{1,\alpha}}\leq C\|\tilde{\jmath}^\|_0-\gamma_\star j^\|_0\|_{C^{1,\alpha}}\leq C\varepsilon\|\psi\|_{C^{2,\alpha}}.$$
		
		Thus, we can study the next building block of $A^U$ in terms of its $\Omega$ counterpart. Consider the magnetic field $\tilde{W}$ solving 
		
		\begin{equation*}
			\left\{\begin{array}{ll}
				\nabla\times \tilde{W}=\tilde{\jmath} & \text{in }U\\
				\nabla\cdot \tilde{W}=0 & \text{in }U\\
				\tilde{W}\cdot \tilde{W}=0 & \text{on }\partial U
			\end{array}\right. .
		\end{equation*}
		We can now take the difference between $\tilde{W}$ and $\gamma_\star W$. Due to the condition on the determinant of $\nabla\gamma$, $\gamma_\star W$ is divergence free again, so the equation that $\tilde{W}-\gamma_\star W$ equals 
		\begin{equation*}
			\left\{\begin{array}{ll}
				\nabla\times \left(\tilde{W}-\gamma_\star W\right)=\tilde{\jmath}-\gamma_\star j +\textsf{H}& \text{in }U\\
				\nabla\cdot \tilde{W}=0 & \text{in }U\\
				(\tilde{W}-\gamma_\star W)\cdot \tilde{n}= -(\gamma_\star W)\cdot \tilde{n}& \text{on }\partial U
			\end{array}\right. ,
		\end{equation*}
	with $\textsf{H}$ equals 
	
	$$\textsf{H}_i=-\left(\frac{\partial\gamma_i}{\partial x_j}-\delta_{ij}\right)\varepsilon_{jkl}\partial_k W_l-\varepsilon_{jkl}\partial_k\left(\left(\delta_{\ell s}-\frac{\partial\gamma_\ell}{\partial x_s}\right)W_s\right).$$
	
	Due to the assumptions on $\gamma$, we find that $\|\tilde{n}-\gamma_\star n\|_{C^{2,\alpha}}\leq C\varepsilon.$ Since $W\cdot n=0$, and due to the previous computations we made to relate $\tilde{\jmath}$ and $\gamma_\star j$, as well as classical estimates for the $C^{2,\alpha}$ norm of the solutions of div-curl systems, we obtain that 
	
	$$\|\tilde{W}-\gamma_\star W\|_{C^{2,\alpha}}\leq C\varepsilon \|\psi\|_{C^{2,\alpha}}.$$ 
	
	Finally, we have to relate $\text{div}_{\partial U_-}\tilde{W}$ with $(\text{div}_{\partial\Omega_-}W)^\|\circ \gamma^{-1}$. Due to the estimate above, it is enough to bound $\|\text{div}_{\partial U_-}(\gamma_\star W^\|)-(\text{div}_{\partial\Omega_-}W^\|)\circ\gamma^{-1}\|_{C^{2,\alpha}}.$ This is, again, an easy consequence on the condition on $\gamma$.

		As a result, taking $\varepsilon$ small enough, find out that $A^U[\gamma_\star B_0]$ and $\textsf{A}$ defined as 
		
		$$\textsf{A}\phi:=A^{\Omega}[B_0](j_0\circ\gamma^{-1})\circ \gamma$$
		
	are close in the operator norm, so  we can argue as in Theorem \ref{dominiosgenerales2D} to conclude that the operator $A$ is invertible, so the result is proven. 
	\end{proof}
	
	One might notice that the restrictions on Theorem \ref{dominiosgenerales3D} are somehow more restrictive than its two dimension counterpart. This is due to the condition on the jacobian of $\gamma$, which is put in place in order to make sure that $\gamma_\star B_0$ remains divergence free. At a first glance, this looks like a much worse theorem than \ref{dominiosgenerales2D}, where we allowed for any diffeomorphism, as long as it was close to the identity. This is, however, not a real restriction. Indeed, if we are given a domain $U$ and a diffeomorphism $\gamma:\Omega\longrightarrow U$ so that $\nabla\gamma$ is close to the identity in the sense of Theorem \ref{dominiosgenerales3D}, then with no loss of generality we can dilate our $\Omega$ so that the its volume coincides with that of $U$. Then, we have two volume forms in $U$: $\text{det}\,\nabla\gamma d^3y$ and $d^3y$, both of which have the same integral. Now we just employ the classical result by Moser (c.f. \cite{Moser}) that allows us to find a difeomorphism of $U$ into itself that transforms one volume form into the other. This leads to the fact that, up to dilating $\Omega$, we can assume $\gamma$ to have constant determinant.

	\section*{Declarations}

	 \textbf{Acknowledgements} The authors gratefully acknowledge the fruitful discussions and suggestions with Diego Alonso-Orán.
	
	 Daniel Sánchez-Simón del Pino is funded by the Deutsche Forschungsgemeinschaft (DFG, German Research Foundation) under Germany’s Excellence Strategy - GZ 2047/1, Projekt-ID 390685813 and by the Bonn International Graduate School of Mathematics (BIGS) at the Hausdorff Center for Mathematics. J. J. L. Velázquez gratefully acknowledge the support by the Deutsche Forschungsgemeinschaft (DFG) through the collaborative research centre The mathematics of emerging effects (CRC 1060, Project-ID 211504053) and the DFG under Germany’s Excellence Strategy -EXC2047/1-390685813.

	\vspace{0.5cm}
	
	\textbf{Conflicts of interest.} All authors declare that they do not have conflicts of interest.
	
	\vspace{0.5cm}
	\textbf{Data availability statement.} Data sharing not aplicable to this article as no datasets were generated or
	analysed during the current study.
	
	\bibliographystyle{siam}
	\bibliography{referencias}
	
\end{document}